\documentclass[10pt,a4paper,oneside]{amsart}

\usepackage{amssymb,amsthm}
\usepackage[T1]{fontenc}
\usepackage{lmodern}
\swapnumbers

\usepackage[all]{xy}
\usepackage{hyperref}
\usepackage[shortlabels]{enumitem}
\usepackage{aliascnt}
\usepackage{microtype}

\newcommand{\CatCu}{\ensuremath{\mathrm{Cu}}}
\newcommand{\CuSgp}{$\CatCu$-sem\-i\-group}
\newcommand{\CuMor}{$\CatCu$-mor\-phism}
\newcommand{\axiomO}[1]{(O#1)}
\newcommand{\Cu}{\ensuremath{\mathrm{Cu}}}
\newcommand{\andSep}{\,\,\,\text{ and }\,\,\,}
\newcommand{\ca}{C*-algebra}
\newcommand{\NNbar}{\overline{\mathbb{N}}}
\newcommand{\St}{\mathrm{St}}
\newcommand{\DF}{DF}
\newcommand{\NN}{\mathbb{N}}
\newcommand{\CC}{\mathbb{C}}
\newcommand{\RR}{\mathbb{R}}
\newcommand{\rc}{\mathrm{rc}}
\newcommand{\Aff}{\mathrm{Aff}}

\renewcommand{\leq}{\leqslant}
\renewcommand{\geq}{\geqslant}

\DeclareMathOperator{\QT}{QT}
\DeclareMathOperator{\LAff}{LAff}
\DeclareMathOperator{\Ped}{Ped}

\newtheorem{lma}{Lemma}[section]

\newaliascnt{thmCt}{lma}
\newtheorem{thm}[thmCt]{Theorem}
\aliascntresetthe{thmCt}

\newaliascnt{corCt}{lma}
\newtheorem{cor}[corCt]{Corollary}
\aliascntresetthe{corCt}

\newaliascnt{prpCt}{lma}
\newtheorem{prp}[prpCt]{Proposition}
\aliascntresetthe{prpCt}

\newtheorem*{lma*}{Lemma}

\newtheorem*{thm*}{Theorem}

\theoremstyle{definition}

\newaliascnt{dfnCt}{lma}
\newtheorem{dfn}[dfnCt]{Definition}
\aliascntresetthe{dfnCt}

\newaliascnt{pgrCt}{lma}
\newtheorem{pgr}[pgrCt]{}
\aliascntresetthe{pgrCt}

\newaliascnt{rmkCt}{lma}
\newtheorem{rmk}[rmkCt]{Remark}
\aliascntresetthe{rmkCt}

\numberwithin{equation}{section}

\begin{document}

\title[C*-algebras of stable rank one and their Cuntz semigroups]{C*-algebras of stable rank one and their Cuntz semigroups}

\date{\today}
\author{Ramon Antoine}
\author{Francesc Perera}
\author{Leonel Robert}
\author{Hannes Thiel}

\address{
R.~Antoine \& F.~Perera, Departament de Matem\`{a}tiques,
Universitat Aut\`{o}noma de Barcelona,
08193 Bellaterra, Barcelona, Spain}
\email[]{ramon@mat.uab.cat, perera@mat.uab.cat}

\address{
L.~Robert,
Department of Mathematics,
University of Louisiana at Lafayette,
Lafayette, 70504-3568, USA}
\email{lrobert@louisiana.edu}

\address{
H.~Thiel, Mathematisches Institut, Universit\"at M\"unster,
Einsteinstr.~62, 48149 M\"unster, Germany}
\email[]{hannes.thiel@uni-muenster.de}



\begin{abstract}
The uncovering of new structure on the Cuntz semigroup of a C*-algebra of stable rank one leads to several applications:
We answer affirmatively, for the class of stable rank one C*-algebras, a conjecture by Blackadar and Handelman on dimension functions, the Global Glimm Halving problem, and the problem of realizing functions on the cone of $2$-quasitraces as ranks of Cuntz semigroup elements. We also gain new insights into the comparability properties of positive elements in C*-algebras of stable rank one.
\end{abstract}

\maketitle

\section{Introduction}
The Murray-von Neumann equivalence of projections is one of the fundamental concepts in operator algebra theory. 
It serves as the basis for the type classification of von Neumann algebra factors. 
Further, it leads to the construction of the Murray-von Neumann monoid of projections and of its enveloping group, the $K_0$-group, both important invariants associated to a C*-algebra. While the abundance of projections in a von Neumann algebra makes the Murray-von Neumann monoid of projections a very appropriate invariant, this is less so for arbitrary C*-algebras, which may lack any nontrivial projections. 
A general recipe to remedy this problem is to substitute projections by positive elements. The Cuntz comparison relation among the positive elements of a C*-algebra is a natural analogue of the Murray-von Neumann comparison of projections (with some caveats). From this relation, the Cuntz semigroup is built in very much the same way that the Murray-von Neumann monoid is constructed from Murray-von Neumann equivalence classes of projections. 

The Cuntz semigroup is a very sensitive device since it captures a great deal of the structure of the C*-algebra it is attached to. In the early work of Cuntz, Blackadar, and Handelman (\cite{Cun78DimFct, BlaHan82DimFct}), it was used as a tool to study the traces and quasitraces on a C*-algebra (which induce functionals on the Cuntz semigroup). More recently, the Cuntz semigroup has been used to formulate numerous ``regularity properties" of the sort that appear in the classification program for simple nuclear C*-algebras. Notably, almost unperforation in the Cuntz semigroup features prominently in the classification program and in the work on the Toms-Winter conjecture; see \cite{Win12NuclDimZstable, Ror04StableRealRankZ, CasEviTikWhiWin19arX:NucDimSimple, KirRor14CentralSeq, Sat12arx:TraceSpace, TomWhiWin15ZStableFdBauer}.
Yet another use of the Cuntz semigroup, and of the functor associated to it, is as a classification invariant for nonsimple C*-algebras; see \cite{CiuEll08InvCaSR1, Rob12LimitsNCCW}.

We now briefly recall the definition of the Cuntz semigroup. Given a C*-algebra $A$ and positive elements $a,b\in A$,  we say that $a$ is Cuntz subequivalent to $b$, and write $a\precsim b$, if there is a sequence $(d_n)_n$ in $A$ such that $d_n^*bd_n\to a$ in norm. We say that $a$ is Cuntz equivalent to $b$, and write $ a\sim b$, if both $a\precsim b$ and $b\precsim a$ occur. Let us consider these relations applied to the positive elements of  $A\otimes\mathcal K$, where $\mathcal K$ denotes the C*-algebra of compact operators on a separable Hilbert space. The Cuntz semigroup $\Cu(A)$ is the set of Cuntz equivalence classes of positive elements of $A\otimes\mathcal K$ endowed with the order induced by Cuntz subequivalence and with the addition operation induced by orthogonal sums. If instead of positive elements in $A\otimes\mathcal K$ we consider positive elements in matrix algebras over $A$, we arrive at the non-complete Cuntz semigroup $W(A)$, which is the object originally defined by Cuntz in \cite{Cun78DimFct}. We always have  $W(A)$ embedded in $\Cu(A)$, and also that $\Cu(A)\cong W(A\otimes\mathcal{K})$ (see \cite{CowEllIva08CuInv}). Our focus here will be largely on $\Cu(A)$. 

A module picture of $\Cu(A)$ was made available by Coward, Elliott and Ivanescu in \cite{CowEllIva08CuInv}.
In this picture, one defines suitable notions of equivalence and subequivalence among the countably generated Hilbert C*-modules over $A$. The set of such equivalence classes becomes an ordered semigroup with the addition operation induced by direct sums and with the order induced by the subequivalence relation.
It was proved in \cite{CowEllIva08CuInv} that the resulting object is isomorphic to $\Cu(A)$, as defined above.

Recall that a unital C*-algebra has stable rank one if its set of invertible elements is dense, while a nonunital C*-algebra has stable rank one if its unitization does. The class of C*-algebras of stable rank one is closed under natural constructions such as matrix formation, corners, and inductive limits.
If $A$ is a simple, unital, stably finite C*-algebra that absorbs the Jiang-Su algebra $\mathcal{Z}$, then $A$ has stable rank one; see \cite[Theorem~6.7]{Ror04StableRealRankZ}. Stable rank one in itself does not constitute a regularity property of the kind encountered in the Elliott classification program, such as $\mathcal Z$-stability or finite nuclear dimension. For example, Toms' examples of non-regular C*-algebras in \cite{Tom06FlatDimGrowth, Tom08ClassificationNuclear} have stable rank one.

The Hilbert C*-modules picture of the Cuntz semigroup
simplifies considerably for C*-algebras of stable rank one: $\Cu(A)$ consists of the set of isomorphism classes
of countably generated Hilbert C*-modules over $A$ with addition induced by direct sums and order by Hilbert C*-module embeddings (see \cite{CowEllIva08CuInv}). Also under the stable rank one assumption, the Cuntz subequivalence relation on positive elements adopts a form closely resembling Murray-von Neumann equivalence: $a\precsim b$ if and only if there is $x\in A$ such that $a=x^*x$ and $xx^*\in \overline{bAb}$ (see \cite[Proposition 2.5]{CiuEllSan11LimitType1}).

In this paper we investigate the Cuntz semigroups of C*-algebras of stable rank one. By unraveling fine structural properties of these objects, we are able to resolve relevant questions on dimension functions and on divisibility and comparability properties of C*-algebras of stable rank one. These results represent an advance in the theory of C*-algebras and push further the work by the fourth author in \cite{Thi17arX:RksOps}, as we detail below.
One of our key results is as follows:

\begin{thm*}[\ref{thm:CuARiesz}, \ref{Cstarinf-semilattice}]
Let $A$ be a C*-algebra of stable rank one. Then $\Cu(A)$ has the Riesz interpolation property.
If $A$ is also separable, then every pair of elements in $\Cu(A)$ has an infimum, and addition in $\Cu(A)$ is distributive over the infimum operation.
\end{thm*}

This theorem proves especially useful when combined with the properties encapsulated in the abstract axioms of Cu-semigroups. Equipped with these tools, we  tackle  a  number of questions which we describe next.

\begin{pgr}\emph{A conjecture by Blackadar and Handelman.}
Let $A$ be a unital C*-algebra. A map $d\colon W(A)\to [0,\infty)$ is called a dimension function if it is additive, order-preserving and maps the class of the unit to $1$. In other words, a dimension function is a state on $W(A)$. Denote by $\DF(A)$ the set of all dimension functions endowed with the topology of pointwise convergence. Blackadar and Handelman conjectured in \cite{BlaHan82DimFct} that $\DF(A)$ is a Choquet simplex for all C*-algebras $A$.	This conjecture has been confirmed in a number of instances, but remains open in general;
see	\cite{Per97StructurePositive,BroPerTom08CuElliottConj,AntBosPerPet14GeomDimFct,Sil16arX:ConjDF}.
The Riesz Interpolation Property in the Cuntz semigroup readily implies that $\DF(A)$ is a Choquet simplex.
We thus confirm the Blackadar-Handelman conjecture for all unital C*-algebras of stable rank one: 

\begin{thm*}[\ref{thm:BlacHanconjecture}]
Let $A$ be a unital C*-algebra of stable rank one. Then  $\DF(A)$ is a Choquet simplex.
\end{thm*}
\end{pgr}

\begin{pgr}
\emph{The Global Glimm Halving Problem.}
A result of Glimm says that if a C*-algebra $A$ has an irreducible representation of dimension at least $k\in\NN$, then there exists a non-zero *-homomorphism from $M_k(C_0((0,1]))$ into $A$.
The Global Glimm Halving Problem was formulated for the first time by Kirchberg and R{\o}rdam in
\cite[Definition~4.12]{KirRor02InfNonSimpleCalgAbsOInfty}, while studying nonsimple purely infinite C*-algebras (where it was termed the Global Glimm Halving Property). For a unital C*-algebra $A$, 
the problem asks to prove the existence for all $k\in \NN$ of a *-homomorphism $M_k(C_0((0,1]))\to A$  whose  range generates $A$ as a closed two-sided ideal, provided that $A$ has no nonzero, finite dimensional representations. Assuming an affirmative answer to this problem, Kirchberg and R{\o}rdam  show that the notions of pure infiniteness and weak pure infiniteness agree; see \cite[Theorem~9.1]{KirRor02InfNonSimpleCalgAbsOInfty}. The Global Glimm Halving Problem  has been answered affirmatively whenever $A$ has Hausdorff, finite dimensional, primitive ideal space (see \cite[Theorem~4.3]{BlaKir04GlimmHalving}), and whenever $A$ is a C*-algebra with real rank zero (\cite{EllRor06Perturb}).
The problem remains open in general. 

In \cite{RobRor13Divisibility}, the Global Glimm Halving Problem is translated into an equivalence of divisibility properties on the Cuntz semigroup of the C*-algebra. We rely on this alternative formulation in order to solve the problem affirmatively for C*-algebras of stable rank one. In the unital case, this reads as follows:

\begin{thm*}[\ref{CstarGG}, \ref{nonsepGG}]
Let $A$ be a unital C*-algebra of stable rank one, and let $k\in\NN$.
Then $A$ has no nonzero representations of dimension less than $k$ if and only if there exists a *-homomorphism $\varphi\colon M_k(C_0((0,1]))\to A$ with full range.
\end{thm*}

We note that the theorem above does not require $A$ to have no nonzero finite dimensional representations.
One does not expect that this strong solution of the Global Glimm Halving Problem holds for general \ca{s}.

We also remark that the solution of the Global Glimm Halving Problem for stable rank one \ca{s} is a significant step forward from the real rank zero case. Indeed, while the primitive ideal space of a real rank zero C*-algebra has a basis of compact, open sets (hence, it is zero dimensional if it is also Hausdorff),
there is no dimensional restriction on the primitive ideal space of  a stable rank one \ca{}. For example, if $X$ is any compact, Hausdorff space and $\mathcal{R}$ is the Jacelon-Razak algebra, then $C(X,\mathcal{R})$ has stable rank one by \cite[Corollary~3.8]{San12arX:ReductionDim}, while its primitive ideal space is homeomorphic to $X$. Thus, one cannot just think of C*-algebras of stable rank one as generalized bundles over one dimensional spaces.
\end{pgr}

\begin{pgr}
\label{pgr:rank}
\emph{Realizing functions on $\QT(A)$ as ranks of Cuntz semigroup elements}.
As mentioned above, the Cuntz semigroup was introduced in \cite{Cun78DimFct} as a tool to study quasitraces on \ca{s}. 
The seminal paper of Blackadar and Handelman \cite{BlaHan82DimFct} continued the study of quasitraces and states on the Cuntz semigroup. This work was extended further in \cite{BlaKir04PureInf} and \cite{EllRobSan11Cone}, in order to allow for $[0,\infty]$-valued quasitraces and functionals. It follows from these works that a lower semicontinuous $[0,\infty]$-valued 2-quasitrace $\tau$ on a  \ca{} $A$ gives rise to a function $d_\tau\colon \Cu(A)\to [0,\infty]$ that preserves addition, order and suprema of increasing sequences.  
More precisely, given a positive element $a\in A\otimes\mathcal K$, we set 
\[
d_\tau([a]) = \lim\limits_{n\to\infty} \tau(a^{1/n}).
\]  

Let $\QT(A)$ denote the set of lower semicontinuous $[0,\infty]$-valued 2-quasitraces.
Fix an element $[a]\in \Cu(A)$ and consider the map $\QT(A)\to [0,\infty]$ given by $\tau \mapsto d_\tau([a])$.
This is called the \emph{rank} induced by $[a]$. (Observe that  if $A=M_n(\CC)$ and $\tau$ is the normalized trace
on $M_n(\CC)$, then $d_\tau([a])$ is the rank of $a$.)

The \emph{rank problem} asks to describe the functions on $\QT(A)$ that arise as ranks of elements of $\Cu(A)$.
Ranks of Cuntz semigroup elements are linear, lower semicontinuous, and satisfy a technical approximation property whose definition we defer to \autoref{dualofcone}.
The collection of all functions with these properties is denoted by $L(\QT(A))$.
One can then ask, more concretely, whether all functions in $L(\QT(A))$ can be realized as ranks of Cuntz semigroup elements.

If $A$ is simple, then nonzero functions in $L(\QT(A))$ are in natural bijection with the lower semicontinuous, affine functions defined on the simplex of normalized quasitraces $\QT(A)_1$ and with values in $(0,\infty]$.
In this setting, the rank problem was first raised by N.~Brown, and has been solved in a number of instances:
For simple C*-algebras that tensorially absorb the Jiang-Su algebra~$\mathcal{Z}$, the problem is solved in \cite[Theorem~5.5]{BroPerTom08CuElliottConj} in the exact, unital case, and in \cite[Corollary~6.8]{EllRobSan11Cone} dropping both exactness and existence of a unit. Assuming simplicity, exactness, strict comparison of positive elements, and that $\QT(A)_1$ is a Bauer simplex with finite dimensional extreme boundary, a solution is obtained in \cite{DadTom10Ranks}.

The fourth author obtained in \cite[Theorem~8.11]{Thi17arX:RksOps} a solution to the rank problem for every separable, simple, non-elementary, unital C*-algebra $A$ of stable rank one. 
More concretely, given a lower semicontinuous, affine function $f\colon\QT(A)_1\to(0,\infty]$, there exists a positive element $a\in A\otimes\mathcal{K}$ such that $d_\tau([a])=f(\tau)$ for all $\tau$. 

In this paper we extend the techniques developed in \cite{Thi17arX:RksOps} and obtain solutions to the rank problem in different settings. By removing the assumptions of simplicity and existence of a unit we obtain:

\begin{thm*}[\ref{thm:realizingprob1}]
Let $A$ be a separable C*-algebra of stable rank one that has no nonzero, elementary ideal-quotients (that is, there are no closed, two-sided ideals $J\subseteq I$ of $A$ such that $I/J$ is a nonzero elementary C*-algebra).
Then every function in $L(QT(A))$ can be realized as the rank of a Cuntz semigroup element.
\end{thm*}

In the unital case the previous result translates into the following theorem. (We show in \autoref{prob2nonsep} that separability can be dropped in the theorem below.)

\begin{thm*}[\ref{thm:realizingprob2}]
Let $A$ be a separable, unital C*-algebra of stable rank one that has no finite dimensional representations. Then every lower semicontinuous, affine function on $\QT(A)_1$ with values in $(0,\infty]$ can be realized as the rank of a Cuntz semigroup element.
\end{thm*}
\end{pgr}

These realization results are key ingredients in establishing the results regarding comparability properties in the next section.

\begin{pgr}\emph{Comparability properties}. 
Comparability properties in the Cuntz semigroup, such as strict comparison (equivalently, almost unperforation), $m$-comparison, and finite radius of comparison, measure degrees of  regularity of the C*-algebra; see, for example, \cite{Rob11NuclDimComp, Tom06FlatDimGrowth, Tom08InfFamily, BlaRobTikTomWin12AlgRC}. For 
simple nuclear C*-algebras, the Toms-Winter conjecture asserts the equivalence of the properties of $\mathcal Z$-stability, finite nuclear dimension, and strict comparison (in the Cuntz semigroup). Regularity in the Cuntz semigroup, however, may be encountered in C*-algebras that are both non-nuclear and tensorially prime. For example, the reduced C*-algebra of the free group on infinitely many generators has strict comparison. 

The additional structure in the Cuntz semigroup brought about by the stable rank one property entails that seemingly different comparability properties are in fact equivalent.	Although our results do not require the assumption of simplicity, we highlight here the simple unital case (see \autoref{SecSupersoft} for the relevant definitions):

\begin{thm*}[\ref{thm:locweak}, \ref{thm:strict}]
Let $A$ be a simple, unital, separable C*-algebra of stable rank one. 
\begin{enumerate}[{\rm (i)}]
\item		
$A$ has finite radius of comparison in the sense of Toms (\cite{Tom06FlatDimGrowth}) if and only if the subsemigroup $W(A)$ consists precisely of the elements in $\Cu(A)$ whose rank is a bounded function on  the set of $2$-quasitracial states. 
\item
If $A$ has either $m$-comparison for some $m\in \NN$ (in the sense defined by Winter in \cite{Win12NuclDimZstable}) or local weak comparison (in the sense defined by Kirchberg and R{\o}rdam in \cite{KirRor14CentralSeq}) then $A$ has strict comparison.
\end{enumerate}	
\end{thm*}
\end{pgr}

In Section \ref{secnonsep}, we show that some of the results established in 
Sections \ref{sec:GGH} and \ref{sec:realizeRank} continue to hold removing the assumption of separability. This is accomplished using results of model theory for C*-algebras. For background on this theory, we refer the reader to \cite{FarHarLupRobTikVigWin06arX:ModelThy}.

\subsection*{Acknowledgments}

This work was initiated during the intensive research programme `Operator Algebras: Dynamics and Interactions' at the Centre de Recerca Matem\`{a}tica (CRM) in Barcelona, in the Spring of 2017.
The authors would like to thank the CRM for financial support and for providing inspiring working conditions.

Part of this research was conducted while the authors attended the workshop `Future Targets in the Classification Program for Amenable C*-Algebras' at the Banff International Research Station (BIRS), September 2017, and while they attended the Mini-workshop on the Cuntz semigroup at the University of Houston, June 2018.
The authors would like to thank the involved institutions for their kind hospitality.

The two first named authors were partially supported by MINECO (grant No.\ MTM2014-53644-P and No.\  MTM2017-83487-P), and by the Comissionat per Universitats i Recerca de la Generalitat de Catalunya (grant No.\ 2017-SGR-1725).
The fourth named author was partially supported by the Deutsche Forschungsgemeinschaft (DFG, German Research Foundation) under the SFB 878 (Groups, Geometry \& Actions) and under Germany's Excellence Strategy EXC 2044-390685587 (Mathematics M\"{u}nster: Dynamics-Geometry-Structure).

It is also a pleasure to thank the anonymous referees for their thorough reading of a former version of this work and for providing many detailed comments and suggestions which have largely improved the paper.

\section{Preliminaries}	

In this section we recall basic notions concerning the Cuntz semigroup of a C*-algebra and the category $\CatCu$ it belongs to. For a fuller account, we refer the reader to \cite{AraPerTom11Cu}, \cite{OrtRorThi11CuOpenProj}, \cite{AntPerThi18TensorProdCu}, and the references therein.

\begin{pgr}
\label{pgr:Cuntz}
\emph{The Cuntz semigroup.} 
Let $A$ be a C*-algebra. Denote by  $A_+$ the positive elements in $A$. Let us recall the definition of the Cuntz semigroup of $A$ in terms of positive elements: 
Given $a,b\in A_+$, one says that $a$ is \emph{Cuntz subequivalent} to $b$, denoted $a\precsim b$, if there exists a sequence $(d_n)_n$ in $A$ such that $d_n^*bd_n\to a$ in norm.
The elements $a$ and $b$ are \emph{Cuntz equivalent}, denoted $a\sim b$ , if $a\precsim b$ and $b\precsim a$. This is an equivalence relation. Let $[a]$ denote the equivalence class of $a$.
The Cuntz semigroup of $A$ is defined as
\[
\Cu(A)= \big\{ [a] : a\in (A\otimes\mathcal{K})_+ \big\}.
\]
That is, $\Cu(A)$ is the set of Cuntz equivalence classes of positive elements in the C*-algebra $A\otimes\mathcal{K}$.
(Here, and in the sequel, $\mathcal{K}$ denotes the C*-algebra of compact operators on the Hilbert space $\ell^2(\NN)$.)
The Cuntz semigroup $\Cu(A)$ is endowed with the order $[a]\leq [b]$ if $a\precsim b$ and the addition operation $[a]+[b]=[a'+b']$, where $a',b'\in (A\otimes\mathcal K)_+$ are chosen in such a way that $a\sim a'$, $b\sim b'$ and $a'b'=0$  (such elements always exist).
In this way, $\Cu(A)$ is an abelian, partially ordered semigroup.

Given $\varepsilon>0$ and $a\in A_+$, we denote by $(a-\varepsilon)_+$ the element $f_{\varepsilon}(a)$, where $f_{\varepsilon}(t)=\max(t-\varepsilon, 0)$. An important technical tool in Cuntz subequivalence is proved in \cite[Lemma 2.2]{KirRor02InfNonSimpleCalgAbsOInfty}:
Given $a,b\in A_+$ such that $\Vert a-b\Vert<\varepsilon$, then there exists a contraction $d\in A$ such that $(a-\varepsilon)_+=d^*bd$ (and, in particular, $(a-\varepsilon)_+\precsim b$). It is also known, and commonly used, that if $(a-\varepsilon)_+\precsim b$ for all $\varepsilon>0$, then $a\precsim b$.

Recall that a \ca{}\ $A$  is termed \emph{elementary} if it is isomorphic to the C*-algebra of compact operators on some Hilbert space.  
In this case, the map that assigns to each operator its rank induces an isomorphism $\Cu(A)\cong\NNbar$, where $\NNbar=\{0,1,\dots,\infty\}$.

We will focus largely on Cuntz semigroups of C*-algebras of stable rank one. As pointed out in the introduction, in this case the Cuntz semigroup is isomorphic to the set of isomorphism classes of countably generated Hilbert C*-modules over the C*-algebra. The Hilbert C*-modules picture of $\Cu(A)$ is developed in \cite{CowEllIva08CuInv}. In the case that $A$ has stable rank one, this picture adopts the following simpler form: given Hilbert C*-modules $H_1$ and $H_2$ over $A$, we have $[H_1]\leq [H_2]$ if $H_1$ embeds in $H_2$ as a Hilbert C*-submodule, and $[H_1]+[H_2]=[H_1\oplus H_2]$; see \cite[Theorem~3]{CowEllIva08CuInv}. 
\end{pgr}

\begin{pgr}
\emph{The category $\CatCu$.}
Some of the properties of the Cuntz semigroup of a C*-algebra can be abstracted into a category termed $\CatCu$, whose objects are called \emph{abstract Cuntz semigroups}, or simply \emph{\CuSgp{s}}. We recall the main definitions.

Throughout this paper all semigroups will be abelian, written additively, and with a zero element denoted by $0$. We will assume that our ordered semigroups are positively ordered. In particular, if $x+z=y$ for elements $x,y,z$ in such a semigroup, then $x\leq y$.

Let $S$ be an ordered semigroup. Given $x,y\in S$, let us write $x\ll y$ if whenever $(y_n)_n$ is an increasing sequence in $S$ such that the supremum $\sup_n y_n$ exists and satisfies $y\leq \sup_n y_n$, then there exists $n_0$ such that $x\leq y_{n_0}$.
This is a transitive relation on $S$, sometimes called the \emph{way-below relation} or also the \emph{compact containment relation}; see \cite[Definition~I-1.1, p.49]{GieHof+03Domains} and \cite[Paragraph~2.1.1, p.11]{AntPerThi18TensorProdCu} for details.
If $x\in S$ satisfies $x\ll x$, then we say that $x$ is a \emph{compact} element.

The semigroup $S$ is called a \CuSgp{} if it satisfies the following axioms:
\begin{enumerate}
	\item[\axiomO{1}]
	Every increasing sequence in $S$ has a supremum.
	\item[\axiomO{2}]
	For each $x\in S$ there exists a sequence $(x_n)_n$ such that $x_n\ll x_{n+1}$ for every $n$, and $x=\sup_n x_n$.
	\item[\axiomO{3}]
	If $x'\ll x$ and $y'\ll y$, then $x'+y'\ll x+y$.
	\item[\axiomO{4}]
	If $(x_n)_n$ and $(y_n)_n$ are increasing sequences in $S$, then $\sup_n (x_n+y_n)=\sup_n x_n+\sup_n y_n$.
\end{enumerate}

We call a sequence $(x_n)_n$ satisfying $x_n\ll x_{n+1}$ for all $n$ a \emph{$\ll$-increasing sequence}. It is sometimes also called a \emph{rapidly increasing sequence}.

Given \CuSgp{s} $S$ and $T$, a \emph{\CuMor{}} from $S$ to $T$ is a map $S\to T$ that preserves 0, addition, order, the relation $\ll$ and suprema of increasing sequences.
The category $\CatCu$ has as objects the \CuSgp{s}, and as morphisms the \CuMor{s}.
\end{pgr}

\begin{pgr}
It was proved in \cite{CowEllIva08CuInv} that the Cuntz semigroup $\Cu(A)$ of a C*-algebra $A$ is a \CuSgp.
Further, every *-homomorphism $\varphi\colon A\to B$ between C*-algebras induces a \CuMor{} $\Cu(\varphi)\colon\Cu(A)\to\Cu(B)$ by sending the class of $a\in (A\otimes \mathcal K)_+$ to the class of $\varphi(a)\in (B\otimes \mathcal K)_+$. This defines a functor from the category of C*-algebras to the category $\CatCu$. By \cite[Corollary~3.2.9.]{AntPerThi18TensorProdCu}, this functor preserves arbitrary inductive limits (sequential inductive limits are  covered by \cite[Theorem~2]{CowEllIva08CuInv}).
\end{pgr}

\begin{pgr}
\label{pgr:O5O6}\emph{Almost algebraic order}.
The Cuntz semigroup of a C*-algebra is known to satisfy an additional axiom which we now describe. 
Let $S$ be a \CuSgp. We say that $S$ has \emph{almost algebraic order}, or that $S$ satisfies axiom \axiomO{5}, if given $x',x,z\in S$ such that $x'\ll x\leq z$, there exists $w\in S$ such that $x'+w\leq z\leq x+w$. A consequence of \axiomO{5} that we use frequently below is that if $x\leq z$ and $x$ is compact (that is, $x\ll x$), then $x+w=z$ for some $w\in S$.

If $A$ is a C*-algebra, then $\Cu(A)$ satisfies \axiomO{5}; see \cite[Lemma~7.1]{RorWin10ZRevisited}.
In fact, a strengthening of \axiomO{5}, defined in \cite[Definition~4.1]{AntPerThi18TensorProdCu}, also holds for the Cuntz semigroups of all C*-algebras (see \cite[Proposition~4.6]{AntPerThi18TensorProdCu}). However, we will not make use of this stronger form of \axiomO{5} in this paper. 
\end{pgr}

\begin{pgr}
\label{pgr:weakcancellation}
A \CuSgp{} $S$ is said to have \emph{weak cancellation} if $x+z\ll y+z$ implies $x\ll y$ for all $x,y,z\in S$. 
This condition can be rephrased in a number of ways, which we include below for completeness.
\begin{lma*} 
	\label{lma:wcancellation} Let $S$ be a \CuSgp. Then the following conditions are equivalent:
	\begin{enumerate}[{\rm (i)}]
		\item $S$ has weak cancellation;
		\item If $x,y,z\in S$ are such that $x+z\ll y+z$, then $x\leq y$.
		\item If $x,y,z, z'\in S$ are such that $x+z\leq y+z'$ and $z'\ll z$, then  $x\leq y$.
	\end{enumerate}
\end{lma*}
\begin{proof}
	It is clear that (i) implies (ii). Assume (ii) and that $x+z\leq y+z'$ for $x,y,z,z'\in S$ with $z'\ll z$. Let $x'\ll x$. Then 
	\[
	x'+z'\ll x+z\leq y+z'.\] 
	By (ii), $x'\leq y$. Since $x'$ is arbitrary, we get by \axiomO{2} that $x\leq y$. Therefore (iii) is proved. Finally, suppose that (iii) holds, and let $x,y,z\in S$ be such that $x+z\ll y+z$. Choose $y'\ll y$ and $z'\ll z$ such that $x+z\leq y'+z'$. By (iii), $x\leq y'\ll y$, and thus (i) holds. 
\end{proof}
In the coming sections we make frequent use of the fact that if $S$ has weak cancellation, then it has cancellation of compact elements, that is, if $x+z\leq y+z$ and $z\ll z$, then $x\leq y$. Indeed, this follows at once from (iii) of the lemma above.

If $A$ is a C*-algebra with stable rank one, then $\Cu(A)$ has weak cancellation. This is proven in \cite[Theorem~4.3]{RorWin10ZRevisited} for the `non-complete' version of the Cuntz semigroup $W(A)$. The result can also be applied to $\Cu(A)$, since $A\otimes \mathcal K$ has stable rank one when $A$ does, and, as mentioned in the introduction, $\Cu(A)\cong W(A\otimes \mathcal K)$ (see \cite[Appendix]{CowEllIva08CuInv}).
\end{pgr}

\begin{pgr}
\label{countable}
Let $S$ be a \CuSgp. Recall that $S$ is said to be \emph{countably based} if there exists a countable subset $B\subseteq S$ such that every element in $S$ is the supremum of a $\ll$-increasing sequence with elements in $B$.
If $A$ is a separable C*-algebra, then $\Cu(A)$ is countably based; see for example \cite{AntPerSan11PullbacksCu}, or \cite[Proposition~5.1.1]{Rob13Cone}. One important consequence of having a countably based semigroup is recorded in the following basic result:

\begin{lma*}
	Every upward directed set in a countably based \CuSgp{} has a supremum.
\end{lma*}	
\end{pgr}

\begin{pgr}
\label{pgr:ideal}
Let $S$ be a \CuSgp.
An \emph{ideal} of $S$ is an order-hereditary submonoid $I$ of $S$ that is closed under suprema of increasing sequences.
We define $x\leq_I y$ to mean that $x\leq y+z$ for some $z\in I$, and write $x\sim_I y$ if both $x\leq_I y$ and $y\leq_I x$ happen. 
The quotient \CuSgp{} $S/I$ is defined as $S/\!\sim_I$. We refer to \cite[Section~5.1]{AntPerThi18TensorProdCu} for details.

If $A$ is a C*-algebra and $I$ is a closed, two-sided ideal of $A$, then the inclusion map $I\to A$ induces a \CuMor{} $\Cu(I)\to\Cu(A)$ that identifies $\Cu(I)$ with an ideal in $\Cu(A)$. Further, it was proved in \cite{CiuRobSan10CuIdealsQuot} that the quotient map $A\to A/I$ induces an isomorphism $\Cu(A)/\Cu(I)\cong \Cu(A/I)$. Moreover, the assignment $I\mapsto\Cu(I)$ defines a natural bijection between closed, two-sided ideals of $A$ and ideals of $\Cu(A)$; see \cite[Proposition~5.1.10]{AntPerThi18TensorProdCu}.
\end{pgr}

The following proposition is a crucial ingredient in the proofs of Theorems~\ref{thm:CuARiesz}, \ref{Cstarinf-semilattice}, and~\ref{thm:upwardalpha}.
By embedding the Cuntz semigroup of a C*-algebra as an ideal of a larger Cuntz semigroup, it introduces suitable compact elements associated to elements of the original Cuntz semigroup.

\begin{prp}
\label{prp:unitization}
Let $A$ be a stable C*-algebra, and let $a\in A_+$.
Then there exists a C*-algebra $B$ and a projection $p_a\in B$ such that:
\begin{enumerate}[{\rm (i)}]
	\item
	$A$ is a closed, two-sided ideal of $B$.
	\item
	For $x\in\Cu(A)$, we have $x\leq [a]$ in $\Cu(A)$ if and only if $x\leq [p_a]$ in $\Cu(B)$.
	\item
	If $A$ has stable rank one, then so does $B$. 
\end{enumerate}
\end{prp}
\begin{proof}
In order to construct $B$, we first consider the (right) Hilbert C*-module $H=\overline{aA}$. Since $H$ is singly generated, it follows from Kasparov's stabilization theorem 
(\cite[Theorem~1.1.24]{JenTho91ElementsKK}) that $H$ is a direct summand of $\ell^2(A)$. That is, there is a Hilbert C*-module $H'$ such that $\overline{aA}\oplus H' \cong \ell^2(A)$. On the other hand, since $A$ is a stable C*-algebra, $\ell^2(A)\cong A$ as Hilbert C*-modules (\cite[Lemma~1.3.2]{JenTho91ElementsKK}). Thus, $\overline{aA}$ is isomorphic to a complemented Hilbert C*-submodule of $A$. Denote by $M(A)$ the multiplier algebra of $A$, which is isomorphic to the algebra of adjointable operators on $A$ (see, for example, 
\cite[p.5]{JenTho91ElementsKK}). Then the projection onto the said submodule yields a projection $p_a\in M(A)$ such that $\overline{aA}\cong p_aA$.
We define $B=C^*(p_a,A)\subseteq M(A)$. 

(i): By construction, $A$ is a closed, two-sided ideal of $B$.

(ii): Let $x\in\Cu(A)$. Since $A$ is stable, there exists $b\in A_+$ such that $x=[b]$.
Suppose that $x\leq [p_a]$ in $\Cu(B)$. Then $b\precsim p_a$ in $B$, and thus for every $\varepsilon>0$ there exists $w\in B$ such that $\Vert b-w^*p_aw\Vert<\varepsilon$. This implies that there is a contraction $d\in B$ such that $(b-\varepsilon)_+=d^*w^*p_awd$ (see \autoref{pgr:Cuntz}).
Set $v=p_awd\in p_aB$. Then we have $(b-\varepsilon)_+=v^*v$. As $v^*v\in A$ and $A$ is a closed, two-sided ideal of $B$ by (i), we also have $v\in A$. Therefore $v\in p_aB\cap A=p_aA\cong \overline{aA}$. Now we have that, as Hilbert C*-modules over $A$, 
\[
\overline{(b-\varepsilon)_+A}=\overline{v^*vA}\cong\overline{vv^*A}\subseteq p_aA\cong\overline{aA}.
\]
Thus, $\overline{(b-\varepsilon)_+A}$ embeds in $\overline{aA}$. Hence $(b-\varepsilon)_+\precsim a$, by \cite[Section~6]{CowEllIva08CuInv} (see also the proof of \cite[Theorem~4.33]{AraPerTom11Cu}, or \cite[Proposition~4.6]{OrtRorThi11CuOpenProj}). 
Since $\varepsilon>0$ is arbitrary, we conclude that $x\leq [a]$ in $\Cu(A)$.

Conversely, suppose that $x\leq [a]$ in $\Cu(A)$.
To show that $x\leq [p_a]$ in $\Cu(B)$, it suffices to prove that $[a]\leq [p_a]$ in $\Cu(B)$.
The latter follows since
\[
\overline{aB}=\overline{aA}\cong p_aA\subseteq p_aB,
\]
that is, $\overline{aB}$ embeds in $p_aB$ as Hilbert C*-modules over $B$.

(iii): Assume that $A$ has stable rank one. By construction, $B/A\cong \CC$. Thus, $B$ is an extension of $A$ and $\CC$, which both have stable rank one. Using \cite[Theorem~4.11]{Rie83DimSRKThy}, it follows that $B$ has stable rank one.
\end{proof}

\section{Riesz Interpolation and infima}

In this section, we prove that the Cuntz semigroup $\Cu(A)$ of any C*-algebra $A$ of stable rank one has the Riesz Interpolation Property.
If $A$ is also separable, then it follows that every pair of elements in $\Cu(A)$ has an infimum. Further, this semilattice structure is compatible with addition; see \autoref{Cstarinf-semilattice}.

In the sequel, we write $x_1,x_2\leq y_1,y_2$ to mean $x_i\leq y_j$ for $i,j=1,2$.

\begin{pgr}
The following axiom was introduced in \cite{Thi17arX:RksOps}.
A \CuSgp{} $S$ is said to satisfy axiom \axiomO{6+} if for every $a,b,c,x',x,y',y\in S$ satisfying
\[
a\leq b+c\,, \quad x'\ll x\leq a,b\,, \andSep y'\ll y\leq a,c,
\]
there exist $e,f\in S$ such that
\[
a\leq e+f\,, \quad x'\ll e\leq a,b\,, \andSep y'\ll f\leq a,c.
\]
Note that \axiomO{6+} is equivalent to the following property: for every $a,b,c,x',x\in S$ satisfying $a\leq b+c$ with $x'\ll x\leq a,b$, 
there exists $e\in S$ such that $a\leq e+c$ and $x'\ll e\leq a,b$. The equivalence between these two formulations is implicit in \cite[Lemma~6.3]{Thi17arX:RksOps}, and many times we will use the latter.

Axiom \axiomO{6+} is a strengthening of the axiom \axiomO{6} of almost Riesz decomposition introduced in \cite{Rob13Cone}. Unlike \axiomO{6}, which is known to hold for the Cuntz semigroup of any C*-algebra, there are C*-algebras whose Cuntz semigroup does not satisfy \axiomO{6+}. However, it was shown in \cite[Theorem~6.4]{Thi17arX:RksOps} that the Cuntz semigroup of any C*-algebra of stable rank one satisfies \axiomO{6+}.
\end{pgr}

\begin{lma}
\label{upwardlemma}
Let $S$ be a \CuSgp, and let $B\subseteq S$ be an order-hereditary subset of $S$ that is closed under suprema of increasing sequences.
Define
\[
B_\ll = \big\{x\in S : \text{there is y}\in B\text{ such that }x\ll y \big\}\,.
\]
If $B_\ll$ is an upward directed set, then this is also the case for $B$.
\end{lma}	
\begin{proof}
Let $x,y\in B$. Choose $\ll$-increasing sequences $(x_n)_n$ and $(y_n)_n$ in $S$ such that $x=\sup_n x_n$ and $y=\sup_n y_n$.
Then $x_n,y_n\in B_\ll$ for each $n$. Since $B_\ll$ is upward directed, there exists $z_1\in B_\ll$ such that $x_1,y_1\leq z_1$.
Suppose that, for $n\geq 1$, there are $z_1\leq z_2\leq\dots\leq z_n$ in $B_\ll$ such that $x_n,y_n\leq z_n$.
Using again that $B_\ll$ is upward directed, we may choose $z_{n+1}\in B_\ll$ such that $x_{n+1},y_{n+1},z_n\leq z_{n+1}$.
Now let $z=\sup_n z_n$. By construction $x,y\leq z$. Further $z$ belongs to $B$ since by assumption this set is closed under suprema of increasing sequences.
\end{proof}	

The lemma below is contained in \cite{Thi17arX:RksOps}, though not explicitly stated. We reproduce the proof here for convenience.

\begin{lma}
\label{lma:ht}
Let $S$ be a weakly cancellative \CuSgp{} satisfying \axiomO{5} and \axiomO{6+}, and let $e,x\in S$. Assume that $e$ is compact.
Then the set
\[
\big\{ z\in S : z\leq e,x \big\}
\]
is upward directed.	
\end{lma}	
\begin{proof}
Since the set $\{ z\in S : z\leq e,x \}$ is order-hereditary and closed under suprema of increasing sequences, it suffices to show by \autoref{upwardlemma} that the set
\[
\big\{ z'\in S : \text{there is }z\in S\text{ such that }z'\ll z\leq e,x \big\}
\]
is upward directed.

Let $z_1', z_2' \in S$ be such that there are $z_1, z_2\in S$ with
\[
z_1'\ll z_1\leq e,x, \andSep z_2'\ll z_2\leq e,x.
\]
By \axiomO{5} applied to $z_1'\ll z_1\leq e$, there exists $w\in S$ such that $z_1'+w\leq e\leq z_1+w$. Since $z_1\leq x$, we obtain $e\leq x+w$. We now apply \axiomO{6+} to this inequality and $z_2'\ll z_2\leq e,x$. Thus, we find $y\in S$ such that $e\leq y+w$ and $z_2'\ll y\leq e,x$. Hence
\[
z_1'+w\leq e\ll e\leq y+w,
\]
where we have used that $e$ is compact. By weak cancellation in $S$, we  obtain $z_1'\ll y$. Hence, $z_1',z_2'\ll y\leq e,x$. Choose $y'\in S$ with $z_1',z_2'\ll y'\ll y$. Then $y '$ has the desired properties.
\end{proof}

\begin{pgr}
\label{pgr:RieszInt}
Recall that an ordered semigroup $S$ has the \emph{Riesz Interpolation Property} if given $u,v,x,y\in S$ such that $u,v\leq x,y$, then there exists $z\in S$ with $u,v\leq z\leq x,y$. 
\end{pgr}

\begin{thm}
\label{thm:CuARiesz}
Let $A$ be a C*-algebra of stable rank one. Then $\Cu(A)$ has the Riesz Interpolation Property.
\end{thm}	
\begin{proof}
Let $x,y\in\Cu(A)$. We must show that the set $\{ z\in \Cu(A) : z\leq x,y \}$ is upward directed. If $x$ is compact, this follows from \autoref{lma:ht}. We next reduce the general case to this case relying on  \autoref{prp:unitization}.

We may assume that $A$ is stable. Choose $a\in A_+$ such that $x=[a]$.
Applying \autoref{prp:unitization} for $A$ and $a$, we obtain a C*-algebra $B$ with stable rank one that contains $A$ as a closed, two-sided ideal, and a projection $p_a\in B$ such that $z\in\Cu(A)$ satisfies $z\leq x$ if and only if $z\leq[p_a]$. Since $[p_a]$ is compact in $\Cu(B)$, and since $B$ has stable rank one, it follows from \autoref{lma:ht} that the set $\{ z\in\Cu(B) :  z\leq [p_a],y \}$ is upward directed. The inclusion $A\subseteq B$ identifies $\Cu(A)$ with an ideal in $\Cu(B)$. We claim that
\[
\big\{ z\in \Cu(A) : z\leq x,y \big\}
= \big\{ z\in\Cu(B) :  z\leq [p_a],y \big\},
\]
from which the result will follow.

Indeed, the inclusion `$\subseteq$' follows using that $x\leq [p_a]$.
To prove the converse inclusion, take $z\in \Cu(B)$ such that $z\leq [p_a],y$.
Since $\Cu(A)$ is an ideal of $\Cu(B)$ and $y\in\Cu(A)$, we have $z\in\Cu(A)$.
Now, since also $z\leq [p_a]$, we may use \autoref{prp:unitization} (ii) to conclude that $z\leq x$.
\end{proof}	

\begin{pgr}\emph{Inf-semilattice ordered semigroups}.
Recall that a partially ordered set $S$ is called an \emph{inf-semilattice}, or also a \emph{meet-semilattice}, if for every pair of elements $x$ and $y$ of $S$, the greatest lower bound of the set $\{x,y\}$ exists in $S$. We shall follow the usual notation and denote such infimum by $x\wedge y$.

We  say that an ordered semigroup $S$ is \emph{inf-semilattice ordered} if $S$ is an inf-semilattice and addition is distributive over the meet operation, that is,
\begin{equation}\label{distributiveformula}
(x+z)\wedge (y+z)=(x\wedge y)+z,
\end{equation}
for all $x,y,z\in S$.
\end{pgr}

\begin{lma}
\label{prp:infSupportProj}
Let $A$ be a stable C*-algebra and let $a\in A_+$.
Let the \ca{} $B$ and the projection $p_a\in B$ be as in \autoref{prp:unitization}.
Let $x\in\Cu(A)$ be such that $[p_a]\wedge x$ exists in $\Cu(B)$. Then $[a]\wedge x$ exists in $\Cu(A)$ and
\[
[a]\wedge x=[p_a]\wedge x.
\]
\end{lma}
\begin{proof}
Let $w=[p_a]\wedge x$. Since $w\leq x$ and since $\Cu(A)$ is an ideal of $\Cu(B)$, we obtain $w\in\Cu(A)$.
Now, we also have that $w\leq [p_a]$. Hence, $w\leq [a]$ by \autoref{prp:unitization}(ii).
Thus, $w$ is a lower bound for $[a]$ and $x$.

To show that $w$ is the largest lower bound, let $y\in\Cu(A)$ satisfy $y\leq [a]$ and $y\leq x$. Then $y\leq [p_a]$ in $\Cu(B)$, again by \autoref{prp:unitization}(ii). Therefore $y\leq [p_a]\wedge x=w$. Hence, $[a]\wedge x=[p_a]\wedge x$, as desired.
\end{proof}

\begin{thm}
\label{Cstarinf-semilattice}
Let $A$ be a separable C*-algebra of stable rank one. Then $\Cu(A)$ is an inf-semilattice ordered semigroup.
\end{thm}	
\begin{proof}
Without loss of generality, we may assume that $A$ is stable. By \autoref{thm:CuARiesz}, $\Cu(A)$ has the Riesz Interpolation Property.
Thus, given $x,y\in \Cu(A)$ the set $\{ z\in \Cu(A) : z\leq x,y\}$ is upward directed. Since $A$ is separable, $\Cu(A)$ is countably based.
Applying \autoref{countable}, we conclude that $\{ z\in\Cu(A) : z\leq x,y \}$ has a supremum, which is precisely $x\wedge y$.
Thus, $\Cu(A)$ is an inf-semilattice.

\emph{Claim: Given $a,b,c\in \Cu(A)$, we have:}
\begin{equation}
\label{byO6plus}
a\leq b+c \Rightarrow a\leq (a\wedge b)+(a\wedge c).
\end{equation}
Indeed, applying \axiomO{6+} in $a\leq b+c$ with $x=x'=y=y'=0$, we obtain $a\leq e+f$ for some $e\leq a,b$ and $f\leq a,c$. The existence of infima proves the claim.

In order to prove the distributivity of $\wedge$ over addition, we only need to show that
\begin{equation}
\label{ineq}
(x+z)\wedge (y+z)\leq ( x\wedge y)+z,
\end{equation}
for all $x,y,z\in\Cu(A)$, as the opposite inequality is straightforward.

We will first prove \eqref{ineq} in the case that both $x$ and $z$ are compact elements and then, through successive generalizations, extend this to the general case.

\emph{Step 1: We show that the inequality \eqref{ineq} is valid when $x$ and $z$ are compact.}
Let $w=(x+z)\wedge (y+z)$. Choose $w'\in \Cu(A)$ such that $w'\ll w$. Applying \axiomO{5} for the inequality $w'\ll w\leq x+z$, we find $v\in\Cu(A)$ such that $w'+v\leq x+z\leq w+v$. We get $x+z\leq y+z+v$. As $A$ has stable rank one, $\Cu(A)$ has cancellation of compact elements (see \autoref{lma:wcancellation} and the comments afterwards), and since $z$ is compact by assumption, we obtain $x\leq y+v$.
By \eqref{byO6plus}, $x\leq (x\wedge y)+v$. Adding $z$ on both sides we get $x+z\leq (x\wedge y)+v+z$. Hence, using that $x+z$ is compact,
\[
w'+v\leq x+z\ll x+z\leq (x\wedge y)+z+v\,.
\]
It now follows from weak cancellation that $w'\leq (x\wedge y)+z$.
Since $w'$ is arbitrary satisfying $w'\ll w$, the inequality \eqref{ineq} holds.

\emph{Step 2: We show that the inequality \eqref{ineq} is valid when $x$ is compact.}
Write $x+z=[b]$, with $b\in A_+$. Let $B$ and $p_b\in B$ be the C*-algebra of stable rank one and the projection, respectively, obtained in \autoref{prp:unitization}. Let $f=[p_{b}]\in \Cu(B)$, which is compact. Then $x+z\leq f$ and $f\wedge w=(x+z)\wedge w$ for all $w\in \Cu(A)$, by \autoref{prp:infSupportProj}. Since $x\leq f$ and $x$ is compact, there exists $z'\in \Cu(B)$ such that $x+z'=f$ (see \autoref{pgr:O5O6}).
Let us show that $z'$ is also compact. Since $x+z'=f$ is compact, there exists $z''\in\Cu(B)$ such that $z''\ll z'$ and $x+z''= f$.	Hence $f=x+z'=x+z''$, and by cancellation of compact elements in $\Cu(B)$, $z'=z''\ll z'$. Thus $z'$ is compact.

Now $x+z\leq f=x+z'$, and by cancellation of compact elements in $\Cu(B)$, we have that $z\leq z'$.
Since $x$ and $z'$ are compact in $\Cu(B)$, we may apply Step~1 to conclude that
\[
(x+z')\wedge (y+z')\leq (x\wedge y)+z'.
\]
Since $z\leq z'$, we get
\[
(x+z)\wedge (y+z)\leq (x\wedge y)+z'.
\]
Using \eqref{byO6plus}, we deduce that
\begin{align*}
(x+z)\wedge (y+z) 
&\leq (x\wedge y) + \big( z'\wedge (x+z)\wedge(y+z) \big) \\
&\leq (x\wedge y) + \big( z'\wedge (x+z) \big).
\end{align*}
The proof of Step~2 will be complete once we show $z'\wedge(x+z)=z$.
By cancellation of compact elements, and since $x$ is compact by assumption, this is equivalent to showing that $(z'\wedge (x+z))+x=z+x$.

Since $x$ and $z'$ are compact elements in $\Cu(B)$, we may use Step~1 to obtain
\[
\big( z'\wedge (x+z) \big) + x = (z'+x)\wedge (x+z+x).
\]
Now, we apply \autoref{prp:infSupportProj} at the second step and conclude
\[
(z'+x)\wedge (x+z+x)= f\wedge (x+z+x)=(z+x)\wedge (x+z+x)=z+x.
\]
Therefore $(z'\wedge (x+z))+x=z+x$, as desired.

\emph{Step 3: We show that the inequality \eqref{ineq} holds in general.}
Choose $a\in A_+$ such that $x=[a]$. 
Let $B$ and $p_a\in B$ be the C*-algebra of stable rank one and the projection, respectively, obtained in \autoref{prp:unitization}.
Let $e=[p_a]\in \Cu(B)$.
By Step~2, \eqref{ineq} holds in $\Cu(B)$ with $e$ in place of $x$.
This means that
\[
(e+z)\wedge (y+z)\leq (e\wedge y)+z.
\]

Now, by \autoref{prp:infSupportProj} we have $e\wedge y=x\wedge y$.
Therefore, the right hand side of the above inequality is precisely $(x\wedge y)+z$.
On the other hand, the left hand side dominates $(x+z)\wedge (y+z)$.
This proves the inequality in general.
\end{proof}

\begin{rmk}
\label{general}
Let $S$ be an inf-semilattice ordered semigroup, and let $x_i^{(k)}\in S$ for $k=1,\ldots,n$ and $i=1,\ldots,N_k$.
It follows from \eqref{distributiveformula} and induction that
\[
\sum_{k=1}^n \big(\bigwedge_{i=1}^{N_k} x_i^{(k)}\Big)=\bigwedge_{(i_1,\ldots,i_n)} \Big(\sum_{k=1}^n x_{i_k}^{(k)}\Big),
\]	
where $(i_1,\ldots,i_n)$ on the right hand side runs through $\{1,\ldots,N_1\}\times \cdots\times \{1,\ldots,N_n\}$.
\end{rmk}

\begin{rmk}
\label{automaticO6}
If $S$ is an inf-semilattice ordered \CuSgp, then $S$ satisfies \axiomO{6+}. Indeed, if we are given elements $a,b,c, x',x,y',y\in S$ such that 
\[
a\leq b+c\,, \quad x'\ll x\leq a,b\,, \andSep y'\ll y\leq a,c,
\]
then let $e=a\wedge b$ and $f=a\wedge c$. We clearly have $x'\ll x\leq e$ and $y'\ll y\leq f$. On the other hand, applying the formula obtained in \autoref{general} at the second step, we obtain
\[
a\leq (2a)\wedge (a+c)\wedge (a+b)\wedge (b+c)=(a\wedge b)+ (a\wedge c)=e+f\,.
\]
\end{rmk}

\begin{rmk}
Let $S$ be an inf-semilattice ordered \CuSgp.
Given $x\in S$ and an increasing sequence $(y_n)_n$ in $S$, we have
\[
\sup_n (x\wedge y_n)=x\wedge \sup_n y_n.
\]
Indeed, the inequality `$\leq$' follows since for each $k\in\NN$ we have $x\wedge y_k\leq x\wedge \sup_n y_n$.
To show the converse inequality, let $z'\in S$ be such that $z'\ll x\wedge\sup_n y_n$.
Since $z'\ll \sup_n y_n$,  there exists $k\in\NN$ such that $z'\leq y_k$.
Since also $z'\leq x$, we obtain $z'\leq x\wedge y_k\leq\sup_n(x\wedge y_n)$.
Finally, passing to the supremum over all $z'\ll x\wedge\sup_n y_n$, the desired inequality follows.
\end{rmk}

Next, we note some consequences of \autoref{Cstarinf-semilattice}.

\begin{cor}
\label{cor:preservationinfima}
Let $A$ be a separable C*-algebra of stable rank one, let $I$ be a closed, two-sided ideal of $A$, and let $\pi_I\colon A\to A/I$ denote the quotient map.
Then $\Cu(\pi_I)\colon \Cu(A)\to \Cu(A/I)$ preserves infima.
\end{cor}	
\begin{proof}
We view $\Cu(I)$ as an ideal of $\Cu(A)$ as in \autoref{pgr:ideal}.
Since $I$ is separable, $\Cu(I)$ has a largest element that we denote by $\omega_I$.
Notice that $2\omega_I=\omega_I$, and thus $\omega_I+\Cu(A)$ is an ordered subsemigroup of $\Cu(A)$.
By \cite[Theorem~1.1]{CiuRobSan10CuIdealsQuot}, $\Cu(\pi_I)$ is an ordered semigroup isomorphism
from $\omega_I+\Cu(A)$ to $\Cu(A/I)$.
It thus suffices to show that the map $x\mapsto x+\omega_I$ from $\Cu(A)$ to the subsemigroup $\omega_I+\Cu(A)$ preserves infima. 
Indeed, for $x,y\in \Cu(A)$ it follows from \autoref{Cstarinf-semilattice} that
\[
(x+\omega_I)\wedge (y+\omega_I)=(x\wedge y)+\omega_I.\qedhere
\]
\end{proof}

Another application of \autoref{Cstarinf-semilattice} allows us to compute the Cuntz semigroup of a particular case of pullbacks (see also \cite[Theorem 3.3]{AntPerSan11PullbacksCu}).

\begin{cor}
\label{pullback}
Let $A$ be a separable C*-algebra of stable rank one, and let $I,J\subseteq A$ be closed, two-sided ideals of $A$.
Then 
\[ 
\Cu(A/(I\cap J))\cong \Cu(A/I)\oplus_{\Cu(A/(I+J))} \Cu(A/J),
\]
where the right hand side denotes the pullback semigroup of pairs $(\bar{s},\bar{t})\in  \Cu(A/I)\oplus \Cu(A/J)$
such that $\bar{s}$ and $\bar{t}$ agree when mapped to $\Cu(A/(I+J))$.  	
\end{cor}
\begin{proof}
As in the proof of \autoref{cor:preservationinfima}, given an ideal $K$ of a separable C*-algebra~$B$, we denote by $\omega_K$ the largest element in $\Cu(K)$, and we identify $\Cu(B/K)$ with $\Cu(B)+\omega_K$.
Thus $\Cu(\pi_K)$ is identified with the map $\Cu(B)\to\Cu(B)+\omega_K$ given by $z\mapsto z+\omega_K$.

Observe that $\omega_{I+J}=\omega_I+\omega_J$. Therefore, the map $\Cu(A/I)\to\Cu(A/(I+J))$ is identified with the map 
$\Cu(A)+\omega_I\to\Cu(A)+\omega_I+\omega_J$  given by $z\mapsto z+\omega_J$. Likewise, the map $\Cu(A/J)\to\Cu(A/(I+J))$ is identified with the map $\Cu(A)+\omega_J\to\Cu(A)+\omega_I+\omega_J$ given by $z\mapsto z+\omega_I$.

Now, denote by $S$ the algebraic pullback of the diagram
\[
\xymatrix@R-10pt@C-10pt{
	& \Cu(A)+\omega_I\ar[d]	\\
	\Cu(A)+\omega_J\ar[r]&\Cu(A)+\omega_I+\omega_J
}	
\]
We clearly have a map $\Cu(A)+\omega_{I\cap J}\to S$, given by $z\mapsto (z+\omega_I,z+\omega_J)$.
Given  $z_1\in\Cu(A)+\omega_I$ and $z_2\in\Cu(A)+\omega_J$ with $(z_1,z_2)\in S$, we need to show that there exists a unique element $z\in \Cu(A)+\omega_{I\cap J}$ such that $z+\omega_I=z_1$ and $z+\omega_J=z_2$.

\emph{Existence:} 
We show that $z=z_1\wedge z_2$ is as required.
Since $(z_1,z_2)\in S$, we have 
\[
z_2+\omega_I=z_2+\omega_I+\omega_J=z_1+\omega_I+\omega_J=z_1+\omega_J.
\]
Using this equality at the second step, and \autoref{Cstarinf-semilattice} at the first step, we obtain
\[
(z_1\wedge z_2)+\omega_I=(z_1+\omega_I)\wedge (z_2+\omega_I)=z_1\wedge(z_1+\omega_J)=z_1.
\]	
Symmetrically, $(z_1\wedge z_2)+\omega_J=z_2$. Observe also that $z\in \omega_{I\cap J}+\Cu(A)$. Indeed, since $z_1=z_1+\omega_I$ and $\omega_{I\cap J}+\omega_I=\omega_I$, we get
\[
z_1+\omega_{I\cap J}=z_1+\omega_I+\omega_{I\cap J}=z_1+\omega_I=z_1.
\]
Similarly, $z_2+\omega_{I\cap J}=z_2$.
Applying \autoref{Cstarinf-semilattice} again, we get
\[
z+\omega_{I\cap J}
= (z_1\wedge z_2)+\omega_{I\cap J}
= (z_1+\omega_{I\cap J})\wedge(z_2+\omega_{I\cap J})
= z_1\wedge z_2
= z.
\]

\emph{Uniqueness:}
Suppose that $z'\in \Cu(A)+\omega_{I\cap J}$ satisfies $z'+\omega_I=z_1$ and $z'+\omega_J=z_2$.
Notice that $\omega_{I\cap J}=\omega_I\wedge\omega_J$. 
Then, using \autoref{Cstarinf-semilattice} at the third step, we obtain
\[
z'
=z'+\omega_{I\cap J}
=z'+(\omega_I\wedge\omega_J)
=(z'+\omega_I)\wedge (z'+\omega_J)
=z_1\wedge z_2=z. \qedhere
\]
\end{proof}

\begin{rmk}
\autoref{pullback} fails to hold if we drop the stable rank one hypothesis.
For example, set $A=M_2(C(S^2))$ and take $I=M_2(C_0(U))$ and $J=M_2(C_0(V))$, where $U$ and $V$ are disjoint open caps of the sphere.
Let $p,q\in M_2(C(S^2))$ be rank one projections with different classes in $K_0(C(S^2))$.
(For instance, $p$ is $e_{11}\otimes 1$ and $q$ is the Bott projection.)
Then the images of $p$ and $q$ are Cuntz equivalent in $A/I$ and $A/J$, but $[p]\neq [q]$.
\end{rmk}

\section{A conjecture of Blackadar and Handelman}
\label{sec:BHconj}

Let $A$ be a unital C*-algebra. 
Using upper-left corner embeddings $M_n(A)\to M_{n+1}(A)$, set $M_\infty(A)=\bigcup_n M_n(A)$, which has the structure of a local C*-algebra. Recall that the classical (non-complete) Cuntz semigroup $W(A)$ of $A$ is defined as
\[
W(A)=M_\infty(A)_+/\!\!\sim\,;
\]
see \cite{Cun78DimFct}.
It can also be described as the subsemigroup of $\Cu(A)$ of those classes $[a]$ with a representative $a\in M_\infty(A)_+$.
If $A$ has stable rank one, then $W(A)$ is a hereditary subset of $\Cu(A)$ by \cite[Lemma~3.4]{AntBosPer11CompletionsCu}, that is, whenever $x,y\in \Cu(A)$ satisfy $x\leq y$ with $y\in W(A)$, then $x\in W(A)$. 
Then $W(A)$ may alternatively be described as
\[
W(A) = \big\{ x\in \Cu(A) : x\leq n[a]\hbox{ for some } a\in A_+,n\in \NN \big\}.
\]

Following \cite{Cun78DimFct}, we denote the Grothendieck group of $W(A)$ by $K_0^*(A)$.
It is a partially ordered group with positive cone $K^*_0(A)_+=\{\overline{x}-\overline{y} : y\leq x\text{ in }W(A)\}$, where we denote by $\overline{x}$ the image of $x\in W(A)$ in $K_0^*(A)$.
A \emph{state} on $K_0^*(A)$ is an additive, order-preserving map $\lambda\colon K_0^*(A)\to\mathbb{R}$ with $\lambda(\overline{[1_A]})=1$.

In \cite[Section~3]{Cun78DimFct}, Cuntz defined a \emph{(normalized) dimension function} on $A$ as a map $d\colon M_\infty(A)_+\to[0,\infty)$ that satisfies $d(a\oplus b)=d(a)+d(b)$ for all $a,b\in M_\infty(A)_+$, $d(a)\leq d(b)$ whenever $a\precsim b$, and $d(1_A)=1$.
Each dimension function $d$ induces a state $\lambda_d$ on $K_0^*(A)$ by setting $\lambda_d(\overline{[a]}-\overline{[b]})=d(a)-d(b)$, for $a,b\in M_\infty(A)_+$.
This defines a bijection between the set $\DF(A)$ of dimension functions on $A$ and the set $\St(K_0^*(A))$ of states on $K_0^*(A)$
(see \cite[Proposition~4.3]{Cun78DimFct}, which is formulated for the case that $A$ is simple, but works in general).
Moreover, it is straightforward to verify that this bijection identifies the natural structures of $\DF(A)$ and $\St(K_0^*(A))$ as compact convex sets.

In \cite{BlaHan82DimFct}, Blackadar and Handelman conjectured that $\DF(A)$ is always a Choquet simplex.
This has been confirmed for various classes of C*-algebras:
in \cite[Corollary~4.4]{Per97StructurePositive} for C*-algebras with real rank zero and stable rank one;
in \cite[Theorem~4.1]{AntBosPerPet14GeomDimFct} for certain C*-algebras with stable rank two;
in \cite[Theorem 3.4]{Sil16arX:ConjDF} for C*-algebras with finite radius of comparison and finitely many extreme quasitraces.

In view of results obtained in \cite{AntBosPerPet14GeomDimFct}, it was asked in \cite[Problem~3.13]{AntBosPerPet14GeomDimFct} for which C*-algebras $A$ is $K_0^*(A)$  an interpolation group. 
We answer this question affirmatively for C*-algebras of stable rank one, thereby also confirming Blackadar and Handelman's conjecture for these C*-algebras. 
Recall that an interpolation group is a partially ordered abelian group $G$ such that, whenever $x_1,x_2,y_1,y_2\in G$ satisfy $x_1,x_2\leq y_1,y_2$, then there is $z\in G$ with $x_1,x_2\leq z\leq y_1,y_2$.

\begin{thm}
\label{thm:BlacHanconjecture}
Let $A$ be a unital C*-algebra of stable rank one. Then $K_0^*(A)$ is an interpolation group and $\DF(A)$ is a Choquet simplex.
\end{thm}
\begin{proof}
By \autoref{Cstarinf-semilattice}, we know that $\Cu(A)$ has the Riesz Interpolation Property.
This property passes to $W(A)$ since $W(A)$ is hereditary in $\Cu(A)$.
Indeed, if $u,v\leq x,y$ in $W(A)$, then there is $w\in\Cu(A)$ such that $u,v\leq w\leq x,y$ and, since $W(A)$ is hereditary in $\Cu(A)$, we have $w\in W(A)$. Now, the Grothendieck group of a semigroup with the Riesz interpolation property is an interpolation group (see  \cite[Lemma~4.2]{Per97StructurePositive}). Therefore  $K_0^*(A)$ is an interpolation group.
Finally, using for example \cite[Theorem~10.17]{Goo86GpsInterpolation}, we obtain that $\St(K_0^*(A))$ is a Choquet simplex, and thus so is $\DF(A)$.
\end{proof}

\section{The Global Glimm halving Problem}\label{sec:GGH}

\emph{The Global Glimm Halving Problem} has been posed in various forms;
see, for example, \cite[Definition~1.2]{BlaKir04GlimmHalving} and \cite[Question~1.2]{EllRor06Perturb}.
We recall that one formulation is as follows: If $A$ is a unital C*-algebra without finite dimensional representations, is there a *-homomorphism $\varphi\colon M_2(C_0((0,1]))\to A$ with full range?
(Recall that a subset of a C*-algebra is called full if it generates the C*-algebra as a closed, two-sided ideal.)
As mentioned in the introduction, this question was first considered, implicitly, in \cite[Section~4]{KirRor02InfNonSimpleCalgAbsOInfty}, where it was shown that if it has an affirmative answer for a  weakly purely infinite C*-algebra $A$, then $A$ is in fact purely infinite.
The Global Glimm Halving Problem is solved affirmatively in \cite{BlaKir04GlimmHalving} for C*-algebras with Hausdorff primitive spectrum of finite dimension, and in \cite{EllRor06Perturb} for all C*-algebras of real rank zero.

In \autoref{CstarGG} below we solve the Global Glimm Halving Problem affirmatively for separable C*-algebras of stable rank one, by using an equivalence obtained in \cite{RobRor13Divisibility} between  this problem and certain divisibility properties in the Cuntz semigroup. We use and improve some of these tools for the stable rank one case, and we even obtain a sharper result that characterizes when a C*-algebra of stable rank one has irreducible representations of a given finite dimension. Further, in \autoref{nonsepGG} we remove the separability assumption.
Our line of attack consists of first establishing results on divisibility of elements of \CuSgp{s}, which are subsequently translated into a solution of the Global Glimm Halving Problem.  

\begin{pgr}
\label{pgrweakdiv}
\emph{Divisibility in Cuntz semigroups}.	
Let $S$ be a \CuSgp, $x\in S$ and $k\in\NN$.
Let us recall the divisibility properties introduced in \cite[Definitions~3.1, 5.1]{RobRor13Divisibility}.
\begin{enumerate}[(i)]
	\item
	Given $n\in\NN$, we say that $x$ is \emph{$(k,n)$-divisible} if for each $x'\in S$ satisfying $x'\ll x$ there exists $y\in S$ such that $ky\leq x$ and $x'\leq ny$. 
	
	\item
	We  say that $x$ is \emph{$(k,\omega)$-divisible} if for each $x'\in S$ satisfying $x'\ll x$ there exist $y\in S$ and $n\in \NN$ such that $ky\leq x$ and $x'\leq ny$.
	
	\item
	Given $n\in \NN$, we say that $x$ is \emph{weakly $(k,n)$-divisible} if for each $x'\in S$ with $x'\ll x$ there exist $y_1,\ldots,y_n\in S$ such that $ky_j\leq x$ for all $j$ and $x'\leq \sum_{j=1}^n y_j$.  
	
	\item
	We say that $x$ is \emph{weakly $(k,\omega)$-divisible}  if for each $x'\in S$ with $x'\ll x$ there exist $n\in\NN$ and $y_1,\ldots,y_n\in S$ such that $ky_j\leq x$ for all $j$ and $x'\leq \sum_{j=1}^n y_j$. 
\end{enumerate}

Observe that in~(i) and~(ii) we can always arrange for $y$ to satisfy $ky\ll x$ and $x'\ll ny$ (rather than $ky\leq x$ and $x'\leq ny$), by first choosing $x''\in S$ such that $x'\ll x''\ll x$, then choosing $\tilde{y}$ such that $k\tilde{y}\leq x$ and $x''\leq n\tilde{y}$, and then choosing $y$ such that $y\ll\tilde{y}$ and $x'\ll ny$. Similarly, in~(iii) and~(iv) $y_1,\ldots,y_n$ can be chosen such that $ky_j\ll x$ and $x'\ll \sum_{j=1}^n y_j$ at no cost.
\end{pgr}

\begin{pgr}
\label{pgr:full}
Given a \CuSgp{} $S$, and $x\in S$.
We set $\infty x = \sup_n nx$, and we say that $x$ is \emph{full} provided that $y\leq \infty x$ for any $y\in S$.
Let $A$ be a \ca{} and $a\in A_+$. 
Then $a$ is full in $A$ if and only if $[a]$ is full in $\Cu(A)$. 
This follows for instance from the natural correspondence between closed, two-sided ideals in $A$ and ideals in $\Cu(A)$; 
see \autoref{pgr:ideal}.
In \autoref{lma:basic_cones}, we characterize fullness of $x$ in terms of the rank of $x$.
\end{pgr}

Given $x$ and $y$ in a partially ordered semigroup $S$, we say that $y$ \emph{dominates} $x$, and write $x\propto y$, if there exists $n\in\NN$ such that $x\leq ny$.

\begin{lma}
\label{wedgefull}
Let $S$ be an inf-semilattice ordered \CuSgp, and let $x,y_1,\ldots,y_n$ be elements in $S$ such that $x\propto y_k$ for $k=1,\ldots,n$. Then $x\propto \bigwedge_k y_k$. In fact, if $x\leq My_k$ for all $k$, then $x\leq N (\bigwedge_{k=1}^n y_k)$ where $N=n(M-1)+1$.
\end{lma}	
\begin{proof}
It is enough to prove the last assertion. Assume $M\in\NN$ is such that $x\leq My_k$ for $k=1,\ldots,n$.
Set $N=n(M-1)+1$. 
By \autoref{general}, we have
\[
N\bigwedge_{k=1}^n y_k=\sum_{j=1}^{N} \bigwedge_{k=1}^n y_k
=\bigwedge  \left(\sum_{k=1}^{N} y_{i_k}\right),
\] 	
where the infimum on the right hand side runs through all sums with $N$ terms taken from the set $\{y_1,\dots,y_n\}$. Since $N=n(M-1)+1$, each of these sums contains at least one of the $y_k$ repeated $M$ times, whence it is greater than or equal to $x$. Thus, $N(\bigwedge_{i=1}^n y_k)$ is greater than or equal to $x$, as desired.
\end{proof}

\begin{lma}
\label{CuGGlm}
Let $S$ be an inf-semilattice ordered \CuSgp{} satisfying \axiomO{5} and weak cancellation.
Let $k\in\NN$ and $x',x,y_1,\dots,y_n\in S$ be such that $x'\ll x$, $x'\ll \sum_{j=1}^n y_j$, and $ky_j\leq x$ for each $j$.
Then there exist $z_1,\dots,z_k\in S$ such that $\sum_{j=1}^k z_j\leq x$ and $x'\propto z_j$ for each $j$.
More precisely, we have $x'\leq M z_j$ where 
\[
M=\max\{ n^r(k-r)+n^{r-1} : r=1,\dots, k\}.
\]
\end{lma}
\begin{proof}
We will prove the result by induction over $k$.
The case $k=1$ is trivial taking $z_1=x$.
Let us assume $k>1$ and that the result holds for $k-1$.

Let $x', x, y_1,\ldots, y_n$ be as in the statement of the lemma.
Choose $y_1',\ldots,y_n'\in S$ such that $y_j'\ll y_j$ for each $j$, and such that $x'\ll \sum_{j=1}^n y_j'$.
For each $j$, choose $y_j''\in S$ such that $y_j'\ll y_j''\ll y_j$.
Apply \axiomO{5} to $(k-1)y_j'\ll(k-1)y_j''\leq x$ to obtain $w_j\in S$ such that
\[
(k-1)y_j'+w_j\leq x\leq (k-1)y_j''+w_j\,.
\]
Multiplying by $k$ in $x\leq (k-1)y_j''+w_j$ we get
\[
kx\leq  (k-1)ky_j''+kw_j\,.
\]
Since $(k-1)ky_j''\ll (k-1)x$, we get by weak cancellation that $x\leq kw_j$ (see \autoref{pgr:weakcancellation}).

Set $w=\bigwedge_{j=1}^n w_j$. Note that, since $w_j\leq x$ for all $j$, we have $w\leq x$.
By \autoref{wedgefull} we have $x\leq (n(k-1)+1)w$.
Choose $w',w''\in S$ such that $w'\ll w''\ll w$ and $x'\leq (n(k-1)+1)w'$.
Using \axiomO{5} in the inequality $w'\ll w''\leq x$, we obtain $\tilde x\in S$ such that $w'+\tilde x\leq x\leq w''+\tilde x$. For each $j$, we have
\[
(k-1)y_j'+w_j\leq x\leq \tilde x +w''\,.
\]
Since $w''\ll w_j$, we get by weak cancellation that $(k-1)y_j'\leq \tilde x$.
Hence,  $\sum_{j=1}^n y_j'\leq n \tilde x$.
Observe  also that, by \autoref{general},
\[
n \left( \big( \sum_{j=1}^n y_j' \big) \wedge \tilde x \right)
= \bigwedge_{l=0}^n \left( (n-l)\big(\sum_{j=1}^n y_j'\big)+l\tilde x\right).
\]
Further, any of the terms of the infimum on the right hand side is greater than $\sum_{j=1}^n y_j'$. Since  $x'\ll \sum_{j=1}^n y_j'$,  we have  $x'\ll n((\sum_{j=1}^n y_j')\wedge \tilde{x})$.
Choose $\tilde{x}'$ such that $\tilde x' \ll (\sum_{j=1}^n y_j')\wedge \tilde x$ and $x'\leq n\tilde x'$.
By construction, we can apply induction on $\tilde x',\tilde x, y_1',\dots,y_n'$ to find $z_1,\dots,z_{k-1}$ such that 
$\sum_{i=1}^{k-1}z_i\leq \tilde x$ and $\tilde x'\leq M_0 z_i$ for $i=1,\dots,k-1$, where 
\[
M_0=\max\{ n^{s}(k-1-s)+n^{s-1} : s=1,\dots, k-1 \}.\]
Set $z_k=w'$. We have 
\[
\sum_{j=1}^k z_j \leq \tilde x + w'\leq x.
\]
Moreover, $x'\leq n\tilde x'\leq nM_0 z_j$ for $j=1,\dots,k-1$ and $x'\leq (n(k-1)+1)z_k$.
Since $M\geq\max\{M_0n,n(k-1)+1\}$, this completes the proof of the induction step.		
\end{proof}

\begin{thm}
\label{CuGG}
Let $S$ be an inf-semilattice ordered \CuSgp{} satisfying \axiomO{5} and weak cancellation.
Let $k\in\NN$ and let $x\in S$.	Then $x$ is weakly $(k,n)$-divisible for some $n\in\NN$ (weakly $(k,\omega)$-divisible) if and only if $x$ is $(k,N)$-divisible for some $N\in\NN$ ($(k,\omega)$-divisible).
Moreover, given $n\in\NN$, the corresponding $N$ may be chosen to depend only on $k$ and $n$ (and not on $S$ or $x$).
\end{thm}	
\begin{proof}
The backward implications are clear. To show their converses, let $x'\in S$ satisfy $x'\ll x$.
By assumption, there exist $y_1,\ldots,y_n\in S$ such that $ky_j\leq x$ for all $j$, and $x'\ll \sum_{j=1}^n y_j$.
Apply \autoref{CuGGlm} to obtain $M\in\NN$ and $z_1,\ldots,z_k\in S$ such that $\sum_{j=1}^k z_j\leq x$ and $x'\leq Mz_j$ for each $j$.	Set $N=k(M-1)+1$ and $z=\bigwedge z_i$.	Then $kz\leq x$ and $x'\leq Nz$ by \autoref{wedgefull}.
\end{proof}

The following result is an improved version of \cite[Lemma~2.5]{RobRor13Divisibility} that is available for C*-algebras with stable rank one.

\begin{lma}
\label{prp:mapFromCone}
Let $A$ be a C*-algebra with stable rank one. Let $k\in\NN$, $x\in\Cu(A)$, and $b\in A_+$ satisfy $kx\leq[b]$.
Then there exists a *-homomorphism $\varphi\colon M_k(C_0((0,1])) \to \overline{bAb}$ such that $[\varphi(e_{11}\otimes\iota)]=x$.  (Here, we have identified $M_k(C_0((0,1]))$ with $M_k\otimes C_0((0,1])$, and $e_{11}\otimes \iota$ denotes the elementary tensor of the diagonal matrix unit with the identity function.) 
\end{lma}
\begin{proof}
Since $\Cu(A)\cong\Cu(A\otimes \mathcal K)$, we may assume that $A$ is stable and that $x\neq 0$.
Given $c,d\in A_+$, we write $c\approx d$ if there exists $r\in A$ with $c=r^*r$ and $rr^*=d$.
Since $A$ has stable rank one, we have $c\precsim d$ (Cuntz subequivalence) if and only if $c\approx d'\in\overline{dAd}$ for some $d'$.
The forward implication is recorded in \cite[Proposition~2.5]{CiuEllSan11LimitType1} (see also \cite[6.2]{OrtRorThi11CuOpenProj}) and only requires the assumption that $\overline{dAd}$ has stable rank one.
The converse direction holds in general: if $c=r^*r$ and $rr^*=d'\in\overline{dAd}$, then $c\sim d'\precsim d$.

Choose pairwise orthogonal elements $a_1,\ldots,a_k\in A_+$ with $[a_j]=x$ for each $j$. 
Then $\sum_{j=1}^k [a_j] = [\sum_{j=1}^k a_j] = kx \leq [b]$.
Choose $r\in A$ with $\sum_{j=1}^k a_j=r^*r$ and $rr^*\in\overline{bAb}$. Let $r=v|r|$ be the polar decomposition of $r$ in $A^{**}$.
Set $b_j=v^*a_jv$ for each $j$. Then $b_1,\ldots,b_k$ are pairwise orthogonal elements in $\overline{bAb}$ satisfying $[b_j]=[a_j]=x$ for each $j$.
Set $c_1=b_1/\|b_1\|$.
For $j=2,\ldots,k$, we use that $c_1\precsim b_j$ to choose $c_j\in\overline{b_jAb_j}$ with $c_1\approx c_j$.
Then $c_1,c_2,\ldots,c_k$ are pairwise orthogonal, pairwise equivalent (in the sense of $\approx$) elements in $\overline{bAb}$.
As noted in \cite[Remark~2.3]{RobRor13Divisibility}, we obtain a *-homomorphism $\varphi\colon M_k(C_0((0,1])) \to \overline{bAb}$ satisfying $[\varphi(e_{jj}\otimes\iota)]=c_j$ for all $j$.  In particular, $[\varphi(e_{11}\otimes\iota)]=[c_1]=[a_1]=x$.
\end{proof}

\begin{thm}
\label{CstarGG}
Let $A$ be a unital separable C*-algebra of stable rank one, and let $k\in\NN$.
Then $A$ has no nonzero representations of dimension less than $k$ if and only if there exists a *-homomorphism $\varphi\colon M_k(C_0((0,1]))\to A$ with full range.
\end{thm}
\begin{proof}	
If $\pi\colon A\to M_j(\CC)$ is a representation with $j<k$ and $\varphi\colon M_k(C_0((0,1]))\to A$ is any *-homomorphism, then $\pi\circ\varphi=0$.
Thus, $\ker(\pi)$ contains the ideal generated by the range of any such $\varphi$.
If there exists $\varphi\colon M_k(C_0((0,1]))\to A$ with full range then $\pi$ must be the zero representation.
This proves the easy direction.

Suppose now that $A$ has no nonzero representations of dimension less than $k$. Let $1\in A$ be the unit of $A$.
We have by \cite[Theorem~5.3]{RobRor13Divisibility} that $[1]$ is weakly $(k,n)$-divisible in $\Cu(A)$ for some $n\in\NN$.	Since $A$ is separable and of stable rank one, $\Cu(A)$ is an inf-semilattice ordered \CuSgp{} satisfying \axiomO{5} and weak cancellation. We thus obtain from \autoref{CuGG} that $[1]$ is $(k,N)$-divisible for some $N\in\NN$.	Hence, we can choose $x\in \Cu(A)$ such that $kx\leq [1]$ and $[1]\leq Nx$. 
By \autoref{prp:mapFromCone}, there exists a *-homomorphism $\varphi\colon M_k(C_0((0,1]))\to A$ such that $[\varphi(e_{11}\otimes \iota)]=x$.	Since $x$ is full, so is $\varphi(e_{11}\otimes \iota)$ (see \autoref{pgr:full}) and $\varphi$ has full range.
\end{proof}

\begin{rmk}
It is possible to adapt the previous proof to nonunital C*-algebras.
In this case, however, rather than a *-homomorphism with full range, we obtain for each $a\in A$ in the Pedersen ideal of $A$ a *-homomorphism $\varphi\colon M_k(C_0((0,1]))\to A$ such that the ideal generated by the range of $\varphi$ contains $a$ (assuming that $A$ has no nonzero representations of dimension less than $k$). This can be improved if we start with the assumption that $A$ has no \emph{elementary quotients}.
In this case we can get $\varphi\colon M_k(C_0((0,1]))\to A$ with full range for each $k\in\NN$, even in the nonunital case.	We prove this in \autoref{fullscattered} below. We first establish an improved form of divisibility of full elements (\autoref{smallsoft}) which will also be used in \autoref{SecSupersoft}.
\end{rmk}

\begin{pgr}
\label{dfn:soft}
Let $S$ be a \CuSgp.
Recall that $x\in S$ is said to be \emph{soft} if for all $x'\in S$ with $x'\ll x$ we have $(k+1)x'\leq kx$ for some $k\in\NN$ (see \cite[Definition~5.3.1]{AntPerThi18TensorProdCu}.) Recall that a subsemigroup $T$ of $S$ is said to be \emph{absorbing} provided that $t+s\in T$ for any $t\in T$ and $s\in S$ such that $s\leq\infty t$. By \cite[Theorem~5.3.11 (2)]{AntPerThi18TensorProdCu}, the subsemigroup of soft elements in a \CuSgp{} is absorbing.
The following result is essentially \cite[Proposition~6.4]{EllRobSan11Cone}, but we include a proof for completeness.
\begin{lma*}
	\label{lma:soft}
	Let $S$ be a \CuSgp. Let $(x_j)_j$ be a sequence in $S$ such that $x_j\propto x_{j+1}$ for each~$j$.
	Then $x=\sum_{j=1}^\infty x_j$ is soft.	
\end{lma*}
\begin{proof}
	Let $x'\in S$ satisfy $x'\ll \sum_{j=1}^\infty x_j$.
	Then there exists $n$ such that $x'\leq \sum_{j=1}^n x_j$.
	We can now find $k\in\NN$ such that $\sum_{j=1}^n x_j\leq kx_{n+1}$ and hence
	\[
	(k+1) x'
	\leq kx'+\sum_{j=1}^n x_j
	\leq kx'+kx_{n+1}
	\leq k\sum_{j=1}^{n+1}x_j
	\leq k\sum_{j=1}^\infty x_j
	= kx.\qedhere
	\]
\end{proof}	
\end{pgr}

\begin{lma}
\label{prp:pre-smallsoft}
Let $S$ be a \CuSgp{} satisfying \axiomO{5} and weak cancellation, let $y\in S$ be full and $(3,\omega)$-divisible, and let $c_1,c_2\in S$ satisfy $c_1,c_2\ll\infty$.
(Here, $\infty=\infty y$ is the largest element in $S$.) Then there exist $z,w\in S$ such that $w$ is full, and 
\[
2z + w \leq y, \quad c_1,c_2\propto z \ll \infty.
\]
\end{lma}
\begin{proof}
Choose $y'\in S$ such that $y'\ll y$ and $c_1,c_2\propto y'$. Then choose $y''\in S$ such that $y'\ll y''\ll y$.
Since $y$ is $(3,\omega)$-divisible, we obtain $\tilde{z}\in S$ such that $3\tilde{z}\ll y$ and $y''\propto\tilde{z}$, as noted at the end of \autoref{pgrweakdiv}. Choose $z\in S$ such that $z\ll\tilde{z}$ and $y'\propto z$. Applying \axiomO{5} to $2z\ll 2\tilde{z}\leq y$, we obtain $w\in S$ such that 
\[
2z + w \leq y \leq 2\tilde{z} + w.
\]
Then $c_1,c_2\propto y'\propto z \ll \tilde{z}\leq\infty$.
Further, we have
\[
y+2y = 3y \leq 6\tilde{z} + 3w, \andSep 6\tilde{z}\ll 2y,
\]
whence we get $y\leq 3w$ by weak cancellation (see \autoref{pgr:weakcancellation}).
Thus, $w$ is full.
\end{proof}

\begin{thm}
\label{smallsoft}
Let $A$ be a separable C*-algebra of stable rank one that has no elementary quotients.
Then for every full element $x\in\Cu(A)$ and every $n\in\NN$ there exists a soft full element $z\in\Cu(A)$ such that $nz\leq x$.
\end{thm}
\begin{proof}
We first establish the following claim: \emph{Every full element in $\Cu(A)$ is $(3,\omega)$-divisible.}
To prove the claim, let $w\in\Cu(A)$ be full.	Choose $a\in (A\otimes\mathcal{K})_+$ such that $w=[a]$.
Set $B=\overline{aAa}$, which is a full hereditary sub-\ca{} of $A\otimes\mathcal{K}$ (see \autoref{pgr:full}).
Since $A$ has no elementary quotients, neither does $A\otimes\mathcal{K}$.
By Brown's stabilization theorem, we have $B\otimes\mathcal{K}\cong A\otimes\mathcal{K}$.
It follows that $B$ has no elementary quotients, and in particular no finite dimensional representations.
Then $w$ is weakly $(3,\omega)$-divisible by \cite[Theorem~5.3 (iii)]{RobRor13Divisibility}.
Applying \autoref{CuGG}, this implies that $w$ is $(3,\omega)$-divisible.

Now, to prove the theorem, it suffices to consider the case $n=2$.
Let $x\in\Cu(A)$ be full. Choose a $\ll$-increasing sequence $(x_j)_j$ with supremum $x$.
Set $w_0=x$ and $z_0=0$. We inductively find $z_j,w_j\in\Cu(A)$ such that
\[
2z_j + w_j \leq w_{j-1},\quad x_j,z_{j-1}\propto z_j\ll\infty
\]
for $j\geq 1$, and such that $w_j$ is full for $j\geq 0$.

To find $z_j,w_j$ for $j\geq 1$, assume that $z_{j-1}$ and $w_{j-1}$ have been chosen.
Since $w_{j-1}$ is full, it is $(3,\omega)$-divisible by the above claim.
Applying \autoref{prp:pre-smallsoft} (with $y=w_{j-1}$, $c_1=x_j$ and $c_2=z_{j-1}$), we obtain $z_j,w_j\in S$ with the claimed properties.

Set $z=\sum_{j=1}^\infty z_j$.
For each $k\geq 1$, we have
\[
2(z_1+\ldots+z_k) 
\leq 2(z_1+\ldots+z_{k-1})+w_{k-1}
\leq 2(z_1+\ldots+z_{k-2})+w_{k-2}
\leq \ldots \leq w_0,
\]
and thus $2z\leq w_0=x$.
Further, we deduce from $z_j\propto z_{j+1}$ for all $j$ and the lemma in \autoref{lma:soft} that $z$ is soft.
For each $j$, we have $x_j\propto z_{j+1}\leq z$ and thus $x_j\leq \infty z$.
Hence, $x\leq \infty z$, and so $z$ is full.
\end{proof}

\begin{thm}
\label{fullscattered}
Let $A$ be a separable C*-algebra of stable rank one that has no elementary quotients.
Then for each $k\in\NN$ there exists a *-homomorphism $\varphi\colon M_k(C_0((0,1]))\to A$ with full range. 
\end{thm}
\begin{proof}
Let $a\in A_+$ be full, and let $k\in\NN$.
Then $x=[a]$ is full in $\Cu(A)$ (see \autoref{pgr:full}).
Using \autoref{smallsoft}, we obtain a full element $z\in\Cu(A)$ with $kz\leq x$.
By \autoref{prp:mapFromCone}, there exists a *-homomorphism $\varphi\colon M_k(C_0((0,1]))\to \overline{aAa}\subseteq A$ such that $[\varphi(e_{11}\otimes \iota)]=z$.
This *-homomorphism has full range (see \autoref{pgr:full}).
\end{proof}

\section{The cone of functionals and its dual}
\label{backgroundF}

In this section we provide basic results on the cone $F(S)$ of functionals on a \CuSgp{} $S$ and its dual $L(F(S))$.
We formulate the problem of realizing functions in $L(F(S))$ as ranks of elements in $S$, which will be tackled in \autoref{sec:realizeRank}.
The main result of this section is \autoref{hat-inf-preserving}, which shows that the natural map $S\to L(F(S))$ preserves infima.
This is used repeatedly in the following sections.

\begin{pgr}
\emph{Functionals.}
\label{pgr:functionals}
Let $S$ be a \CuSgp. 
A map $\lambda\colon S\to [0,\infty]$ is called a \emph{functional} if $\lambda$ is additive, order-preserving, $\lambda (0)=0$, and it also preserves the suprema of increasing sequences.
Let us denote as customary the set of all functionals on $S$ by $F(S)$.

A functional $\lambda$ in $F(S)$ is said to be \emph{densely finite} if every element of $S$ can be written as a supremum of an increasing sequence in $\{x\in S : \lambda(x)<\infty\}$.
This is equivalent to saying that $\lambda(x)<\infty$ whenever there exists $\tilde x\in S$ with $x\ll \tilde x$.
We denote by $F_0(S)$ the set of densely finite functionals.

The set $F(S)$ is endowed with operations of addition and scalar multiplication by nonzero, positive real numbers (both defined pointwise).
Further, $F(S)$ is equipped with a topology that, in terms of convergence, is described as follows:
Given $\lambda\in F(S)$ and a net $(\lambda_i)_{i\in I}$ in $F(S)$, we have $\lambda_i\to \lambda$ if
\[
\limsup \lambda_i(x')\leq \lambda(x)\leq \liminf \lambda_i(x)\hbox{ for all }x',x\in S \text{ such that }x'\ll x.
\]
With this topology, $F(S)$ is a compact Hausdorff space; see \cite[Theorem~4.8]{EllRobSan11Cone}.

Given a C*-algebra $A$, there is a natural bijection between $F(\Cu(A))$ and the set $\QT(A)$
of $[0,\infty]$-valued, lower semicontinuous $2$-quasitraces on $A$;
see \cite[Theorem~4.4]{EllRobSan11Cone}.
This bijection sends $\lambda\in F(\Cu(A))$ to $\tau_\lambda\colon A_+\to[0,\infty]$, given by 
\[
\tau_\lambda(a)=\displaystyle{\int_0^{\|a\|} \lambda([(a-t)_+])dt},
\]
for $a\in A_+$. Given $\tau\in\QT(A)$, the corresponding functional $\lambda_\tau\in F(\Cu(A))$ is given by $\lambda_\tau([a])=\lim_{n} \tau(a^{1/n})$, for $a\in A_+$.

The following statements are equivalent:
\begin{enumerate}
	\item
	$\lambda_\tau\in F_0(\Cu(A))$, that is, $\lambda_\tau$ is densely finite;
	\item
	$\tau$ is densely finite;
	\item
	$\tau$ is finite on $\Ped(A)_+$, the positive part of the Pedersen ideal of $A$.
\end{enumerate}
To prove this, set $D_\tau=\{a\in A_+ : \tau(a)<\infty \}$.
Since $\tau$ is order-preserving and satisfies $\tau(a+b)\leq 2\tau(a)+2\tau(b)$ for all $a,b\in A_+$ (\cite[Section~2.9]{BlaKir04PureInf}; see also \cite[Corollary~II.1.11]{BlaHan82DimFct}), we deduce that $D_\tau$ is a unitarily invariant, hereditary cone.
It follows that $\mathrm{span}(D_\tau)$ is an ideal of $A$ with $D_\tau=\mathrm{span}(D_\tau)\cap A_+$.
Using that the Pedersen ideal $\Ped(A)$ is the smallest dense ideal of $A$, we deduce that~(2) and~(3) are equivalent.

To show that~(1) implies~(3), let $a\in\Ped(A)_+$.
By properties of the Pedersen ideal, it follows that $a\leq (a_1-\varepsilon)_+ + \ldots + (a_n-\varepsilon)_+$ for some $a_1,\ldots,a_n\in A_+$ and $\varepsilon>0$.
Then
\[
\tau(a) \leq \lambda_\tau([a]) \leq \lambda_\tau([(a_1-\varepsilon)_+]) + \ldots + \lambda_\tau([(a_n-\varepsilon)_+]) < \infty.
\]

To show that~(3) implies~(1), let $x,\tilde{x}\in\Cu(A)$ with $x\ll\tilde{x}$.
Choose $b_1,\ldots,b_m\in A_+$ such that $x\ll[b_1]+\ldots+[b_m]$.
Then choose $\varepsilon>0$ such that $x\leq \sum_{j=1}^m[(b_j-\varepsilon)_+]$. 
Now it follows that $\lambda_\tau(x)<\infty$, since for every $b\in A_+$ and $\varepsilon>0$, we have $(b-\varepsilon)_+^{1/n}\leq\tfrac{2}{\varepsilon}(b-\tfrac{\varepsilon}{2})_+$ and thus
\[
\lambda_\tau( [(b-\varepsilon)_+] ) = \lim_n \tau( (b-\varepsilon)_+^{1/n} ) \leq \tau( \tfrac{2}{\varepsilon}(b-\tfrac{\varepsilon}{2})_+ ) < \infty.
\]
\end{pgr}

\begin{pgr}\emph{Extreme functionals and chisels.}
\label{extremechisel}
Let $S$ be a \CuSgp.
A densely finite functional $\lambda\in F_0(S)$ is said to be \emph{extreme} if whenever $\mu\in F(S)$ and $C\in (0,\infty)$ satisfy $\mu\leq C\lambda$, then $\mu=0$ (the zero functional) or there exists $c\in(0,\infty)$ such that $\mu=c\lambda$. 
Notice that the zero functional is extreme.

Let $\lambda\in F_0(S)$ be an extreme functional.
If $\lambda$ is not the zero functional, we define the \emph{chisel} $\sigma_\lambda$ at $\lambda$ as the function $\sigma_\lambda\colon F(S)\to [0,\infty]$ such that
\[
\sigma_\lambda(\mu)
=\begin{cases}
0, &\text{ if } \mu=0; \\
c, &\text{ if } \mu=c\lambda\hbox{ and }c\in (0,\infty); \\
\infty, &\text{ otherwise,}
\end{cases}
\]
for $\mu\in F(S)$. 
We define $\sigma_0$---the chisel at the zero functional---as the function that is zero at $0$
and $\infty$ otherwise. 
It is straightforward to check that $\sigma_\lambda$ is both linear (with respect to the cone structure in $F(S)$) and lower semicontinuous.
The notion of chisel was first introduced in \cite{Thi17arX:RksOps}. Note that we are using a slight generalization of that definition.
\end{pgr}	

\begin{pgr}\emph{Edwards' condition.}
\label{pgr:Edward}
Let $S$ be a \CuSgp{} and let $\lambda\in F(S)$.
We say that~$S$ satisfies \emph{Edwards' condition} for $\lambda$ if
\begin{align}
\label{dfn:Edwards:eq}
\inf\big\{ \lambda_1(x)+\lambda_2(y) : \lambda=\lambda_1+\lambda_2 \big\}
= \sup \big\{ \lambda(z) : z\leq x,y \big\},
\end{align}
for all $x,y\in S$;
see \cite[Definition~4.1]{AntPerRobThi19arX:Edwards}. 
By \cite[Theorem~5.3]{AntPerRobThi19arX:Edwards}, if $A$ is a C*-algebra, then $\Cu(A)$ satisfies Edwards' condition for all functionals on $\Cu(A)$.

If $\lambda$ is extreme and densely finite, and $\lambda=\lambda_1+\lambda_2$, then each of $\lambda_1$ and $\lambda_2$ is the zero functional or a scalar multiple of $\lambda$.
Using this, one can show that the left hand side of \eqref{dfn:Edwards:eq} is $\min \{ \lambda(x),\lambda(y) \}$.
Thus, $S$ satisfies Edward's condition for an extreme $\lambda\in F_0(S)$ if and only if
\[
\min \big\{ \lambda(x),\lambda(y) \big\}=\sup \big\{ \lambda(z) : z\leq x, y \big\}
\]
for all $x,y\in S$.
This form of Edwards' condition appears in \cite[Definition~4.1]{Thi17arX:RksOps}.

If $S$ is an inf-semilattice, then the right hand side of \eqref{dfn:Edwards:eq} is $\lambda(x\wedge y)$.
Hence, in this case, $S$ satisfies Edward's condition for an extreme, densely finite $\lambda$ if and only if $\min \big\{ \lambda(x),\lambda(y) \big\}=\lambda(x\wedge y)$ for all $x,y\in S$.
\end{pgr}	

\begin{pgr}\emph{Dual of $F(S)$.}
\label{dualofcone}
Let $S$ be a \CuSgp{} satisfying \axiomO{5}. We now describe the appropriate notion of dual for the cone $F(S)$.
Denote by $\mathrm{Lsc}(F(S))$ the set of functions $f\colon F(S)\to[0,\infty]$ that are additive, order-preserving, homogeneous (with respect to nonzero, positive scalars), lower semicontinuous, and satisfy $f(0)=0$. We endow $\mathrm{Lsc}(F(S))$ with pointwise order, pointwise addition, and pointwise scalar multiplication by nonzero positive scalars. Given $x\in S$, we define the function $\widehat{x}\colon F(S)\to [0,\infty]$ by evaluation, namely:
\[
\widehat{x}(\lambda)=\lambda(x),\quad \text{for } \lambda\in F(S).
\]
Then $\widehat{x}$ belongs to $\mathrm{Lsc}(F(S))$. We call $\widehat{x}$ the \emph{rank} of $x$. Further, the map $S\to\mathrm{Lsc}(F(S))$ defined by $x\mapsto\widehat{x}$ preserves addition, order, and suprema of increasing sequences.

The \emph{realification} of $S$, denoted by $S_R$, was introduced in \cite{Rob13Cone} as the smallest subsemigroup of $\mathrm{Lsc}(F(S))$ that is closed under suprema of increasing sequences and contains all elements of the form $\tfrac{1}{n}\widehat{x}$ for $x\in S$ and $n\geq 1$.
It was proved in \cite[Proposition~3.1.1]{Rob13Cone} that $S_R$ is a \CuSgp{} satisfying \axiomO{5};
see also \cite[Proposition~7.5.6]{AntPerThi18TensorProdCu}.
We remark that $S_R$ can be identified with the tensor product of \CuSgp{s} $S\otimes [0,\infty]$ as defined and studied in \cite{AntPerThi18TensorProdCu}.

Given $f,g\in\mathrm{Lsc}(F(S))$, we write $f\lhd g$ if $f\leq(1-\varepsilon)g$ for some $\varepsilon>0$ and if $f$ is continuous at each $\lambda\in F(S)$ satisfying $g(\lambda)<\infty$.
We denote by $L(F(S))$ the subsemigroup of $\mathrm{Lsc}(F(S))$ consisting of those $f\in\mathrm{Lsc}(F(S))$ that can be written as the pointwise supremum of a sequence $(f_n)_{n\in\NN}$ in $\mathrm{Lsc}(F(S))$ such that $f_n\lhd f_{n+1}$ for all $n\in\NN$.
\end{pgr}

\begin{lma}
	\label{prp:extFctl}
	Let $S$ be a \CuSgp, let $I\subseteq S$ be an ideal, and let $\lambda\colon I\to[0,\infty]$ be a functional.
	Define $\tilde{\lambda}\colon S\to[0,\infty]$ by
	\[
	\tilde{\lambda}(x)=
	\begin{cases}
	\lambda(x), &\text{ if } x\in I; \\
	\infty, &\text{ otherwise}.
	\end{cases}
	\]
	Then $\tilde{\lambda}$ is a functional on $S$.
\end{lma}
\begin{proof}
	To show that $\tilde{\lambda}$ is order-preserving, let $x\leq y$ in $S$. If $y\notin I$, then $\tilde{\lambda}(y)=\infty$, and clearly $\tilde{\lambda}(x)\leq\tilde{\lambda}(y)$. If on the other hand $y\in I$, then $x\in I$ as well, since $I$ is an ideal of $S$, and thus $\tilde{\lambda}(x)=\lambda(x)\leq\lambda(y)=\tilde{\lambda}(y)$.
	
	To prove additivity, let $x,y\in S$. Observe that $x+y\in I$ if and only if both $x,y\in I$. If $x,y\in I$, then 
	\[
	\tilde{\lambda}(x+y)
	=\lambda(x+y)
	=\lambda(x)+\lambda(y)
	=\tilde{\lambda}(x)+\tilde{\lambda}(y).
	\] 
	On the other hand, if either $x\notin I$ or $y\notin I$, then $x+y\notin I$, and so $\tilde{\lambda}(x+y)=\infty=\tilde{\lambda}(x)+\tilde{\lambda}(y)$.
	Similarly, one proves that $\tilde{\lambda}$ preserves suprema of increasing sequences.
\end{proof}

The result below is known. We highlight it here for future reference as it will be used frequently. Recall from \autoref{pgr:full} that an element $x$ in a $\CatCu$-semigroup $S$ is said to be \emph{full} if $y\leq \infty x$ for all $y\in S$.
\begin{lma}
\label{lma:basic_cones} Let $S$ be a $\CatCu$-semigroup satisfying \axiomO{5}.
\begin{enumerate}[{\rm (i)}]
	\item We have $S_R=L(F(S))$. Thus $L(F(S))$ is a $\CatCu$-semigroup and $\widehat{x}\in L(F(S))$ for every $x\in S$.
	\item If $x,y\in S$ and $s,t\in (0,\infty]$ satisfy $x\ll y$ and $s<t$, then $s\widehat{x}\ll t\widehat{y}$. 
	\item For $x,y\in S$, we have that $x\leq\infty y$ if and only if $\widehat{x}\leq\infty\widehat{y}$. In particular, $x$ is full in $S$ if and only if $\widehat{x}$ is full in $L(F(S))$.
\end{enumerate}
\end{lma}
\begin{proof}
(i) is a consequence of \cite[Theorem~3.2.1]{Rob13Cone}, and (ii) is exactly \cite[Lemma~2.2.5]{Rob13Cone}.

Let us prove (iii). Suppose that $\widehat{x}\leq\infty\widehat{y}$. Define $\lambda\colon S\to [0,\infty]$
by $\lambda(z)=0$ if $z\leq \infty y$ and $\lambda(z)=\infty$ otherwise. Then $\lambda\in F(S)$ (by \autoref{prp:extFctl}),
and $\widehat x(\lambda)\leq \infty\widehat y(\lambda)=0$. Hence, $x\leq \infty y$.
\end{proof}

The remarks above apply to Cuntz semigroups of C*-algebras.
Given a C*-algebra~$A$, we have a natural map
\begin{align*}
\Cu(A) &\to \Cu(A)_R=L(F(\Cu(A))),
\end{align*}
given by $[a] \mapsto \widehat{[a]}$, for $a\in (A\otimes\mathcal{K})_+$, where $\widehat{[a]}(\lambda)=\lambda([a])$ for all $\lambda\in F(\Cu(A))$.

\begin{pgr}\emph{The problem of realizing functions as ranks.}
\label{probrealizing} 
Let $S$ be a \CuSgp{} satisfying \axiomO{5}.
Recall that the function $\widehat{x}\in L(F(S))$ is called the rank of $x\in S$.
The problem of realizing functions on $F(S)$ as ranks of elements in $S$ consists of finding necessary and sufficient conditions for the map $x\mapsto \widehat{x}$ to be a surjection from $S$ to $L(F(S))$.
In \autoref{thm:realizingprob1} we solve this problem when $S$ is the Cuntz semigroup of a separable C*-algebra of stable rank one.  
\end{pgr}

\begin{pgr}\emph{The problem of realizing full functions as ranks.}
\label{probrealizingfull}
Let $S$ be a \CuSgp{} satisfying \axiomO{5} and let $f\in L(F(S))$.

Let us see that $f$ is full if and only if $f(\lambda)=0$ implies $\lambda=0$ for $\lambda\in F(S)$, that is, if $f$ is strictly positive on the nonzero functionals.
Indeed, if $f$ is strictly positive on the nonzero functionals, then for any $g\in L(F(S))$ we have $g\leq \infty f$, and thus $f$ is full. Conversely, suppose that $f$ is full and let $\lambda\in F(S)$ be nonzero.
Then there are $x,y\in S$ with $x\ll y$ and such that $\widehat{x}(\lambda)\neq 0$.
Since $f$ is full and $\widehat{x}\ll 2\widehat{y}$ (see \autoref{lma:basic_cones} (ii)) we have $\widehat{x}\leq nf$ for some $n\in\NN$ and thus $f(\lambda)\neq 0$.

A variation on the problem of realizing functions on $F(S)$ as ranks is as follows: Under what conditions is the map $x\mapsto \widehat{x}$ a surjection from the subsemigroup of full elements of $S$ to the subsemigroup of full elements  of $L(F(S))$? 
In Theorems~\ref{mainrealizing} and~\ref{thm:realizingprob2} we address this problem when $S$ is the Cuntz semigroup of a separable C*-algebra of stable rank one.

Assume that $S$ contains a full, compact element $u$.
In this case, the subsemigroup of full elements of $L(F(S))$ admits a somewhat more concrete description, which we now give. 
Let $F_u(S)$ denote the set of functionals normalized at $u$, that is, the set of $\lambda\in F(S)$ such that $\lambda(u)=1$.
Then $F_u(S)$ is a compact, convex set.
Let $\LAff(F_u(S))_{++}^\sigma$ denote the set of affine functions $f\colon F_u(S)\to (0,\infty]$ such that 
$f^{-1}((t,\infty])$ is open and $\sigma$-compact for all $t\in\mathbb{R}$.  
\end{pgr}

\begin{prp}
\label{surjectiveres}
Let $S$ be a \CuSgp{} satisfying \axiomO{5} and let $u\in S$ be a full, compact element.
Then the restriction map $f\mapsto f|_{F_u(S)}$ is a bijection from the set of full functions in $L(F(S))$ to $\LAff(F_u(S))_{++}^\sigma$.
\end{prp}	
\begin{proof}
\emph{Well-definedness:}
Let $f\in L(F(S))$ be full. 
As pointed out in \autoref{probrealizingfull}, $f$ is nonzero on $F_u(S)$. 
Hence, the range of $f|_{F_u(S)}$ is contained in $(0,\infty]$. 
Let $t\in(0,\infty)$.
By the lower semicontinuity of $f$, we get that $f^{-1}((t,\infty])\cap F_u(S)$ is open in $F_u(S)$.
To see that this set is also $\sigma$-compact, choose a $\lhd$-increasing sequence $(f_n)_n$ in $L(F(S))$ with $f=\sup_n f_n$. 
Given $n$, we have $f_n\ll f_{n+1}$ in $L(F(S))$ by \cite[Proposition~5.1]{EllRobSan11Cone}, and since $\infty\widehat{u}$ is the largest element of $L(F(S))$, we conclude that $f_n\propto\widehat{u}$.
As $\widehat{u}|_{F_u(S)}\equiv 1$, we get that $f_n$ is finite on $F_u(S)$.
Then $f_{n+1}$ is also finite on $F_u(S)$, and since $f_n$ is continuous at functionals where $f_{n+1}$ is finite, $f_n$ is continuous on $F_u(S)$.
Hence, the sets on the right hand side of
\[
f^{-1}((t,\infty])\cap F_u(S)=\bigcup_{n,m=1}^\infty (f^{-1}_n([t+\frac 1 m,\infty])\cap F_u(S)),
\] 
are compact. 
It follows that the restriction $f|_{F_u(S)}$ belongs to $\LAff(F_u(S))_{++}^\sigma$.

\emph{Injectivity:}
Let $f,g\in L(F(S))$ be full such that $f|_{F_u(S)}=g|_{F_u(S)}$.
Given $\lambda\in F(S)$, we need to verify $f(\lambda)=g(\lambda)$.
This is clear if $\lambda(u)=0$ (since then $\lambda$ is the zero functional) or $\lambda(u)\in(0,\infty)$ (since then $\tfrac{\lambda}{\lambda(u)}\in F_u(S)$).
So assume $\lambda(u)=\infty$.
The fullness of $f$ implies that $\infty f$ is the largest element in $L(F(S))$.
Since $u$ is compact, we have $\widehat{u}\ll 2\widehat{u}$ in $L(F(S))$ by \autoref{lma:basic_cones} (ii).
Hence, $\widehat{u}\leq Mf$ for some $M\in\NN$, and so $f(\lambda)=\infty$.
Analogously, we have $g(\lambda)=\infty$.

\emph{Surjectivity:} 
Let $g\in \LAff(F_u(S))_{++}^\sigma$.
By \cite[Corollary~I.1.4]{Alf71CpctCvxSets}, there exists an increasing net of affine, continuous functions $g_i\colon F_u(S)\to (0,\infty]$ with supremum~$g$.
Exploiting the $\sigma$-compactness of the sets $g^{-1}((t,\infty])$, we can choose from this net an increasing sequence $(g_n)_n$ with supremum $g$;
see \cite[Lemma~4.2]{TikTom15CuSgpNonunital}.
Next, multiplying if necessary the functions $g_n$ by scalars, we can arrange for $g_n\leq (1-\varepsilon_n)g_{n+1}$ for some $\varepsilon_n>0$ and all $n$, while maintaining that $g=\sup_n g_n$.
For each $n$ define $\tilde g_n\colon F(S)\to [0,\infty]$ by
\begin{equation*}
\label{extensionf} 
\tilde{g}_n(\lambda)
= \begin{cases} 
\lambda(u)g_n(\lambda(u)^{-1}\lambda),  &\text{if } 0<\lambda(u)<\infty \\
0, &\text{if } \lambda=0 \\
\infty, &\text{otherwise.} \end{cases} 
\end{equation*}
Then $\tilde g_n\lhd \tilde g_{n+1}$ for all $n$.
Hence, $\tilde g=\sup_n \tilde g_n$ belongs to $L(F(S))$.
We have $\tilde g|_{F_u(S)}=g$, proving the desired surjectivity. 
\end{proof}

\begin{lma}
\label{multipleInf}
Let $S$ be an inf-semilattice ordered \CuSgp, let $x,y\in S$, and let $n\in\NN$.
Then $(n(x\wedge y))^{\wedge} = (nx\wedge ny)^{\wedge}$ in $L(F(S))$.
\end{lma}
\begin{proof}
We first establish the case $n=2$.
Using \autoref{general} we have
\[
3(x\wedge y)=(3x)\wedge (2x+y)\wedge (2y+x)\wedge (3y).
\]
Similarly,
\[
(2x\wedge 2y)+(x\wedge y)=(3x)\wedge (2x+y)\wedge (2y+x)\wedge (3y).
\]
This proves that $3(x\wedge y)=(2x\wedge 2y)+(x\wedge y)$. 
Hence,
\[
(2(x\wedge y))^{\wedge} +(x\wedge y)^{\wedge} = (2x\wedge 2y)^{\wedge} + (x\wedge y)^{\wedge}.
\]
Since $(x\wedge y)^\wedge$ is dominated both by $(2(x\wedge y))^{\wedge}$ and by $(2x\wedge 2y)^{\wedge}$, we can cancel it in the equality above to obtain $(2(x\wedge y))^{\wedge}=(2x\wedge 2y)^{\wedge}$. 

Applying the case $n=2$ repeatedly, we arrive at $(2^k(x\wedge y))^{\wedge}=(2^kx\wedge 2^ky)^\wedge$ for all $k\in \NN$. To complete the proof of the lemma, it now suffices to show that if the desired equality is true for some $n+1\geq 2$, then it is also true for $n$. 
We have
\begin{align*}
(nx\wedge ny) + (x\wedge y)  &= (n+1)x\wedge(nx+y)\wedge(ny+x)\wedge(n+1)y\\
&\leq (n+1)x\wedge (n+1)y.
\end{align*}
Hence, using that the desired equality is true for $n+1$, we get
\[
(nx\wedge ny)^{\wedge} + (x\wedge y)^{\wedge} \leq ((n+1)x\wedge (n+1)y)^\wedge=(n+1)(x\wedge y)^{\wedge}.
\]
As before, we can cancel $(x\wedge y)^{\wedge}$ to conclude that $(nx\wedge ny)^{\wedge} \leq (n(x\wedge y))^{\wedge}$.
The converse inequality is clear.
\end{proof}

\begin{pgr}
It is not always the case that $(2x)\wedge (2y)=2(x\wedge y)$ for all $x,y$ in the Cuntz semigroup of a separable C*-algebra of stable rank one.
Take for example a separable, unital C*-algebra $A$ of stable rank one and with 2-torsion in $K_0(A)$, that is, such that $2g=0$ for some nonzero $g\in K_0(A)$.
Say $g=[p]-[q]$ for projections $p,q\in M_\infty(A)$.
Then $e=[p]$ and $f=[g]$ are compact elements in $\Cu(A)$ such that $2e=2f$ but $e\neq f$. 
We have $(2e)\wedge (2f)\neq 2(e\wedge f)$.
Indeed, suppose for a contradiction that $(2e)\wedge (2f)=2(e\wedge f)$.
Then 
\[
2e
= (2e)\wedge (2f)
= 2(e\wedge f)
= (e\wedge f)+(e\wedge f)
\leq e+f. 
\]
By cancellation of compact elements, we obtain $e\leq f$, and a symmetrical argument proves $f\leq e$, which is impossible.
\end{pgr}	

Let $S$ be a countably based, inf-semilattice ordered \CuSgp{} satisfying \axiomO{5}.
By \autoref{automaticO6}, $S$ satisfies \axiomO{6+} and hence the weaker axiom \axiomO{6} introduced in \cite{Rob13Cone}.
Thus $L(F(S))$ is an inf-semilattice ordered \CuSgp, by \cite[Theorem~4.2.2]{Rob13Cone}.
In the proof below we  use \cite[Proposition~2.2.6]{Rob13Cone}, which asserts that $x,y\in S$ satisfy $\widehat{x}\leq\widehat{y}$ if and only if for every $x'\in S$ with $x'\ll x$ and every $\varepsilon>0$ there exist $M,N\in\NN$ such that $\tfrac{M}{N}>1-\varepsilon$ and $Mx'\leq Ny$.

\begin{thm}
\label{hat-inf-preserving}
Let $S$ be a countably based, inf-semilattice ordered \CuSgp{} satisfying \axiomO{5}.
Then the map $S\to L(F(S))$, given by $x\mapsto \widehat{x}$, preserves infima.
\end{thm}
\begin{proof}
Let $x,y\in S$.
The inequality $ \widehat{x}\wedge\widehat{y}\geq\widehat{x\wedge y}$ is straightforward.

\emph{Claim: Let $z\in S$ and $n\geq 1$ satisfy $\tfrac{1}{n}\widehat{z}\leq\widehat{x},\widehat{y}$.
	Then $\tfrac{1}{n}\widehat{z}\leq\widehat{x\wedge y}$.}
To prove the claim, let $z'\in S$ satisfy $z'\ll z$, and let $\varepsilon>0$.
Since $\widehat{z}\leq \widehat{nx}$, we can apply \cite[Proposition~2.2.6]{Rob13Cone} to obtain $M_1, N_1\in\NN$ such that $\tfrac{M_1}{N_1}>1-\varepsilon$ and $M_1z'\leq N_1 nx$.
Similarly, there exist $M_2, N_2\in\NN$ such that $\tfrac{M_2}{N_2}>1-\varepsilon$ and $M_2z'\leq N_2ny$.
Then $M_1N_2z'\leq N_1N_2 nx$ and $M_2N_1 z'\leq N_1N_2ny$, and we get
\[
\min\{M_1N_2, M_2N_1\}z' \leq (N_1N_2 nx)\wedge (N_1N_2 ny).
\]
Passing to $L(F(S))$ and using \autoref{multipleInf}, we obtain that
\[
\min\{M_1N_2,M_2N_1\}\widehat{z'}
\leq ((N_1N_2 nx)\wedge (N_1N_2 ny))^{\wedge}
= N_1N_2n(\widehat{x\wedge y}),
\]
and since 
\[
(1-\varepsilon)\leq \frac{\min\{M_1N_2,M_2N_1\}}{N_1N_2},
\]
we get $(1-\varepsilon)\frac{\widehat{z'}}{n}\leq \widehat{x\wedge y}$.
The claim follows using that $\widehat{z}=\sup_{z'\ll z}\sup_{\varepsilon>0} (1-\varepsilon)\widehat{z'}$.

To prove the theorem, we use that $L(F(S))=S_R$ as noted in \autoref{lma:basic_cones} (i), which allows us to choose a sequence $(z_k)_k$ in $S$ and a sequence $(n_k)_k$ of positive integers such that $(\tfrac{\widehat{z_k}}{n_k})_k$ is increasing with supremum $\widehat{x}\wedge\widehat{y}$.
For each $k$, we have $\tfrac{\widehat{z_k}}{n_k}\leq\widehat{x},\widehat{y}$ and thus $\tfrac{\widehat{z_k}}{n_k}\leq\widehat{x\wedge y}$ by the claim.
Hence, $\widehat{x}\wedge\widehat{y}=\sup_k \tfrac{\widehat{z_k}}{n_k}\leq\widehat{x\wedge y}$.
\end{proof}

\section{Realizing functions as ranks}
\label{sec:realizeRank}

In this section we solve the problems of realizing (full) functions on the cone $F(\Cu(A))$ as ranks of Cuntz semigroup elements when $A$ is a C*-algebra of stable rank one. These results are inspired by the ideas in \cite[Section 8]{Thi17arX:RksOps}, and in particular, some of the sets and maps defined here generalize similar ones in \cite{Thi17arX:RksOps} to the non simple and non unital case.

By an \emph {ideal-quotient} of a C*-algebra $A$ we mean a quotient of the form $I/J$, where $J\subseteq I$ are closed-two sided ideals of $A$.
Ideal-quotients thus arise as ideals of the quotients of $A$ or as quotients of its ideals.

\begin{prp}
Let $A$ be a C*-algebra.
Then the following statements hold:
\begin{enumerate}[{\rm (i)}]
	\item
	If $A$ has a nonzero, elementary ideal-quotient, then there exists $\lambda \in F(\Cu(A))$ with
	\[
	\big\{ \widehat{x}(\lambda) : x\in\Cu(A)\}=\{0,1,\dots,\infty \big\}.
	\]	
	\item
	If $A$ is separable and has a nonzero, elementary quotient, then there exists a densely finite $\lambda\in F(\Cu(A))$ such that 
	\[
	\big\{ \widehat{x}(\lambda) : x\in\Cu(A) \hbox{ and  $x$ is full}\}=\{1,\dots,\infty \big\}.
	\]	
\end{enumerate}
\end{prp}	
\begin{proof}
(i):
Assume that $I$ and $J$ are closed, two-sided ideals such that $J\subseteq I$ and $I/J$ is elementary.
Then $\Cu(I/J)\cong\NNbar$ and thus the quotient map $I\stackrel{\pi}{\to} I/J$ induces a surjective \CuMor{} $\Cu(\pi)\colon\Cu(I)\to \Cu(I/J)\cong\NNbar$. Now let $\lambda\colon \Cu(A)\to [0,\infty]$ be given by $\lambda(x)=\Cu(\pi)(x)$ if $x\in\Cu(I)$ and $\lambda(x)=\infty$ otherwise.
Since the range of $\Cu(\pi)$ is $\NNbar$, it is clear from the definition of $\lambda$ that $\{\widehat{x}(\lambda)\colon x\in\Cu(A)\}=\lambda(\Cu(I))=\NNbar$. Further, $\lambda$ is a functional by \autoref{prp:extFctl}.

(ii):
Let $I$ be a closed, two-sided ideal such that $A/I$ is elementary.
Let $\lambda\in F(\Cu(A))$ be the functional obtained in (i), that is, $\lambda=\Cu(\pi)$, where $\pi\colon A\to A/I$.
If $x\in \Cu(A)$ is full then $\widehat x(\lambda)\neq 0$, so that $\widehat x(\lambda)\in \{1,2,\ldots,\infty\}$.
To complete the proof it suffices to show that there exists a full $x$ such that $\lambda(x)=1$.  
Since $\lambda$ is onto, there exists $x_0\in \Cu(A)$ such that $\lambda(x_0)=1$.
Let $\omega_I\in\Cu(I)$ be the largest element of $\Cu(I)$, which exists since $I$ is separable.
Set $x=x_0+\omega_I$.
Clearly, $\lambda(x)=\lambda(x_0)=1$.
Moreover, $x$ is full, for if $y\in \Cu(A)$, then $\lambda(y)\leq \infty\in\NNbar$, from which we deduce that $y\leq \infty x_0+\omega_I=\infty x$.
\end{proof}

In view of the previous proposition, it is clear that in order to realize every element of $L(F(\Cu(A)))$ in the form $\widehat{x}$, with $x\in \Cu(A)$, we must assume that $A$ has no nonzero, elementary ideal-quotients.
Similarly, if $A$ is unital, and $F_u(\Cu(A))$ is the set of functionals normalized at $[1_A]$, then in order to realize elements of $\LAff(F_u(\Cu(A)))_{++}$  in the form $\widehat{x}|_{F_u(\Cu(A))}$ with $x\in\Cu(A)$ full, we must assume that $A$ has no nonzero, finite dimensional representations.
As we show below, if $A$ has stable rank one, then these are the only obstructions.

In the proof of the following theorem we borrow ideas from the closely related \cite[Lemma~8.3]{Thi17arX:RksOps}.

\begin{thm}
\label{thm:upwardalpha}
Let $A$ be a separable C*-algebra of stable rank one, and let $f\in L(F(\Cu(A)))$. 
Then the set
\[
I_f = \big\{ x\in \Cu(A) : \widehat{x'} \ll f \hbox{ for all } x'\ll x\big\}
\]
has a supremum. 
\end{thm}	
\begin{proof}
Since $A$ is separable, $\Cu(A)$ is countably based.
Thus, as noted in \autoref{countable}, it suffices to show that $I_f$ is upward directed.
Clearly $I_f$ is order-hereditary.
It is also closed under the suprema of increasing sequences.
For suppose that $x=\sup_n x_n$, where $(x_n)_n$ is an increasing sequence in $I_f$.
Let $x'\ll x$.
Then $x'\ll x_n$ for some $n$, and so $\widehat{x'}\ll f$ by the definition of $I_f$.
This shows that $x\in I_f$. By \autoref{upwardlemma}, in order to show that $I_f$ is upward directed it suffices to show that the set $G_f=\{x'\in\Cu(A)\colon \text{ there is }x\in I_f\text{ with } x'\ll x\}$ is upward directed. 
We prove this first below. 
We remark that $G_f$ can be alternatively described as follows:
\[
G_f = \big\{ x\in \Cu(A) : \text{there exists }y\in \Cu(A)\text{ such that }x\ll y\text{ and }\widehat{y}\ll f \big\}.
\]
In order to see this, let $x\in G_f$.
Then there exist $y',y$ such that $x\ll y'\ll y$ and $y\in I_f$.
Then $\widehat{y'}\ll f$, and thus $x$ belongs to the right hand side of the equality above. 
Conversely, if $x$ is such that $x\ll y$ and $\widehat{y}\ll f$ for some $y$, then clearly
$y\in I_f$ and therefore $x\in G_f$. 

We now prove that $G_f$ is upward directed.
Let $x_1,x_2\in G_f$.
Choose elements $y_1,y_1',y_2,y_2'\in \Cu(A)$ such that
\[
x_1\ll y_1'\ll y_1,\quad
x_2\ll y_2'\ll y_2,\andSep \widehat{y_1},\widehat{y_2}\ll f.
\]
Choose $f'',f'\in L(F(\Cu(A)))$ with $\widehat{y_1},\widehat{y_2}\ll f''\ll f'\ll f$.
Since $L(F(\Cu(A)))=\Cu(A)_R$ by \autoref{lma:basic_cones} (i), we can choose a sequence $(d_n)_n$ in $\Cu(A)$ and a sequence $(k_n)_n$ of positive integers such that $(\tfrac{\widehat{d_n}}{k_n})_n$ is increasing with supremum $f'$.
Then there is $n_0$ such that $f''\leq \tfrac{\widehat{d_{n_0}}}{k_{n_0}}$.
Set $d=d_{n_0}$ and $k=k_{n_0}$.
Then
\[
\widehat{y_1},\widehat{y_2}\ll \frac{\widehat{d}}{k}\ll f.
\]

Let us construct $w\in\Cu(A)$ such that $x_1,x_2\leq w$ and $\widehat{w}\leq \frac{\widehat{d}}{k}$. 
(We will afterwards arrange for a $w\in G_f$.) 
Observe that  $y_1,y_2\leq \infty d$, by \autoref{lma:basic_cones} (iii). 
Hence, there exists $n\in\NN$ such that $y_1',y_2'\leq nd$.
We apply the construction from \autoref{prp:unitization} to $A$ and $d\in\Cu(A)$ to obtain a C*-algebra $B$ of stable rank one and a projection $p_d\in B$ such that $A$ is an ideal of $B$, and such that for any $x\in\Cu(A)$ we have $x\leq d$ precisely when $x\leq [p_d]$ in $\Cu(B)$.
Set $e=[p_d]$, which is a compact element in $\Cu(B)$. 

Then $y_1',y_2'\leq nd \leq n[p_d]=ne$.
By \axiomO{5} applied to $x_i\ll y_i'\leq ne$ for $i=1,2$, we obtain $z_1,z_2\in\Cu(B)$ such that
\begin{align*}
x_1+z_1 &\leq ne\leq y_1'+z_1,\\
x_2+z_2 &\leq ne\leq y_2'+z_2.
\end{align*}
Set $z=z_1\wedge z_2$. 
Note that $z\leq ne$.
Let $\varepsilon_0>0$ be such that $\widehat{y_1}, \widehat{y_2}\leq \frac{1-\varepsilon_0}{k}\widehat{d}$ and set	$g=\frac{1-\varepsilon_0}{k}\widehat{d}$.
Next, choose $0<\varepsilon<\varepsilon_0$ such that $\varepsilon \widehat{z}\leq (\varepsilon_0-\varepsilon)\frac{\widehat{e}}{k}$.
(Such an $\varepsilon$ exists since $z\leq ne$.) 
Then
\[
(1+\varepsilon)g+\varepsilon \widehat{z} 
\leq (1+\varepsilon)(1-\varepsilon_0)\frac{\widehat{e}}{k}+ (\varepsilon_0-\varepsilon)\frac{\widehat{e}}{k}
\leq \frac{\widehat{e}}{k}.
\]

We have $n\widehat{e}\leq \widehat{y_1'}+\widehat{z_1}\leq g+\widehat{z_1}$ and similarly $n\widehat{e}\leq g+\widehat{z_2}$.
Using at the first step that $L(F(\Cu(B)))$ is an inf-semilattice ordered \CuSgp{} (by \cite[Theorem~4.2.2]{Rob13Cone}) and using \autoref{hat-inf-preserving} at the second step, we obtain
\[
n\widehat{e}
\leq g+(\widehat{z_1}\wedge \widehat{z_2})
=g+\widehat{z}.
\]
Next, since $e\ll e$, it follows from \autoref{lma:basic_cones} (ii) that
\[
n\widehat{e}\ll (1+\varepsilon)g+(1+\varepsilon)\widehat{z}.
\]
Choose $z'\in\Cu(B)$ with $z'\ll z$ and $n\widehat{e}\ll (1+\varepsilon)g+(1+\varepsilon)\widehat{z'}$.
Applying \axiomO{5} to $z'\ll z\leq ne$, find $w'\in\Cu(B)$ such that $z'+w'\leq ne\leq z+w'$.
Then
\[
x_1+z\leq x_1+z_1\leq n e\ll ne\leq z+w'.
\]
Recall that $B$ has stable rank one, and thus $\Cu(B)$ has weak cancellation.
Therefore, we have $x_1\leq w'$, and similarly, $x_2\leq w'$.
On the other hand, 
\[
\widehat{z'}+\widehat{w'}\leq n\widehat{e}\ll (1+\varepsilon)g+(1+\varepsilon)\widehat{z'}.
\] 
Therefore $\widehat{w'}\leq (1+\varepsilon)g+\varepsilon \widehat{z}\leq \tfrac{1}{k}\widehat{e}$.

Let $\omega_A\in\Cu(A)$ be the largest element of $\Cu(A)$.
Set $w=w'\wedge\omega_A$, which belongs to $\Cu(A)$ since the inclusion $A\to B$ identifies $\Cu(A)$ with an ideal in $\Cu(B)$.
Using that $x_1,x_2\leq w'$, we get $x_1,x_2\leq w$.
Applying \autoref{hat-inf-preserving} at the first step and last step, and using that $\tfrac{1}{k}\widehat{\omega_A}=\widehat{\omega_A}$ at the third step, we obtain
\[
\widehat{w}=\widehat{w'} \wedge \widehat{\omega_A}
\leq \frac{\widehat{e}}{k} \wedge \widehat{\omega_A}
=\frac{1}{k}(\widehat{e}\wedge \widehat{\omega_A}) 
=\frac{1}{k}(\widehat{e\wedge \omega_A}).
\]
But $e\wedge \omega_A=d$, by \autoref{prp:infSupportProj}.
Thus, $w$ is as desired.

Finally, let us explain how to arrange for a $w\in G_f$: Start by choosing $x_1\ll x_1'\ll y_1'$ and   $x_2\ll x_2'\ll y_1'$. Apply the argument above to obtain $w_1$ such that $x_1',x_2'\leq w_1$ and $\widehat{w_1}\leq \frac{\widehat{d}}{k}$.
Next, choose $w\ll w_1$ such that $x_1,x_2\leq w$.
Then $w\in G_f$ and $x_1,x_2\leq w$, thus showing that $G_f$ is upward directed.
\end{proof}

\begin{dfn}
\label{dfn:mapAlpha}
Let $A$ be a separable C*-algebra of stable rank one.
In view of \autoref{thm:upwardalpha}, we define $\alpha\colon L(F(\Cu(A)))\to\Cu(A)$ by $\alpha (f)=\sup I_f$, where
\[
I_f = \big\{ x\in\Cu(A) : \widehat{x'}\ll f\text{ for all }x'\ll x \big\},
\]
for $f\in L(F(\Cu(A)))$.
\end{dfn}

\begin{prp}
\label{prp:mapAlpha}
Let $A$ be a separable C*-algebra of stable rank one. The following hold:
\begin{enumerate}[{\rm (i)}] 
	\item
	For all $f\in L(F(\Cu(A)))$ we have $\widehat{\alpha(f)}\leq f$.
	\item
	For all $f\in L(F(\Cu(A)))$ we have
	\begin{align*}
	\alpha(f) 
	&=\sup \big\{ x\in\Cu(A) : x\ll y \hbox{ and }\widehat{y}\ll f\hbox{ for some }y \big\} \\
	&=\sup \big\{ x\in\Cu(A) : \widehat{x}\leq (1-\varepsilon)f\hbox{ for some }\varepsilon>0 \big\} \\
	&=\sup \big\{ x\in\Cu(A) : \widehat{x}\ll f \big\}.
	\end{align*}
	\item
	The map $\alpha$ preserves the order, the suprema of increasing sequences, and the infima of pairs of elements.	
	\item
	The map $\alpha$ is superadditive, that is, $\alpha(f)+\alpha(g)\leq\alpha(f+g)$ for all $f,g\in L(F(\Cu(A)))$.
\end{enumerate} 
\end{prp}
\begin{proof}
(i): We have $\widehat{\alpha(f)}\leq f$, since $\alpha(f)=\sup I_f$ and each element $x$ in $I_f$ satisfies $\widehat{x}\leq f$. 

(ii): Set
\[
G_f = \big\{ x\in \Cu(A) : \text{there exists }y\in \Cu(A)\text{ such that }x\ll y\text{ and }\widehat{y}\ll f \big\}.
\]
In the course of the proof of \autoref{thm:upwardalpha} we have shown that $\sup I_f = \sup G_f$, which proves the first displayed equality.
The other displayed equalities follow from the following implications, which we show to hold for every $x\in\Cu(A)$:
\[
x\in G_f 
\quad\Rightarrow\quad \widehat{x}\ll f 
\quad\Rightarrow\quad \widehat{x}\leq(1-\varepsilon)f\hbox{ for some }\varepsilon>0
\quad\Rightarrow\quad x\in I_f.
\]
The first and second implications are clear.
To see the third implication, suppose that $x$ satisfies $\widehat{x}\leq (1-\varepsilon)f$ for some $\varepsilon>0$.
Then $\widehat{x'}\ll f$ for every $x'\ll x$, by \autoref{lma:basic_cones} (ii), and thus $x\in I_f$.

(iii):	\emph{Preserving the order:}
Let $f,g\in L(F(\Cu(A)))$ satisfy $f\leq g$.
Then $I_f\subseteq I_g$, and thus $\alpha(f)\leq \alpha(g)$.

\emph{Preserving suprema of increasing sequences:}
Let $(f_n)_n$ be an increasing sequence of elements in $L(F(\Cu(A)))$, and set $f =\sup_n f_n$.
Since $\alpha$ is order-preserving, the sequence $(\alpha(f_n))_n$ is increasing in $\Cu(A)$.
Set $x =\sup_n \alpha(f_n)$.
Since $\alpha(f_n)\leq \alpha(f)$ for all $n$, we have $x\leq \alpha(f)$.
To prove the converse inequality, let $z\in\Cu(A)$ satisfy $\widehat{z}\ll f$.
Since $\alpha(f)$ is the supremum of all such $z$, it suffices to show that $z\leq x$.
Since $\widehat{z}\ll f$, there is $n\in\NN$ with $\widehat{z}\ll f_n$.
This means that $z\in I_{f_n}$ and thus $z\leq \alpha(f_n)\leq x$.

\emph{Preserving infima:}
Given $f,g\in L(F(\Cu(A)))$, let us show that $\alpha(f\wedge g)=\alpha(f)\wedge \alpha(g)$.
From the fact that $\alpha$ is order preserving we deduce at once that $\alpha(f\wedge g)\leq \alpha(f)\wedge \alpha(g)$.
Let $0<\varepsilon<1$ and suppose that $z\leq \alpha((1-\varepsilon)f)\wedge \alpha((1-\varepsilon)g)$.
Then 
\[
\widehat{z}\leq (1-\varepsilon)f\wedge (1-\varepsilon)g=(1-\varepsilon)(f\wedge g).
\] 
Hence, $z\leq \alpha(f\wedge g)$. 
It follows that $\alpha((1-\varepsilon)f)\wedge \alpha((1-\varepsilon)g)\leq\alpha(f\wedge g)$. 
Letting $\varepsilon\to 0$ and using that $\alpha$ preserves suprema of increasing sequences, we get that $\alpha(f)\wedge\alpha(g)\leq \alpha(f\wedge g)$.

(iv): Let $f,g\in L(F(\Cu(A)))$.
If $x,y\in\Cu(A)$ and $\varepsilon>0$ satisfy $\widehat{x}\leq (1-\varepsilon)f$ and $\widehat{y}\leq (1-\varepsilon) g$, then $\widehat{x+y}=\widehat{x}+\widehat{y}\leq (1-\varepsilon) (f+g)$, which implies that $x+y\leq \alpha(f+g)$.
Passing to the supremum of all such $x$ and $y$, the desired inequality follows.
\end{proof}

In \autoref{SecSupersoft}, we will study the question of when $\alpha$ is additive.

We use the map $\alpha$ to solve the problem of realizing elements of $L(F(\Cu(A)))$ as ranks of Cuntz semigroup elements when $A$ is separable and of stable rank one.
We show that, under suitable hypotheses, $f=\widehat{z}$ for $z=\alpha(f)$.
We first prove that this is the case when $f$ is the chisel of an extreme densely finite functional (see \autoref{extremechisel} and \autoref{lma:chisel_realization} below), and then extend this to either arbitrary or full functions, depending on the hypotheses assumed.

\begin{prp}
Let $A$ be a separable C*-algebra of stable rank one. Let $\lambda\in F_0(\Cu(A))$ be an extreme densely finite functional.
Then the chisel $\sigma_\lambda$ at $\lambda$ belongs to $L(F(\Cu(A)))$. 
\end{prp}	
\begin{proof}
If $\lambda=0$ the proposition holds trivially, as $\sigma_0=\widehat{\infty}$, with $\infty$ denoting the largest element of $\Cu(A)$.
Assume thus that $\lambda\neq 0$.

\emph{Claim: Let $f_1,f_2\in L(F(\Cu(A)))$. Then $(f_1\wedge f_2)(\lambda)=\min(f_1(\lambda),f_2(\lambda))$.}
To prove the claim, let $x,y\in\Cu(A)$ and $m,n\in\NN$ satisfy $\tfrac{\widehat{x}}{m}\leq f_1$ and $\tfrac{\widehat{y}}{n}\leq f_2$.  Using \autoref{hat-inf-preserving} at the first step, and using at the second step that that $\Cu(A)$ satisfies Edwards' condition for $\lambda$ (see \cite[Theorem~5.3]{AntPerRobThi19arX:Edwards}, and the characterization of the Edwards' condition given at the end of \autoref{pgr:Edward}), we get
\[
(n\widehat{x}\wedge m\widehat{y})(\lambda)
= (nx\wedge my)^{\wedge}(\lambda)
= \min \{ n\widehat{x}(\lambda), m\widehat{y}(\lambda) \}.
\] 
Then
\[
(f_1\wedge f_2)(\lambda)
= \frac{1}{mn} (mnf_1\wedge mnf_2)(\lambda)
\geq \frac{1}{mn} (n\widehat{x}\wedge m\widehat{y})(\lambda)
= \min \{ \frac{\widehat{x}(\lambda)}{m},\frac{\widehat{y}(\lambda)}{n} \}.
\]

Since $L(F(\Cu(A)))=\Cu(A)_R$ by \autoref{lma:basic_cones} (i), both $f_1$ and $f_2$ are suprema of increasing sequences whose terms take the form $\tfrac{\widehat{x}}{n}$ and $\tfrac{\widehat{y}}{m}$. 
Therefore, 
\[
(f_1\wedge f_2)(\lambda)\geq \min\{ f_1(\lambda),f_2(\lambda) \}.
\] 
The opposite inequality is straightforward.  This proves the claim.

Let us now show that the set 
\begin{equation}\label{flambda1}
\big\{ f\in L(F(\Cu(A))): f(\lambda)<1 \big\} 
\end{equation}
is upward directed.  Let $f_1,f_2\in L(F(\Cu(A)))$ satisfy $f_1(\lambda),f_2(\lambda)<1$.  Assume without loss of generality that $f_1(\lambda)\leq f_2(\lambda)$. By the claim established above, $(f_1\wedge f_2)(\lambda)=f_1(\lambda)$.  Choose $\varepsilon>0$ such that $f_2(\lambda)+\varepsilon<1$.
Next, choose $g,\tilde g\in L(F(\Cu(A)))$ such that $g\lhd \tilde{g}\ll f_1\wedge f_2$ and $g(\lambda)>f_1(\lambda)-\varepsilon$.

By \cite[Lemma~3.3.2]{Rob13Cone}, if $c\lhd d'\ll d$ in $L(F(\Cu(A)))$, then there exists $e$ such that $c+e=d$ and $c\propto e$, and thus also $d\propto e$.
Applied to $g\lhd\tilde{g}\ll f_1+f_2$, we obtain $h$ such that $g+h=f_1+f_2$ and $f_1+f_2\propto h$.
We have
\[
f_1+h\geq g+h=f_1+f_2.
\]
Hence, since $f_1\propto h$, we may use cancellation in $L(F(\Cu(A)))$ to conclude that $h\geq f_2$.
Symmetrically, $h\geq f_1$.
On the other hand, 
\[
f_1(\lambda)-\varepsilon+h(\lambda)\leq g(\lambda)+h(\lambda)=f_1(\lambda)+f_2(\lambda),
\]
from which we deduce that $h(\lambda)\leq f_2(\lambda)+\varepsilon<1$.

Since $L(F(\Cu(A)))$ is a countably based Cu-semigroup, the upward directed set in \eqref{flambda1}
has a supremum, which we now proceed to prove is precisely $\sigma_\lambda$. 
To this end, it suffices to show that for any nonzero $\mu\in F(\Cu(A))$ such that $\mu\notin (0,\infty)\lambda$ and any $C>0$ there exists $f\in L(F(\Cu(A)))$ such that $f(\lambda)<1$ and $f(\mu)>C$.
To show this, choose $x\in \Cu(A)$ such that $0<\lambda(x)<\infty$, which is possible as $\lambda\neq 0$ by assumption.	Since $\mu$ is not a scalar multiple of $\lambda$ and the latter is extreme by assumption, we have $\mu\not \leq 4C\lambda$. Let $y\in \Cu(A)$ be such that $4C\lambda(y)<\mu(y)$. If $\lambda(y)=0$, then  $f=\frac{2C}{\mu(y)}\widehat{y}$ is as desired. Suppose that $\lambda(y)>0$. Set
\[
f=\frac{1}{4\lambda(x)}\widehat{x} + \frac{1}{4\lambda(y)}\widehat{y}.
\]
Clearly then $f(\lambda)=1/4+1/4<1$. 
Also, 
\[
f(\mu)\geq \frac{1}{4}\cdot \frac{\mu(y)}{\lambda(y)}>C.
\]
Hence, $f$ is as desired.	
\end{proof}	

\begin{lma}[{Cf.\ \cite[Lemma~5.2]{Thi17arX:RksOps}}]
\label{lma:lambda_mu}
Let $A$ be a C*-algebra. 
Let $\lambda,\mu\in F_0(\Cu(A))$ be densely finite functionals, with $\lambda$ extreme, $\mu$ nonzero, and $\mu\notin (0,\infty)\lambda$.
Then for every $\varepsilon>0$ there exists $w\in \Cu(A)$ such that $\lambda(w)<\varepsilon$ and $\mu(w)>\tfrac{1}{\varepsilon}$.	
\end{lma}	
\begin{proof}
If $\lambda=0$ the lemma follows easily.  Let us thus assume that $\lambda\neq 0$.  Observe that $\mu\neq 0$ by assumption.  We claim that there exists $x\in\Cu(A)$ such that $0<\lambda(x)<\infty$ and $0<\mu(x)<\infty$.
Indeed, since $\lambda$ and $\mu$ are nonzero, there exist $x_1,x_2\in\Cu(A)$ with $0<\lambda(x_1)$ and $0<\mu(x_2)$. 
Choose $x_1',x_2'\in\Cu(A)$ with $x_1'\ll x_1$ and $x_2'\ll x_2$ such that $0<\lambda(x_1')$ and $0<\mu(x_2')$. 
Since $\lambda$ and $\mu$ are densely finite, they are finite on $x_1'$ and $x_2'$, and so $x=x_1'+x_2'$ is as desired.

Let us now normalize $\lambda$ and $\mu$ so that $\lambda(x)=\mu(x)=1$. The normalized functionals are multiples of the original functionals by fixed scalars not depending on $\varepsilon$. Thus, the proof of the lemma may be reduced to the normalized functionals.

\emph{Case~1: Suppose that $\lambda\nleq \mu$.}
Let $y\in\Cu(A)$ be such that $\mu(y)<\lambda(y)<\infty$. Set $\delta=\lambda(y)-\mu(y)$. 
Choose numbers $m,n\in\NN$ such that
\[
\frac{1}{\varepsilon}+\frac{\varepsilon}{2}<m\delta \andSep
|\lambda(nx)-\lambda(my)|=|n-m\lambda(y)|<\frac{\varepsilon}{2}\,.
\]

Since $\Cu(A)$ satisfies Edwards' condition for $\lambda$, and since $\lambda$ is extreme and densely finite, there exists $z\leq nx, my$ such that $\min\{n,m\lambda(y)\}-\frac{\varepsilon}{2}<\lambda(z)$ (see \autoref{pgr:Edward}).
Since also $|n-m\lambda(y)|<\frac{\varepsilon}{2}$, we have that $n-\lambda(z)<\varepsilon$. 
Choose $z'\in \Cu(A)$ such that $z'\ll z$ and $n-\lambda(z')<\varepsilon$.
Now, by \axiomO{5} applied to $z'\ll z\leq nx$, there is $w\in S$ such that $z'+w\leq nx\leq z+w$. 
Then $\lambda(w)\leq \lambda(nx)-\lambda(z')<\varepsilon$.
Also,
\begin{align*}
\mu(z)+\mu(w)
\geq \mu(nx)
=n
&=n-\lambda(my)+\lambda(my)\\
&\geq -\frac{\varepsilon}{2}+m\delta+\mu(my)\\
&\geq -\frac{\varepsilon}{2}+m\delta+\mu(z).
\end{align*}
Therefore $\mu(w)\geq -\frac{\varepsilon}{2}+ m\delta\geq \frac{1}{\varepsilon}$, as desired.

\emph{Case~2: Suppose that $\lambda\leq \mu$.}
Since $F(\Cu(A))$ is algebraically ordered (\cite[Proposition~2.2.3]{Rob13Cone}), we can choose a functional $\mu'$ such that $\lambda+\mu'=\mu$. 
Observe that $\mu'(x)=0$.
We cannot have that $\mu'=0$, since $\lambda\neq \mu$.
Thus, as argued at the beginning of the proof, there exists $x_1\in \Cu(A)$ such
$0<\lambda(x_1)<\infty$ and $0<\mu'(x_1)<\infty$.
Let $\lambda'$ and $\mu''$ be the normalizations of $\lambda$ and $\mu'$, respectively, such that $\lambda'(x_1)=\mu''(x_1)=1$.
Observe that $\lambda'\nleq \mu''$, as $\mu''(x)=0$ while $\lambda'(x)\neq 0$.
Thus, we can apply Case~1 to the functionals $\lambda'$ and $\mu''$, normalized at $x_1$, to find the desired $w$.
Since the normalizing factors do not depend on $\varepsilon$, we can arrange for $\lambda(w)<\varepsilon$ and $\mu'(w)>\tfrac{1}{\varepsilon}$, which in turn implies that $\mu(w)>\tfrac{1}{\varepsilon}$.
\end{proof}	

\begin{lma}
\label{lambdafullrange}
Let $A$ be a separable C*-algebra of stable rank one that has no nonzero type~I quotients, and let $\lambda\in F_0(\Cu(A))$ be a nonzero, densely finite functional.
\begin{enumerate}[\rm (i)]
	\item
	For each $\varepsilon>0$ the set $\{x\in \Cu(A):\lambda(x)<\varepsilon\}$ is a full subset of $\Cu(A)$.
	\item
	The range of $\lambda$ is $[0,\infty]$.
\end{enumerate}
\end{lma}	
\begin{proof}
(i):
Let $W\subseteq \Cu(A)$ be the ideal generated by $\{x\in \Cu(A):\lambda(x)<\varepsilon\}$.
Using the natural correspondence between closed, two-sided ideals of $A$ and ideals of $\Cu(A)$ (see \autoref{pgr:ideal}), we let $I\subseteq A$ be the closed, two-sided ideal of $A$ such that $\Cu(I)=W$.
Suppose for the sake of contradiction that $I$ is proper.
Let $x\in\Cu(A)$ with $\lambda(x)<\infty$.
Find $x''\ll x'\ll x$ such that $\lambda(x)-\lambda(x'')<\varepsilon$.
By \axiomO{5}, there exists $w\in\Cu(A)$ such that $x''+w\leq x\leq x'+w$.
Evaluating on $\lambda$ we get $\lambda(w)<\varepsilon$, whence $w\in W$.
Thus, the images of $x$ and $x'$ in $\Cu(A/I)$ agree.
It follows that the image of $x$ in $\Cu(A/I)$ is compact. 

Next, we show that $A/I$ contains a positive element with spectrum $[0,1]$.
Since $A/I$ is not type~I, it follows from Glimm's theorem that there exists a sub-C*-algebra $B\subseteq A/I$ that has a UHF-algebra as a quotient.
In a UHF-algebra it is easy to find a positive element $\tilde{b}$ with spectrum $[0,1]$.
Lift $\tilde{b}$ to a positive, contractive element $b$ in $B$.
Then $b$ has spectrum $[0,1]$ in $B$, and consequently also in $A/I$.
By \cite[Theorem~3.5]{BroCiu09IsoHilbModSF}, if $C$ is a stably finite C*-algebra and $c\in C_+$, then $[c]$ is compact if and only if $0$ is an isolated point of the spectrum of $c$.
Since $A/I$ has stable rank one, it is stably finite, and it follows that $[(b-t)_+]\in\Cu(A/I)$ is not compact for every $t\in(0,1)$.
Let $a\in A_+$ be a lift of $b$.
Then $\lambda([(a-1/2)_+])<\infty$ and $[(a-1/2)_+]$ is mapped to $[(b-1/2)_+]$ in $\Cu(A/I)$, which is not compact. This contradicts what was proved in the previous paragraph.	

(ii): 
Let us first show that $\lambda$ attains arbitrarily small nonzero values.
Fix $\varepsilon>0$.
By part (i), $\{x\in \Cu(A):\lambda(x)<\varepsilon\}$ is a full subset of $\Cu(A)$. So if  $\lambda$ does not attain nonzero values less than $\varepsilon$ then  it is the zero functional, contradicting our assumption. 
Thus, there exists $x\in \Cu(A)$ such that $0<\lambda(x)<\varepsilon$.

It is now clear that the range of $\lambda$ is dense in $[0,\infty]$. Both 0 and $\infty$ are in the range of $\lambda$. Let $t\in (0,\infty)$. Choose $x_0\in \Cu(A)$ such that $\lambda(x_0)<t$. Having chosen $x_0\leq x_1\leq \cdots \leq x_{n-1}$ such that $\lambda(x_{n-1})<t$, choose $z$ such that 
\[
\frac{t-\lambda(x_{n-1})}{2}<\lambda(z)<t-\lambda(x_{n-1}),\] 
and set $x_n=x_{n-1}+z$. Then $\lambda$ attains the value $t$ at $x=\sup_n x_n$. 
\end{proof}

\begin{lma}
\label{lma:chisel_realization}
Let $A$ be a separable C*-algebra of stable rank one that has no nonzero type~I quotients. 
Let $\lambda\in F_0(\Cu(A))$ be an extreme, densely finite functional, and let $\sigma_{\lambda}$ denote its chisel.
Then $\sigma_{\lambda}=\widehat{z}$ for $z=\alpha(\sigma_{\lambda})$, where $\alpha$ is the map from \autoref{dfn:mapAlpha}.
\end{lma}
\begin{proof}
If $\lambda=0$, then $\sigma_{\lambda}$ is the function that is $0$ at $0$ (the zero functional) and $\infty$ otherwise.
In this case, $z=\infty$ (the largest element of $\Cu(A)$), and the lemma holds trivially.
Assume thus that $\lambda\neq 0$.

By \autoref{prp:mapAlpha}~(i), we have $\widehat{z}\leq\sigma_\lambda$. 
We first show that $z$ is full.
Let $x\in \Cu(A)$ satisfy $\lambda(x)<1$.
Then $\widehat{x}\leq (1-\varepsilon) \sigma_{\lambda}$ for a sufficiently small $\varepsilon$, and thus $x\leq \alpha(\sigma_\lambda)=z$ by \autoref{prp:mapAlpha}~(ii).
Hence, by \autoref{lambdafullrange}~(i), $z$ is full. 

Let $0<\varepsilon<1$.
By \autoref{lambdafullrange} (ii), there exists $x\in \Cu(A)$ such that $\lambda(x)=1-\varepsilon$. 
Then $\widehat{x}\leq (1-\varepsilon)\sigma_\lambda$, whence $x\leq z$ by \autoref{prp:mapAlpha}~(ii).
Evaluating at $\lambda$ we get $1-\varepsilon\leq \lambda(z)$.
Since $\varepsilon$ can be arbitrarily small, we obtain $\lambda(z)=1$, that is, $\widehat{z}(\lambda)=1$. 

Let $\mu$ be a nonzero, densely finite functional that is not a scalar multiple of $\lambda$, and let $\varepsilon>0$.
By \autoref{lma:lambda_mu}, there exists $w\in \Cu(A)$ such that $\lambda(w)<1$ and $\mu(w)>\tfrac{1}{\varepsilon}$.
As in the second paragraph of the proof, we get $w\leq z$, from which we obtain that $\mu(z)>\tfrac{1}{\varepsilon}$.
Since $\varepsilon$ can be arbitrarily small, we deduce $\mu(z)=\infty$, that is, $\widehat{z}(\mu)=\infty$. 

We have shown that $\widehat{z}(\mu)=\sigma_\lambda(\mu)$ for all $\mu$ densely finite. Further, since $z$ is full, this equality holds also  for all functionals that are not densely finite, as in this case both sides equal $\infty$.
The lemma is thus proved. 
\end{proof}	

\begin{lma}\label{oldclaim2}
Let $A$ be a separable C*-algebra of stable rank one. Let $f,g\in L(F(\Cu(A)))$ satisfy $g\ll f$.
Then there exists $h\in L(F(\Cu(A)))$ such that $g\leq h\leq f$ and $h$ is continuous on $F_0(\Cu(A))$.	
\end{lma}	
\begin{proof}
Since $f$ is the supremum of a $\lhd$-increasing sequence in $L(F(\Cu(A)))$ (see \autoref{dualofcone}), we
can choose $h,h'\in L(F(\Cu(A)))$ such that $g\leq h\lhd h'\lhd f$. 
Since $h'\lhd f$ implies $h'\ll f$, and since $L(F(\Cu(A)))$ agrees with $\Cu(A)_R$ (by \autoref{lma:basic_cones}~(i)), there exist $y\in \Cu(A)$ and $k\in\NN$ such that $h'\ll \tfrac{\widehat{y}}{k}\leq f$. 
Now choose $y'\ll y$ such that $h'\ll \tfrac{\widehat{y'}}{k}$. 
Then $\widehat{y'}$ is finite on densely finite functionals. 
Hence, $h'$ is finite on densely finite functionals, from which we get that $h$ is continuous on every densely finite functional, as desired.  
\end{proof}	

\begin{thm}
\label{mainrealizing}
Let $A$ be a separable C*-algebra of stable rank one that has no nonzero type~I quotients, and let $f\in L(F(\Cu(A)))$ be a full function.
Then $f=\widehat{z}$ for $z=\alpha(f)$, where $\alpha$ is the map from \autoref{dfn:mapAlpha}.
\end{thm}	
\begin{proof}
By \autoref{prp:mapAlpha} (i), we have $\widehat{z}\leq f$. 
Our goal is then to prove that $f\leq\widehat{z}$. 
Let $g\in L(F(\Cu(A)))$ satisfy $g\ll f$.
Since $f$ is the supremum of all such $g$, it suffices to show that $g\leq\widehat{z}$.

Choose $g_0$ such that $g\ll g_0\ll f$, and then choose $\varepsilon>0$ such that $g_0\ll (1-3\varepsilon)f$.
Below we will find $x\in\Cu(A)$ such that $g\leq \widehat{x} \leq (1-\varepsilon)f$.
By \autoref{prp:mapAlpha} (ii), such $x$ satisfies $x\leq \alpha(f)=z$, which then shows that $g\leq \widehat{z}$.

To find $x\in\Cu(A)$ such that $g\leq \widehat{x} \leq (1-\varepsilon)f$, we set 
\[
K = \big\{ \lambda \in F(\Cu(A)) : f(\lambda)\leq 1 \big\}.
\]

\emph{Claim~1: We have $K\subseteq F_0(\Cu(A))$.}
To prove the claim, let $\lambda\in K$, and let $y',y\in \Cu(A)$ be such that $y'\ll y$. 
Then $\widehat{y'}\ll 2\widehat{y}$ in $L(F(\Cu(A)))$ by \autoref{lma:basic_cones} (ii). Since $2\widehat{y}\leq \infty f$ (by the fullness of $f$), $\widehat{y'}\propto f$.
It follows that $\lambda(y')<\infty$. 
This holds for all $y',y$ with $y'\ll y$, thus showing that $\lambda$ is densely finite, completing the proof of the claim.

It now follows from \cite[Proposition~13.2]{Phe01LNMChoquet} that $K$ is a cap of $F_0(\Cu(A))$, that is, a compact convex set whose complement is also convex. Since $K$ is a cap of $F_0(\Cu(A))$, every extreme point of $K$ is also an extreme functional in $F_0(\Cu(A))$, by \cite[Proposition~13.1]{Phe01LNMChoquet}.

Using that $g\ll g_0$, choose $g_1\in L(F(\Cu(A)))$ with $g\ll g_1\ll g_0$.
Then inductively choose $g_2,g_3,\ldots$ in $L(F(\Cu(A)))$ such that
\[
g \ll \ldots \ll g_3\ll g_2\ll g_1\ll g_0.
\]
Let $\Aff(K)$ denote the set of continuous, affine functions $K\to\RR$, and set
\[
D = \big\{ h\in\Aff(K) : g_n\leq\widehat{x} \text{ and } (1+\frac{1}{n})\widehat{x}|_{K} \leq h \text{ for some } x\in\Cu(A), n\geq 1 \big\}.
\]

\emph{Claim~2: $D$ is downward directed.}
To prove the claim, let $h_1,h_2\in D$.
Choose $x_1,x_2\in\Cu(A)$ and $n_1,n_2\in\NN$ such that
\[
g_{n_1}\leq\widehat{x_1},\quad 
(1+\frac{1}{n_1})\widehat{x_1}|_{K} \leq h_1,\quad
g_{n_2}\leq\widehat{x_2}, \andSep 
(1+\frac{1}{n_2})\widehat{x_2}|_{K} \leq h_2.
\]
Set $n=\max\{n_1,n_2\}$ and $x=x_1\wedge x_2\in\Cu(A)$.
Using \autoref{hat-inf-preserving}, we get
\[
g_n \leq \widehat{x_1}\wedge\widehat{x_2}=\widehat{x}, \andSep
(1+\frac{1}{n})\widehat{x}|_{K} \leq h_1,h_2.
\]
Using that $g_{n+1}\ll g_n$, choose $x'\in\Cu(A)$ such that $x'\ll x$ and $g_{n+1}\leq\widehat{x'}$.
Set $\beta=\tfrac{(n+2)n}{(n+1)^2}$.
Since $\beta<1$, we have $\beta\widehat{x'}\ll\widehat{x}$ by \autoref{lma:basic_cones} (ii).
Using \autoref{oldclaim2}, we obtain $h\in L(F(\Cu(A)))$ with $\beta\widehat{x'} \leq h \leq \widehat{x}$ and such that $h|_K\in\Aff(K)$.
Then $(1+\tfrac{1}{n})h|_{K}$ is a lower bound for $h_1$ and $h_2$ in $D$, since $g_{n+1}\leq\widehat{x'}$ and 
\[
(1+\frac{1}{n+1})\widehat{x'}
\leq (1+\frac{1}{n+1})\frac{1}{\beta}h
= (1+\frac{1}{n})h, \andSep
(1+\frac{1}{n})h|_K 
\leq (1+\frac{1}{n})\widehat{x}|_K 
\leq h_1,h_2,
\]
which proves the claim.

Define $h_0\colon K\to\RR$ as the pointwise infimum of the functions in $D$.
Then $h_0$ is upper semicontinuous and affine.

\emph{Claim~3: We have $g|_{K}\leq h_0 \leq (1-2\varepsilon)f|_{K}$.}
The first inequality follows since $g|_K\leq h$ for every $h\in D$.
Let $\partial_e K$ denote the set of extreme points of $K$.
We first show that $h_0|_{\partial_e K}\leq(1-2\varepsilon)f|_{\partial_e K}$.
Let $\lambda\in\partial_e K$.
If $\lambda=0$ (the zero functional), then $h_0(\lambda)=0=(1-2\varepsilon)f(\lambda)$.
We may thus assume that $\lambda\neq 0$.
Since $\lambda$ is densely finite and $f$ is full, we obtain $f(\lambda)>0$.
Set $x_\lambda =\alpha((1-3\varepsilon)f(\lambda)\sigma_\lambda)$. 
We know by \autoref{lma:chisel_realization} that $\widehat{x_\lambda}=(1-3\varepsilon)f(\lambda)\sigma_\lambda$.
Hence
\[
g_0\ll (1-3\varepsilon)f 
\leq (1-3\varepsilon)f(\lambda)\sigma_\lambda
=\widehat{x_\lambda}.
\]
Choose $x'_\lambda\in\Cu(A)$ such that $x'_\lambda\ll x_\lambda$ and $g_0\leq \widehat{x'_\lambda}$.
By \autoref{prp:mapAlpha} (ii) and the fact that $x_\lambda=\alpha(\widehat{x_\lambda})$, we have $\widehat{x'_\lambda} \ll \widehat{x_\lambda}$.
Using \autoref{oldclaim2}, we obtain $h_\lambda\in L(F(\Cu(A)))$ with $\widehat{x'_\lambda} \leq  h_\lambda \leq \widehat{x_\lambda}$ and such that $h_\lambda$ is continuous on $K$.
Let $n\geq 1$ such that $(1+\tfrac{1}{n})(1-3\varepsilon)\leq(1-2\varepsilon)$.
Then 
\[
g_n\leq g_0\leq\widehat{x'_\lambda}, \andSep
(1+\frac{1}{n})\widehat{x'_\lambda} \leq (1+\frac{1}{n})h_\lambda,
\]
which shows that $(1+\tfrac{1}{n})h_\lambda|_K$ belongs to $D$.
It follows that
\[
h_0(\lambda)\leq (1+\frac{1}{n})h_\lambda(\lambda)
\leq (1+\frac{1}{n})\widehat{x_\lambda}(\lambda)
= (1+\frac{1}{n})(1-3\varepsilon)f(\lambda)
\leq (1-2\varepsilon)f(\lambda).
\]

We have shown that $h_0(\lambda)\leq(1-2\varepsilon)f(\lambda)$ for every $\lambda\in\partial_e K$.
To extend this inequality to all of $K$, note that $h_0$ takes values in $[0,\infty)$, which allows us to consider $d=h_0-(1-2\varepsilon)f|_K$.
Then $d\colon K\to[-\infty,\infty)$ is upper semicontinuous and affine with $d(\lambda)\leq 0$ for every $\lambda\in\partial_e K$.
By Bauer's maximum principle (\cite[Proposition~16.6]{Phe01LNMChoquet}) we get $d\leq 0$ and thus $h_0\leq(1-2\varepsilon)f|_K$, which proves the claim.

For each $h\in D$, define
\[
U_h = \big\{ \lambda\in K : h(\lambda)<(1-\varepsilon)f(\lambda) \big\} \andSep
V_h = \big\{ \lambda\in K : h(\lambda)<1-\varepsilon \big\},
\]
which are open subsets of $K$ as $h$ is continuous and $f|_{K}-h$ is lower semicontinuous.
Using the inequality $h_0\leq (1-2\varepsilon)f|_K$ we see that
\[
K = \bigcup_{h\in D}(U_h\cup V_h).
\]
By compactness of $K$, there is a finite subset $F\subseteq D$ such that $K=\bigcup_{h\in F}(U_h\cup V_h)$.
Since $D$ is downward directed, we can choose $h\in D$ that is a lower bound for $F$.
Then $K=U_h\cup V_h$.
By definition of $D$, we obtain $x\in\Cu(A)$ such that $g\leq \widehat{x}$ and $\widehat{x}|_{K} \leq h$.
The verification of the following claim finishes the proof.

\emph{Claim~4: We have $\widehat x \leq (1-\varepsilon)f$.}
Let $\lambda\in F(\Cu(A))$. If $f(\lambda)=\infty$ the claim holds trivially, so assume that $f(\lambda)<\infty$.
If $f(\lambda)=0$, then as argued in \autoref{probrealizingfull} the fullness of $f$ implies that $\lambda=0$,
and again the claim holds trivially. Suppose that $0<f(\lambda)<\infty$ and let $t=\tfrac{1}{f(\lambda)}$.
Clearly, $t\lambda$ is an element of $K$.
If $t\lambda\in U_h$, then $\widehat{x}(t\lambda)\leq h(t\lambda)<(1-\varepsilon)f(t\lambda)$, and hence $\widehat{x}(\lambda)\leq (1-\varepsilon)f(\lambda)$.
If, on the other hand, $t\lambda\in V_h$, then 
\[
\widehat{x}(t\lambda)\leq h(t\lambda)<1-\varepsilon=(1-\varepsilon)f(t\lambda),
\] 
and again we obtain that $\widehat{x}(\lambda)\leq (1-\varepsilon)f(\lambda)$.
\end{proof}

\begin{rmk}
We use this opportunity to amend the proof of \cite[Lemma~8.1]{Thi17arX:RksOps}.
Given a compact, convex set $K$, it was claimed there that the infimum of finitely many continuous, affine functions on $K$ is again continuous, but in general it is only lower semicontinuous.
The argument in \cite[Lemma~8.1]{Thi17arX:RksOps} can be fixed along the lines of the proof of \autoref{mainrealizing}, by considering a downward directed family of continuous, affine functions analogous to the set $D$ defined there.
\end{rmk}

\begin{pgr}
\label{pgr:observation}
Let $A$ be a C*-algebra and let $I\subseteq A$ be a closed, two-sided ideal of $A$.
Recall that we may regard $\Cu(I)$ as an ideal of $\Cu(A)$ (see \autoref{pgr:ideal}). 
Given $\lambda\in F(\Cu(I))$, define $\tilde{\lambda}\in F(\Cu(A))$ as in \autoref{prp:extFctl}, that is,
\[
\tilde\lambda(x)=\begin{cases}
\lambda(x), &\hbox{if } x\in \Cu(I)\\
\infty, &\hbox{otherwise.}
\end{cases}
\]
The assignment $\lambda\mapsto\tilde{\lambda}$ defines an order-embedding $F(\Cu(I))\to F(\Cu(A))$ which is a right inverse to the restriction map $\lambda\mapsto \lambda|_{\Cu(I)}$ from $F(\Cu(A))$ to $F(\Cu(I))$.
Thus, the restriction map is surjective. It follows that
given $x,y\in\Cu(I)$, we have  $\widehat{x}\leq\widehat{y}$ in $L(F(\Cu(I)))$ if and only if $\widehat{x}\leq\widehat{y}$ in $L(F(\Cu(A)))$.

Consider now the natural map $\iota\colon L(F(\Cu(I)))\to \mathrm{Lsc}(F(\Cu(A)))$ such that $\iota(f)(\lambda)=f(\lambda|_{\Cu(I)})$ for $f\in L(F(\Cu(I)))$ and $\lambda\in F(\Cu(A))$.
Using the description of $L(F(\Cu(I)))$ and $L(F(\Cu(A)))$ as suprema of increasing sequences with elements of the form $\tfrac{\widehat{x}}{n}$ (see \autoref{lma:basic_cones} (i)) we obtain that $\iota$  in fact ranges in $L(F(\Cu(A)))$.
\end{pgr}

\begin{thm}
\label{thm:realizingprob1}
Let $A$ be a separable C*-algebra of stable rank one that has no nonzero, elementary ideal-quotients, and let $f\in L(F(\Cu(A)))$.
Then $f=\widehat{z}$ for $z =\alpha(f)$, where $\alpha$ is the map from \autoref{dfn:mapAlpha}.
\end{thm}
\begin{proof}
The set $\{x\in\Cu(A) : \widehat{x}\leq\infty f\}$ is an ideal of $\Cu(A)$.
Using the bijection between closed, two-sided ideals of $A$ and ideals of $\Cu(A)$ (\autoref{pgr:ideal}), we let $I\subseteq A$ be the closed, two-sided ideal of $A$ such that $\Cu(I)=\{x\in\Cu(A) : \widehat{x}\leq\infty f\}$.
(Here we identify $\Cu(I)$ with its image in $\Cu(A)$ induced by the inclusion of $I$ in $A$.)
Note that $I$ is a separable C*-algebra of stable rank one that has no nonzero, elementary ideal-quotients, and in particular no nonzero type~I quotients.

Since $L(F(\Cu(A)))=\Cu(A)_R$ by \autoref{lma:basic_cones} (i), we can choose a sequence $(x_n)_n$ in $\Cu(A)$ and a sequence $(k_n)_n$ of positive integers such that $(\tfrac{\widehat{x_n}}{k_n})_n$ is increasing with supremum $f$ in $L(F(\Cu(A)))$.
For each $n$, we have $\widehat{x_n}\leq \infty f$ and thus $x_n\in\Cu(I)$.
As noted in \autoref{pgr:observation}, it follows that $(\tfrac{\widehat{x_n}}{k_n})_n$ is an increasing sequence in $L(F(\Cu(I)))$, and we let $f_0$ denote its supremum in $L(F(\Cu(I)))$.
Given $\lambda\in F(\Cu(A))$, we have $f(\lambda)=f_0(\lambda|_{\Cu(I)})$.

\emph{Claim~1: $f_0$ is full in $L(F(\Cu(I)))$.}
To prove the claim, let $g\in L(F(\Cu(I)))$.
Choose a sequence $(y_n)_n$ in $\Cu(I)$ and a sequence $(l_n)_n$ of positive integers such that $(\tfrac{\widehat{y_n}}{l_n})_n$ is increasing in $L(F(\Cu(I)))$ with supremum $g$.
For each $n$, since $y_n$ belongs to $\Cu(I)$, we have $\widehat{y_n}\leq\infty f$ in $L(F(\Cu(A)))$, and it follows that $\widehat{y_n}\leq\infty f_0$ in $L(F(\Cu(I)))$.
Thus, $g=\sup_n \tfrac{\widehat{y_n}}{l_n} \leq \infty f_0$, which proves the claim.

Let $\alpha_I\colon L(F(\Cu(I)))\to\Cu(I)$ be the map from \autoref{dfn:mapAlpha} for $I$.
Set $z_0 =\alpha_I(f_0)$. 
By \autoref{mainrealizing}, we have $f_0=\widehat{z_0}$ in $L(F(\Cu(I)))$.

\emph{Claim~2: We have $z=z_0$.}
Set
\[
L= \big\{ x\in\Cu(A) : \widehat{x}\leq (1-\varepsilon)f\text{ in } L(F(\Cu(A))) \hbox{ for some }\varepsilon>0 \big\}.
\]
By \autoref{prp:mapAlpha}~(ii), $z$ is the supremum of $L$ in $\Cu(A)$.
If $x\in\Cu(A)$ and $\varepsilon>0$ satisfy $\widehat{x}\leq(1-\varepsilon)f$ in $L(F(\Cu(A)))$, then $x$ belongs to $\Cu(I)$ and we have $\widehat{x}\leq(1-\varepsilon)f_0$ in $L(F(\Cu(I)))$.
It follows that $L\subseteq\Cu(I)$ and that $z_0$ is the supremum of $L$ in $\Cu(I)$.
Since $\Cu(I)\subseteq\Cu(A)$ is downward hereditary, the supremum of $L$ in $\Cu(I)$ and in $\Cu(A)$ agree, and thus $z=z_0$, which proves the claim.

Given $\lambda\in F(\Cu(A))$, we deduce that
\[
\widehat{z}(\lambda)
= \lambda(z)
= \lambda|_{\Cu(I)}(z_0)
= f_0(\lambda|_{\Cu(I)})
= f(\lambda),
\]
and thus $\widehat{z}=f$ in $L(F(\Cu(A)))$.
\end{proof}

\begin{thm}
\label{thm:realizingprob2}
Let $A$ be a separable, unital C*-algebra of stable rank one that has no nonzero, finite dimensional quotients.
Set $u=[1_A]$.
Let $F_u(\Cu(A))$ denote the set of functionals $\lambda\in F(\Cu(A))$ normalized at $u$.
Then for each $f\in\LAff(F_u(\Cu(A)))_{++}^\sigma$ there exists $z\in \Cu(A)$ such that $\widehat{z}|_{F_u(\Cu(A))}=f$.

Further, if for some $n\in \NN$ we have that $f(\lambda)\leq n$ for all  $\lambda\in F_u(\Cu(A))$, then $z$ may be chosen such that $z\leq nu$, and there is $a\in (A\otimes M_n)_+$ such that $f(\lambda)=\lambda([a])$ for every $\lambda\in F_u(\Cu(A))$.
\end{thm}	
\begin{proof}
Let $I$ be a closed, two-sided ideal of $A$ such that $A/I$ has type~I.
Choose a maximal ideal $J$ containing $I$.
Then $A/J$ is simple, unital and has type~I, whence it is finite dimensional.
It follows that $A$ has no nonzero type~I quotients.

We can thus apply \autoref{mainrealizing} to realize full functions in $L(F(\Cu(A)))$.
Moreover, by \autoref{surjectiveres}, given a function $f\in\LAff(F_u(\Cu(A)))_{++}^\sigma$, there exists a full function $\tilde f\in L(F(\Cu(A)))$ whose restriction to $F_u(\Cu(A))$ is $f$.
Then $\tilde f=\widehat{z}$ for $z=\alpha(\tilde f)$, and so $\widehat{z}|_{F_u(\Cu(A))}=f$.

Let us address the last assertion of the theorem.
Suppose that $f(\lambda)\leq n$ for all  $\lambda\in F_u(\Cu(A))$. Then $\tilde f\leq n\widehat{u}$.
Choose $y\in\Cu(A)$ such that $\tilde f=\widehat{y}$. 
Set $z=y\wedge nu$. 
Then $z\leq nu$, and using \autoref{hat-inf-preserving}, we have $\widehat{z}=\tilde f\wedge n\widehat{u}=\tilde f$.
Thus $z$ is as desired.
Choose $b\in(A\otimes\mathcal{K})_+$ with $z=[b]$.
Let $1\otimes 1_n$ denote the unit in $A\otimes M_n$.
Then $[b]=z\leq nu=[1\otimes 1_n]$.
Since $A$ has stable rank one, there exists a positive element $a$ in the hereditary sub-\ca{} generated by $1\otimes 1_n$ (that is, $a\in(A\otimes M_n)_+$) with $[a]=[b]=z$;
see \cite[Theorem~4.29]{AraPerTom11Cu} or \cite[Paragraph~6.2]{OrtRorThi11CuOpenProj}.
\end{proof}

\section{Supersoft elements and comparability}
\label{SecSupersoft}

In this section we introduce the notion of supersoft elements in Cuntz semigroups of separable C*-algebra of stable rank one.
We use these elements to advance further the study of comparability properties in the Cuntz semigroups of these C*-algebras.

Recall from \autoref{dfn:soft} that an element $x$ in a \CuSgp{} is \emph{soft} if for every $x'\ll x$ there is $k\in\NN$ with $(k+1)x'\leq kx$.

\begin{dfn}
\label{dfn:supersoft}
Let $A$ be a separable C*-algebra of stable rank one, and let $z\in\Cu(A)$.
We call $z$ \emph{supersoft} if $\alpha(\widehat{z})=z$, where $\alpha$ is the map from \autoref{dfn:mapAlpha}.
Thus, $z$ is supersoft precisely when $z=\sup\{x\in\Cu(A)\colon\widehat{x}\leq (1-\varepsilon)\widehat{z}\text{ for some }\varepsilon>0\}$.
\end{dfn}

\begin{prp}
\label{basicsupersoft}
Let $A$ be a separable C*-algebra of stable rank one.  
\begin{enumerate}[\rm(i)]
	\item
	If $z\in \Cu(A)$ is supersoft, then $z$ is soft.
	\item
	If $x\in \Cu(A)$ is soft, $z\in \Cu(A)$ is supersoft, and $\widehat{x}\leq\widehat{z}$, then $x\leq z$.
	\item
	If $x$ is soft, then $x\leq \alpha(\widehat{x})$ and $\alpha(\widehat{x})$ is supersoft.
\end{enumerate}
\end{prp}	
\begin{proof}
(i):
Let $z=\alpha(\widehat{z})$ be supersoft.
Let $z'\ll z$.
Then $\widehat{z'}\ll \widehat{z}$, by \autoref{prp:mapAlpha}~(ii).
This in turn implies that $z$ is soft (see \cite[Proposition~5.3.3]{AntPerThi18TensorProdCu}).

(ii):
Let $x'\in \Cu(A)$ satisfy $x'\ll x$.
Since $x$ is soft, $\widehat{x'}\ll \widehat{x}\leq \widehat{z}$ (see \autoref{lma:basic_cones} (ii)).
Thus, $\widehat{x'}\ll \widehat{z}$ for every $x'\ll x$, and hence $x\in I_{\widehat{z}}$ (see \autoref{dfn:mapAlpha}). This implies that $x\leq \alpha(\widehat{z})=z$.

(iii):
Let $x'\in \Cu(A)$ satisfy $x'\ll x$.
Since $x$ is soft, we have as in (ii) that $\widehat{x'}\ll \widehat{x}$ and consequently $x\leq\alpha(\widehat{x})$. 
Hence, $\widehat{x}\leq \widehat{\alpha(\widehat{x})}$.
On the other hand, by \autoref{prp:mapAlpha}~(i) we have that $\widehat{\alpha(f)}\leq f$ for any $f$. 
Thus, $\widehat{\alpha(\widehat{x})}\leq \widehat{x}$.
It follows that $\widehat{x}=\widehat{\alpha(\widehat{x})}$.
Hence, $\alpha(\widehat{x})$ is supersoft.
\end{proof}

\begin{pgr}\label{ExistenceofSupersoft}
Let $A$ be a separable C*-algebra  of stable rank one.
Our results on realizing elements of $L(F(\Cu(A)))$ as ranks guarantee the existence of supersoft elements in $\Cu(A)$:
\begin{enumerate}
	\item
	By \autoref{mainrealizing}, if $f\in L(F(\Cu(A)))$ is a full function then $\alpha(f)$ is supersoft, provided that $A$ has no nonzero type~I quotients.
	In particular, this is true if $A$ is unital and has no nonzero, finite dimensional quotients.	
	\item	
	By \autoref{thm:realizingprob1}, the set of supersoft elements agrees with the range of $\alpha$, provided that $A$ has no nonzero, elementary ideal-quotients.
\end{enumerate}

Note that, if $A$ is a unital C*-algebra, the set of full elements in $\Cu(A)$ is in fact a \CuSgp. Indeed, if one shows that for each full element $x\in \Cu(A)$, there exists $x'\in\Cu(A)$ with $x'\ll x$ and such that $x'$ is also full, then suprema in $\Cu(A)$ and in the set of full elements will coincide and \axiomO{1}-\axiomO{4} are easily deduced. Let us find such $x'$. Since $x$ is full, and $A$ is unital, $[1]\ll [1]\leq \infty x$ implies that there exists $N\in \mathbb{N}$ such that $[1]\ll Nx$. Using this inequality, we can find $x'\ll x$ such that $[1]\leq Nx'$, and since $[1]$ is full, so is $x'$. This fact will be used in the proofs below.
\end{pgr}

\begin{thm}
\label{alphadditive}
Let $A$ be a separable C*-algebra of stable rank one, let $x\in\Cu(A)$, and let $f\in L(F(\Cu(A)))$ satisfy $\widehat{x}\leq \infty f$. 
Suppose that we are in one of the following cases:
\begin{enumerate}[\rm(i)]
	\item	
	$A$ is unital, has no nonzero,  finite dimensional quotients, and $f$ is full; 
	\item
	$A$ has no nonzero, elementary ideal-quotients. 
\end{enumerate}
Then
\[
\alpha(f+\widehat{x})=\alpha(f)+x.
\]
\end{thm}	
\begin{proof}
In both cases, we have that $\alpha (f)$ is supersoft. In case (i), this follows using \autoref{mainrealizing}, and in case (ii) using \autoref{thm:realizingprob1}. Thus $\alpha(f)$ is soft (see \autoref{basicsupersoft} (i)).
We have $\widehat{x}\leq \infty f= \infty \widehat{\alpha(f)}$, and thus $x\leq \infty\alpha(f)$ by \autoref{lma:basic_cones} (iii).
Since the subsemigroup of soft elements is absorbing (see \autoref{dfn:soft}), it follows that $\alpha(f)+x$ is soft. 
Using this and the fact that $\widehat{\alpha(f)}=f$, we obtain from \autoref{basicsupersoft} (iii) that
\[
\alpha(f)+x\leq\alpha(\widehat{\alpha(f)}+\widehat{x})=\alpha(f+\widehat{x}).
\]
Let us prove the opposite inequality. 
Assume first that $\widehat{x}\propto f$ (that is, $\widehat{x}\leq Cf$ for some constant $C>0$).
Let $h\in L(F(\Cu(A)))$ be any function such that  $h\lhd f+\widehat{x}$.
In the case (i), assume also that $h$ is full (see \autoref{ExistenceofSupersoft}), and thus in either case we have $\widehat{\alpha(h)}=h$.
Choose  $\varepsilon>0$ such that $h\leq(1-\varepsilon)f+\widehat{x}$. 
We claim that $h\lhd h+\frac{\varepsilon}{2} f$. 
Indeed, notice first that $h\propto f$, since $\widehat{x}\propto f$. It is then clear that for small enough $\delta>0$ we have $h\leq (1-\delta)(h+\frac{\varepsilon}{2} f)$. Further, if $f(\lambda)<\infty$ then $f(\lambda)+\widehat{x}(\lambda)<\infty$ and therefore $h$ is continuous at $\lambda$. 

Consider the element
\[
y=(\alpha((1-\varepsilon)f)+ x)\wedge \alpha(h).
\]
Then 
\[
\widehat{y}=(((1-\varepsilon)f)+ \widehat{x})\wedge h=h.
\]
Hence, $h\ll \widehat{y}+\frac{\varepsilon}{2} f$ (since $h\lhd h+\frac{\varepsilon}{2} f$). 
Choose $y'\in\Cu(A)$ such that $y'\ll y$ and $h\ll \widehat{y'}+\frac{\varepsilon}{2} f$. 
Then $y'\ll y\leq \alpha(h)$, and thus there exists by (O5) a $z\in\Cu(A)$ such that 
\[
y'+z\leq \alpha(h)\leq y+z.
\]
Observe then that
\[
\widehat{y'}+\widehat{z}\leq h\leq \widehat{y'}+\frac{\varepsilon}{2} f.
\]	
It follows that $\widehat{z}\leq \frac{\varepsilon}{2} f$, and so $z\leq \alpha(\varepsilon f)$. 

Using that $\alpha$ is superadditive (\autoref{prp:mapAlpha} (iv)) at the last step, we obtain
\[
\alpha(h)\leq y+z\leq \alpha((1-\varepsilon)f)+ x+\alpha(\varepsilon f)\leq \alpha(f)+x.
\]
Passing to the supremum over all $h\lhd f+\widehat{x}$, and using that $\alpha$ is supremum preserving (\autoref{prp:mapAlpha} (iii)), we get that $\alpha(f+\widehat{x})\leq \alpha(f)+x$, as desired.

Let us finally deal with the case that $\widehat{x}\leq \infty f$. If $x'\ll x$ then $\widehat{x'}\propto f$. 
Hence $\alpha(f+\widehat{x'})=\alpha(f)+x'$. 
Passing to the supremum over all $x'\ll x$ the theorem follows.
\end{proof}

\begin{cor}
Let $A$ be a separable C*-algebra of stable rank one.
\begin{enumerate}[\rm (i)]
	\item
	If $A$ is unital and has no nonzero, finite dimensional quotients, then $\alpha$ is additive on the set of full elements of $L(F(\Cu(A)))$ and its range is an absorbing subsemigroup (\autoref{dfn:soft}) of $\Cu(A)$.
	\item
	If $A$ has no nonzero, elementary ideal-quotients, then $\alpha$ is additive and its range is an absorbing subsemigroup of $\Cu(A)$.  
\end{enumerate}	
\end{cor}	
\begin{proof}
(i): This is a straightforward consequence of the previous theorem.	

(ii):
We will use the following claims:

\emph{Claim~1: Idempotent elements are supersoft.}
To prove the claim, let $w\in\Cu(A)$ satisfy $2w=w$.
Then $\widehat{w}\leq\widehat{w}=(1-\varepsilon)\widehat{w}$ for every $\varepsilon\in(0,1)$, and thus $w\leq\alpha(\widehat{w})$ by \autoref{prp:mapAlpha}~(ii).
For the converse inequality, note that $\widehat{\alpha(\widehat{\omega})}\leq\widehat{\omega}=\infty\widehat{\omega}$, and thus $\alpha(\widehat{\omega})\leq\infty\omega=\omega$, by \autoref{lma:basic_cones}~(iii).
This proves the claim.

\emph{Claim~2: Let $f,g\in L(F(\Cu(A)))$ with $f\leq\infty g$.
	Then $\alpha(f+g)=\alpha(f)+\alpha(g)$.}
To prove the claim, set $x=\alpha(f)$.
Then $\widehat{x}=f\leq\infty g$, and we may apply \autoref{alphadditive} at the second step to obtain
\[
\alpha(f+g)
= \alpha(\widehat{x}+g)
= x + \alpha(g)
= \alpha(f)+\alpha(g),
\]
which proves the claim.

Let us now show that $\alpha$ is additive. Let $f,g\in L(F(\Cu(A)))$. 
Since $\alpha$ is superadditive by \autoref{prp:mapAlpha}~(iv), it remains to show that $\alpha(f+g)\leq\alpha(f)+\alpha(g)$.
Set $w_f=\alpha(\infty f)$ and $w_g=\alpha(\infty g)$.
Then $w_f+w_g$ is idempotent and thus supersoft by Claim~1.
Hence,
\[
\alpha(\infty f)+\alpha(\infty g)
= w_f+w_g=\alpha(\widehat{w_f}+\widehat{w_g})
= \alpha(\infty f+\infty g).
\]

Using at the first and last step that $\alpha$ preserves infima (\autoref{prp:mapAlpha}~(iii)), and using Claim~2 at the third step, we get
\begin{align*}
\alpha(f+g)\wedge w_f
&= \alpha\big( (f+g)\wedge (\infty f) \big)
= \alpha\big( f+(g\wedge(\infty f)) \big)
= \alpha(f) + \alpha( g\wedge(\infty f) ) \\
&\leq \alpha(f)+\alpha(g).
\end{align*}
Similarly,
\[
\alpha(f+g)\wedge w_g\leq \alpha(f)+\alpha(g).
\]
Using the distributivity of addition over infima (\autoref{general}) at the third step, we obtain that
\begin{align*}
\alpha(f+g) 
&= \alpha(f+g)\wedge(w_f + w_g) \\
&\leq (2\alpha(f+g))\wedge(\alpha(f+g)+w_f)\wedge(\alpha(f+g)+w_g)\wedge(w_f + w_g) \\
&= (\alpha(f+g)\wedge w_f)+(\alpha(f+g)\wedge w_g)
\leq 2(\alpha(f)+\alpha(g)).
\end{align*}

Let $0<\varepsilon<1/2$.
Then, using Claim~2 twice at the first step, the inequality just established at the second step, and the fact that $\alpha$ is superadditive at the last step, we obtain
\begin{align*}
\alpha\big( (1-\varepsilon)(f+g) \big) 
&= \alpha\big( (1-2\varepsilon)f \big) + \alpha\big( (1-2\varepsilon)g \big) + \alpha\big( \varepsilon(f+g) \big) \\
&\leq \alpha\big( (1-2\varepsilon)f \big) + \alpha\big( (1-2\varepsilon)g \big) + 2\alpha(\varepsilon f)+2\alpha(\varepsilon g) \\
&\leq \alpha(f)+\alpha(g).
\end{align*}
Letting $\varepsilon\to 0$ we obtain that $\alpha(f+g)\leq \alpha(f)+\alpha(g)$.
\end{proof}	

Below, we repeatedly use that an element $x\in\Cu(A)$ is full if and only if $\widehat{x}\in L(F(\Cu(A)))$ is full (see \autoref{lma:basic_cones}~(iii)).

\begin{pgr}\emph{Radius of comparison}.
Let $A$ be a unital C*-algebra.
Set $u =[1_A]$ and recall that we use $F_u(\Cu(A))$ to denote the set of all $\lambda\in F(\Cu(A))$
such that $\lambda(u)=1$.
Recall from \cite[Definition~6.1]{Tom06FlatDimGrowth} that the \emph{radius of comparison} of $A$, denoted $\textrm{rc}(A)$, is the infimum of the set of $r\in (0,\infty]$ such that
\[ 
\lambda(x)+r\leq \lambda(y)\hbox{ for all }\lambda\in F_u(\Cu(A))\implies x\leq y
\]
for all $x,y\in \Cu(A)$. 
In the case of unital, stable rank one \ca{s}, there is a more convenient restatement of the definition of $\textrm{rc}(A)$ as the infimum of the set of $r\in (0,\infty]$ such that
\[ 
\widehat{x}+r\widehat{u}\leq \widehat{y}\implies x\leq y
\]
for all $x,y\in \Cu(A)$ (see \cite[Proposition~3.2.3]{BlaRobTikTomWin12AlgRC}).
Observe that in this reformulation, the element $y$ is automatically full, since $r\widehat{u}\leq \widehat{y}$ and $r>0$ (see \autoref{lma:basic_cones}~(iii)).
\end{pgr}

Recall that $W(A)$ denotes the set of Cuntz classes of positive elements in $M_\infty(A)$.

\begin{thm}
\label{thm:equivalences_supersoft}
Let $A$ be a separable, unital C*-algebra of stable rank one that has no nonzero, finite dimensional quotients.
Set $u=[1_A]$.
Then the following are equivalent:
\begin{enumerate}[{\rm (i)}]
	\item
	$W(A)=\{x\in \Cu(A): \widehat{x}\leq n\widehat{u}\hbox{ for some }n\in \NN\}$.
	\item
	$W(A)$ contains at least one full supersoft element.	
	\item
	There exists $N\in \NN$ such that  $\widehat{x}\leq\widehat{u}$
	implies $x\leq Nu$ for all $x\in \Cu(A)$.
	\item
	$A$ has finite radius of comparison.
\end{enumerate}		
\end{thm}
\begin{proof}
(i)$\implies$(ii):
Set $y =\alpha(\widehat{u})$, which is a supersoft element.
Since $\widehat{y}=\widehat{u}$, we have by (i) that $y$ is an element of $W(A)$.
It remains to see that $y$ is full, but this follows from the fact that $\widehat{y}=\widehat{u}$ and $u$ is full in $\Cu(A)$ (see \autoref{lma:basic_cones}~(iii)).

(ii)$\implies$(iii):
Let $z\in W(A)$ be a full supersoft element.
Thus, there exist $m,n\in\NN$ such that $u\leq mz\leq nu$.
Now let $x\in\Cu(A)$ be such that $\widehat{x}\leq \widehat{u}$. 
Then $\widehat{x}\leq m\widehat{z}$.
Using at the third step that $\alpha$ is order-preserving (\autoref{prp:mapAlpha}~(iii)), and using \autoref{alphadditive} at the second and fourth step, we get
\[
x \leq x+\alpha(m\widehat{z})=\alpha(\widehat{x}+m\widehat{z})\leq\alpha(2m\widehat{z})=2mz\leq 2nu.
\]

(iii)$\implies$(iv):
Let $N$ be as in (iii).
To show that $\rc(A)\leq N$, let $x,y\in\Cu(A)$ satisfy $\widehat{x}+N\widehat{u}\leq \widehat{y}$.
Set $z =\alpha(\widehat{u})$.
Applying \autoref{alphadditive}, we obtain
\[
x+Nu\leq x+Nu+z=\alpha(\widehat{x}+N\widehat{u}+\widehat{u})\leq \alpha(\widehat{y}+\widehat{z})=y+z.
\] 
By (iii), we have $z\leq Nu$, and therefore $x+Nu\leq y+Nu$.
Hence, $x\leq y$ by cancellation of compact elements.

(iv)$\implies$(i):
Clearly if $x\in W(A)$ then $\widehat{x}\leq n\widehat{u}$ for some $n\in\NN$. 
Suppose conversely that $x\in\Cu(A)$ and $n\in\NN$ satisfy $\widehat{x}\leq n\widehat{u}$.
Let $N\in\NN$ satisfy $N>\mathrm{rc}(A)$.
From $\widehat{x}+N\widehat{u} \leq (N+n)\widehat{u}$ we deduce that  $x\leq (N+n)u$.
Hence, $x\in W(A)$.  
\end{proof}

\begin{pgr}\emph{Strict comparison and local weak $(m,\gamma)$-comparison}.
\label{pgr:strictComparison}
A C*-algebra $A$ is said to have \emph{strict comparison} if whenever $x,y\in \Cu(A)$ satisfy $x\leq\infty y$ and $\lambda(x)<\lambda(y)$ for every $\lambda\in F(\Cu(A))$ with $\lambda(y)=1$, then $x\leq y$ (see \cite[Proposition~6.2]{EllRobSan11Cone} and \cite[Paragraph~7.6.4]{AntPerThi18TensorProdCu}). 
In general, if $S$ is a \CuSgp, then the following conditions are equivalent:
\begin{enumerate}
\item
$S$ has \emph{strict comparison}, that is, whenever $x,y\in S$ satisfy $x\leq\infty y$ and $\lambda(x)<\lambda(y)$ for every $\lambda\in F(S)$ with $\lambda(y)=1$, then $x\leq y$;
\item
Whenever $x,y\in S$ satisfy $\widehat{x}\leq (1-\varepsilon)\widehat{y}$ for some $\varepsilon>0$, then $x\leq y$;
\item
$S$ is \emph{almost unperforated}, that is, whenever $x,y\in S$ satisfy $(n+1)x\leq ny$ for some $n\in\NN$, then $x\leq y$.
\end{enumerate}
Indeed, (1) easily implies (2). That~(2) implies~(3) follows, for example,  from   \cite[Proposition 5.2.20]{AntPerThi18TensorProdCu}. The equivalence between~(1) and~(3) is proven in \cite[Proposition~6.2]{EllRobSan11Cone} (see also\cite[Corollary 4.7]{Ror04StableRealRankZ}).

Let us say that $A$ has \emph{strict comparison on full elements} if whenever $x,y\in \Cu(A)$, with $y$ full, satisfy $\widehat{x}\leq (1-\varepsilon)\widehat{y}$ for some $\varepsilon>0$, then $x\leq y$. Clearly, if $A$ is a  simple C*-algebra  this property agrees with strict comparison.	

	Following \cite[Definition 2.3]{RobTik17NucDim}, we say that $A$ satisfies  \emph{$m$-comparison} for some $m\in \mathbb{N}$ provided $\widehat{x}\leq (1-\varepsilon) \widehat{y_i}$ for some $x,y_0,y_1,\dots,y_m\in \Cu(A)$  and $\epsilon>0$, implies $x\leq \sum_{i=0}^m y_i$. Note that if $A$ is simple and unital, this  definition coincides with \cite[Definition~ 3.1]{Win12NuclDimZstable}. Observe also that $A$ has $0$-comparison precisely when $A$ has strict comparison.

Suppose now that $A$ is unital.
Set $u=[1_A]$ and recall that $F_{u}(\Cu(A))$ denotes the set of all $\lambda\in F(\Cu(A))$ such that $\lambda(u)=1$.
Suppose that there exist $m\in\NN$ and $\gamma\geq 1$ such that if $a,b\in A_+$, with $b$ full, satisfy
\[
\gamma\cdot \sup_{\lambda\in F_{u}(\Cu(A))}\lambda([a])\leq \inf_{\lambda\in F_{u}(\Cu(A))}\lambda([b])
\]
then $[a]\leq m[b]$.
We then say that $A$ has \emph{local weak $(m,\gamma)$-comparison}.
The word local here refers to the fact that we do not choose $a$ and $b$ in $A\otimes\mathcal{K}$ but in $A$.
The case when $A$ is simple and $m=1$ of this property appears in \cite[Definition~2.1]{KirRor14CentralSeq}, where it is called `local weak comparison'.
We show below that if $A$ is a separable, unital C*-algebra of stable rank one that has no nonzero, finite dimensional quotients, then local weak $(m,\gamma)$-comparison implies strict comparison on full elements.
\end{pgr}

\begin{lma}
\label{lma:first_lemma}
Let $A$ be a separable C*-algebra of stable rank one, let $x\in\Cu(A)$, and let $f\in L(F(\Cu(A)))$ satisfy $f\ll\widehat{x}$.
Then there exist $y,z\in \Cu(A)$ such that $f\leq \widehat{y}$, $y+z\leq x$, and $\infty z=\infty x$.
\end{lma}
\begin{proof}
Choose $w\in\Cu(A)$ satisfying $f\ll\widehat{w}\ll\widehat{x}$.
Set $x_1 =x\wedge w$.
By \autoref{hat-inf-preserving}, we have $\widehat{x_1}=\widehat{x}\wedge\widehat{w}=\widehat{w}$.
Therefore $f\ll\widehat{x_1}$ and $x_1\leq x$.
Choose $y\in\Cu(A)$ such that $y\ll x_1$ and $f\leq \widehat{y}$.
Finally, apply \axiomO{5} to $y\ll x_1\leq x$ to obtain $z\in\Cu(A)$ such that $y+z\leq x\leq x_1+z$.
It remains to show that $z$ satisfies $\infty z\geq \infty x$.

Denote by $W$ the ideal generated by $z$, that is, $W=\{z'\in \Cu(A):z'\leq \infty z\}$.
Using the natural correspondence between closed, two-sided ideals of $A$ and ideals of $\Cu(A)$ (see \autoref{pgr:ideal}), we let $I\subseteq A$ be the closed, two-sided ideal of $A$ such that $\Cu(I)=W$.
Passing to $\Cu(A/I)$ by the quotient map, let us denote the images of $x$, $x_1$, and $y$ by $\underline{x}$, $\underline{x}_1$, and $\underline{y}$.
We have $\underline{y}=\underline{x}_1=\underline{x}$, and this element is compact.
Since $\widehat{x_1}\ll\widehat{x}$, we can choose $\varepsilon>0$ with $\widehat{x_1}\leq (1-\varepsilon)\widehat{x}$.
Passing to $\Cu(A/I)$ we obtain $\widehat{\underline{x}}=(1-\varepsilon)\widehat{\underline{x}}$.
Thus, $\underline{x}$ is a compact element on which no functional is finite and nonzero. 
Since $A/I$ is stably finite, it is well known (for example combining \cite[Theorem~3.2]{GooHan76RankFct} and \cite[Theorem~3.3]{BlaRor92ExtendStates}) that this implies $\underline{x}=0$. 
Hence $x\leq\infty z$.
\end{proof}	

\begin{lma}
\label{lma:second_lemma}
Let $A$ be a unital, separable C*-algebra of stable rank one that has no nonzero finite dimensional quotients.
Set $u=[1_A]$.
Let $(z_i)_{i=1}^\infty$ be a sequence of full, supersoft elements in $\Cu(A)$ such that $\widehat{z_i}\leq\widehat{u}$ for all $i$.
Then
\[
\sum_{i=1}^\infty z_i=\sum_{i=1}^\infty (z_i\wedge u).
\] 	
\end{lma}	
\begin{proof}
Using induction over $n$, let us verify that $\sum_{i=1}^n (z_i\wedge u) = (\sum_{i=1}^n z_i)\wedge (nu)$ for each $n\geq 1$.
This is clear for $n=1$. Assume that the equality holds for some $n\geq 1$. Since $z_{n+1}$ is full and supersoft, and since $n\widehat{u} \geq \sum_{i=1}^n\widehat{z_i}$, we may apply \autoref{alphadditive} at the first and last step to conclude that
\[
z_{n+1} + nu 
= \alpha( \widehat{z_{n+1}}+n\widehat{u} )
\geq \alpha( \sum_{i=1}^{n+1}\widehat{z_i} )
= \sum_{i=1}^{n+1} z_i.
\]
Similarly, we get $(\sum_{i=1}^{n} z_i)+u \geq \sum_{i=1}^{n+1} z_i$.
Then, applying the distributivity of addition over~$\wedge$, we get
\begin{align*}
\sum_{i=1}^{n+1} (z_i\wedge u)
&= \left( \big( \sum_{i=1}^n z_i \big)\wedge nu \right)+(z_{n+1}\wedge u) \\
&= \left( \big( \sum_{i=1}^n z_i \big)+z_{n+1} \right) \wedge \left( \big( \sum_{i=1}^n z_i \big)+u \right) \wedge (z_{n+1} + nu) \wedge ((n+1)u) \\
&\geq \big( \sum_{i=1}^{n+1} z_i \big) \wedge ((n+1)u).
\end{align*}
The converse inequality is clear.

Passing to the supremum over all $n$ we get the desired equality.
\end{proof}	

\begin{thm}
\label{thm:locweak}
Let $A$ be a unital, separable C*-algebra of stable rank one that has no nonzero finite dimensional quotients. 
Then the following conditions are equivalent:
\begin{enumerate}[\rm(i)]
	\item
	$A$ has local weak $(m,\gamma)$-comparison for some $m\in\NN$ and $\gamma\geq 1$.
	\item
	For each full element $x\in\Cu(A)$ there exists a full, supersoft element $z\in\Cu(A)$ such that $z\leq x$.
	\item
	$A$ has strict comparison on full elements.
	\item
	The restriction of $\alpha$ to $\{f\in L(F(\Cu(A))):f\hbox{ is full}\}$ is a $\Cu$-morphism into the subsemigroup
	of full elements of $\Cu(A)$.	
\end{enumerate}		
\end{thm}	
\begin{proof}
(i)$\implies$(ii):
Let $m\in\NN$ and $\gamma\geq 1$ such that $A$ has local weak $(m,\gamma)$-comparison.
Let $x\in\Cu(A)$ be full.
As above, set $u=[1_A]$.
Replacing $x$ by $x\wedge u$ if necessary, we may assume that $x\leq u$.
(Note that $x\wedge u$ remains full by \autoref{wedgefull}.)
Using \autoref{smallsoft}, we can choose a sequence $(x_i)_i$ of full elements in $\Cu(A)$ such that $\sum_{i=1}^\infty mx_i\leq x$. 
Indeed, we first find $x_1$ full such that $(m+1)x_1\leq x$, and then inductively we find $x_i$ full such that $(m+1)x_i\leq x_{i-1}$. Now 
\[
\sum_{i=1}^k mx_i\leq (\sum_{i=1}^{k-1}mx_i)+x_{k-1}\leq (\sum_{i=1}^{k-2}mx_i)+x_{k-2}\leq \dots \leq mx_1+x_1\leq x.\] 
Passing to the supremum over $k$ we get the result.

Since $x_i$ is full for each $i$, there exists $n_i\in\NN$ such that $u\leq n_ix_i$. 
Clearly, we may further assume that $\sum_{i=1}^\infty \tfrac{1}{n_i}\leq 1$.
Set
\[
\varepsilon=\sum_{i=1}^\infty \frac1{\gamma n_i}.
\]
Applying \autoref{smallsoft} again, let $z\in\Cu(A)$ be a full, supersoft element such that $\widehat{z}\leq \varepsilon\widehat{u}$.
We claim that $z\leq x$.
Set $t_i=\tfrac{1}{\varepsilon\gamma n_i}$ for each $i\in\NN$, and observe that $\sum_{i=1}^\infty t_i=1$.
Set $z_i=\alpha(t_i\widehat{z})$ for each $i$.
Since $z$ is full, we see that $z_i$ is full for each $i$. 
Using the additivity of $\alpha$ at the second step (\autoref{alphadditive}), that $\alpha$ is supremum preserving at the third step (\autoref{prp:mapAlpha} (iii)), and that $z$ is supersoft at the last step, we have
\[
\sum_{i=1}^{\infty}z_i=\sup_n\sum_{i=1}^n\alpha(t_i\widehat{z})=\sup_n\alpha(\sum_{i=1}^n t_i\widehat{z})=\alpha(\sup_n\sum_{i=1}^nt_i\widehat{z})=\alpha(\widehat{z})=z.
\]
Since $\widehat{z} \leq \widehat{u}$, \autoref{lma:second_lemma} implies
\[
z = \sum_{i=1}^\infty (z_i\wedge u).
\]
By the way we picked the sequence $(t_i)_i$, we have $\gamma n_i\widehat{z_i}\leq \widehat{u}\leq n_i\widehat{x_i}$. Since clearly $\widehat{z_i\wedge u} \leq \widehat{z_i}$, we also have $\gamma n_i\widehat{z_i\wedge u}\leq \widehat{u} \leq n_i\widehat{x_i}$. 
Since $A$ has local weak $(m,\gamma)$-comparison we conclude that $z_i\wedge u \leq  mx_i$ for all $i$. Therefore,
\[
z=\sum_{i=1}^\infty (z_i\wedge u)
\leq \sum_{i=1}^\infty mx_i\leq x,
\]
as desired.

(ii)$\implies$(iii):
Suppose that $x,y\in\Cu(A)$, with $y$ full, satisfy $\widehat{x}\leq (1-\varepsilon)\widehat{y}$ for some $\varepsilon>0$.
Let $x'\in\Cu(A)$ satisfy $x'\ll x$.
By \autoref{lma:first_lemma} there exist $y',w\in\Cu(A)$ such that $\widehat{x}\leq\widehat{y'}$, $y'+w\leq y$, and $w$ is full.
By assumption, there exists a full, supersoft element $z\in\Cu(A)$ such that $z\leq w$.
Then $\widehat{x}+\widehat{z}\leq\widehat{y'}+\widehat{z}$, whence $\alpha(\widehat{x}+\widehat{z})\leq\alpha(\widehat{y'}+\widehat{z})$.
Now, since $z$ is full we have $\widehat{x},\widehat{y'}\leq\infty\widehat{z}$.
Using \autoref{alphadditive} in the second and fourth step, and that $z$ is supersoft in the fifth step, we obtain
\[
x
\leq x+\alpha(\widehat{z})
=\alpha(\widehat{x}+\widehat{z})
\leq\alpha(\widehat{y'}+\widehat{z})
= y'+\alpha(\widehat{z})
= y'+z
\leq y'+w\leq y.
\]

(iii)$\implies$(iv)
We have already shown that $\alpha$ preserves order and suprema of increasing sequences (\autoref{prp:mapAlpha} (iii)), and that $\alpha$ is additive on full functions (\autoref{alphadditive}). It remains to show that it preserves the way below relation.
Let us show first that if $x,y\in\Cu(A)$ are such that  $\widehat{x}\leq \widehat{y}$, and  $x$ is full and soft, then $x\leq y$ (cf. \cite[Theorem~5.2.18]{AntPerThi18TensorProdCu}). Choose a full element $x'\in\Cu(A)$ such that $x'\ll x$ (see \autoref{ExistenceofSupersoft}). Since $x$ is soft, $\widehat{x'}\ll \widehat{x}\leq \widehat{y}$.
By strict comparison on full elements, $x'\leq y$. Passing to the supremum over all full $x'\ll x$, we get $x\leq y$.
Now let $f,g\in L(F(\Cu(A)))$ be full and such that $f\ll g$. Since $\alpha(g)=\sup_{x\ll \alpha(g)} x$, and $z\mapsto \widehat{z}$ is supremum preserving, we can choose $x\ll \alpha(g)$ such that $f\leq \widehat{x}$.  Since $\alpha(f)$ is soft and full, we deduce that $\alpha(f)\leq x\ll \alpha(g)$, as desired.

(iii)$\implies$(i):
This follows taking $m=1$ and any value $\gamma>1$.

(iv)$\implies$(iii):
Suppose that $x,y\in\Cu(A)$, with $y$ full, satisfy $\widehat{x}\leq (1-\varepsilon)\widehat{y}$ for some $\varepsilon>0$.
Let $x'\ll x$.
Then $\widehat{x'}\ll \widehat{y}$, and thus $\widehat{x'}+\widehat{u}\ll \widehat{y}+\widehat{u}$.
Using \autoref{alphadditive} at the first and last steps, and that $\alpha$ is $\ll$-preserving at the second step, we deduce that
\[
x'+\alpha(\widehat{u})
= \alpha(\widehat{x'}+\widehat{u})
\ll \alpha(\widehat{y}+\widehat{u})
= y+\alpha(\widehat{u}).
\]
By weak cancellation, we get $x'\leq y$.
Passing to the supremum over all $x'\ll x$, we obtain $x\leq y$.
\end{proof}

\begin{thm}
\label{thm:strict}
Let $A$ be a separable C*-algebra of stable rank one that has no nonzero elementary ideal-quotients.
Then the following conditions are equivalent:
\begin{enumerate}[\rm(i)]
    	\item $\Cu(A)$ has $m$-comparison for some $m\geq 0$.
	\item
	There exist  $N\in \NN$ and $\gamma\geq 1$ such that $\gamma \widehat{x}\leq \widehat{y}$ implies $x\leq Ny$ for all $x,y\in \Cu(A)$.
	\item
	For each $x\in \Cu(A)$ there exists $y\leq x$ that is supersoft and such that $\infty y=\infty x$.
	\item
	$\Cu(A)$ has strict comparison.

\end{enumerate}
Moreover, these conditions imply that $\alpha$ is a $\Cu$-morphism.
\end{thm}	
\begin{proof}
(i)$\implies$(ii)  Taking $N=m+1$ and $\gamma>1$, this is the particular case of $y_0=\dots=y_m=y$ in the definition of $m$-comparison. 

(ii)$\implies$(iii):
Let $x\in\Cu(A)$ and let $W=\{y\in \Cu(A) \colon y\leq \infty x\}$ be the ideal generated by $x$.
Using the natural correspondence between closed, two-sided ideals of $A$ and ideals of $\Cu(A)$ (see \autoref{pgr:ideal}), we let $I\subseteq A$ be the closed, two-sided ideal of $A$ such that $\Cu(I)=W$.
Let $M\in\NN$ be such that $M\geq\gamma$.
Since $I$ satisfies the assumptions of \autoref{smallsoft}, we can find $x'\in \Cu(I)\subseteq \Cu(A)$ such that $x'$ is soft, $MNx'\leq x$ and $\infty x'=\infty x$.

Now set $y=\alpha(M\widehat{x'})$. 
Note that $y$ is supersoft by \autoref{basicsupersoft} (iii), since $x'$ (and thus also $Mx'$) is soft.
Then 
\[
\gamma\widehat{y}\leq M\widehat{y}=M\widehat{x'}=\widehat{Mx'}.
\]
Hence,  by assumption, $y\leq MNx'\leq x$. 
Since $\infty\widehat{y}=\infty\widehat{x'}=\infty\widehat{x}$, we also obtain that $\infty y=\infty x'=\infty x$ (by \autoref{lma:basic_cones} (iii)).

(iii)$\implies$(iv):
Let $x,y\in\Cu(A)$ and $\varepsilon>0$ satisfy $\widehat{x}\leq (1-\varepsilon)\widehat{y}$.
Let $x'\ll x$.
Then $\widehat{x'} \ll \widehat{y}$.
By \autoref{lma:first_lemma}, there exist $y',z\in \Cu(A)$ such that $\widehat{x'}\leq \widehat{y'}$, $y'+z\leq y$, and $y\leq \infty z$.
Let $w\in \Cu(A)$ be supersoft, such that $w\leq z$ and $\infty w=\infty z$.

We have by construction that $\widehat{x'}\leq \infty\widehat{w}$. Therefore, using \autoref{alphadditive} at the first step, and that $w$ is supersoft at the second step, we get that
\[
\alpha(\widehat{x'}+\widehat{w})=x'+\alpha(\widehat{w})=x'+w.
\]
Thus $x'+w$ is supersoft. Likewise, $y'+w$ is supersoft. Since $\widehat{x'+w}\leq\widehat{y'+w}$, it follows after applying $\alpha$ on both sides that $x'+w\leq y'+w$. Therefore $x'\leq x'+w\leq y'+w\leq y$.
Passing to the supremum over all $x'\ll x$, we get that $x\leq y$.

(iv)$\implies$(i): 	This follows taking $m=0$.

Lastly, let us show that (iv) implies that $\alpha$ is a $\Cu$-morphism.
As in the proof of \autoref{thm:locweak} (iii)$\implies$(iv), we only need to check preservation of the way below relation.
Let $f,g\in L(F(\Cu(A)))$ satisfy $f\ll g$.
As in the proof of \autoref{thm:locweak} (iii)$\implies$(iv), we obtain $x\in \Cu(A)$ such that $f\leq\widehat{x}$ and $x\ll\alpha(g)$.
By \cite[Theorem~5.2.18]{AntPerThi18TensorProdCu}, if elements $y,z$ in a \CuSgp{} with strict comparison satisfy that $\widehat{y}\leq \widehat{z}$, and if $y$ is soft, then $y\leq z$.
Observe now that $\alpha(f)$ is supersoft, and therefore soft. Indeed, we have from \autoref{thm:realizingprob1} that there exists $w\in \Cu(A)$ such that $\widehat{w}=f$ and $\alpha(f)=w$. Since $\widehat{\alpha(f)}=f\leq\widehat{x}$, we get that $\alpha(f)\leq x$, and so $\alpha(f)\ll\alpha(g)$.
\end{proof}

\section{Nonseparable C*-algebras}
\label{secnonsep}
Here we show that the hypothesis of separability can be dropped in some of the results from the previous sections.
To this end, we rely on the model theory of C*-algebras and in particular on the Downward L\"owenheim-Skolem Theorem for C*-algebras.
For the model theory of C*-algebras we refer the reader to \cite{FarHarLupRobTikVigWin06arX:ModelThy}. 

Given a C*-algebra $A$ and a C*-subalgebra $B$, we write $B\prec A$ if $B$ is an \emph{elementary submodel} of $A$.
This means that for every \emph{formula} $\varphi$ in the language of C*-algebras and every $n$-tuple $\overline{a}$ in $B$, we have $\varphi^B(\overline{a})=\varphi^A(\overline{a})$  (see \cite[Definition~2.3.3]{FarHarLupRobTikVigWin06arX:ModelThy}).
By the Downward L\"owenheim-Skolem Theorem (\cite[Theorem~2.6.2]{FarHarLupRobTikVigWin06arX:ModelThy}), every C*-algebra has a separable elementary submodel. Important to us in what follows is that if $B\prec A$ then the induced map $\Cu(B)\to \Cu(A)$ is an order-embedding (\cite[Lemma~8.1.3]{FarHarLupRobTikVigWin06arX:ModelThy}). Recall that an order-preserving map $\varphi\colon S\to T$ between partially ordered sets is an \emph{order-embedding} provided that  $\varphi(x)\leq \varphi(y)$ implies $x\leq y$, for $x,y\in S$.

The next result removes the separability assumption in \autoref{CstarGG}.

\begin{thm}
\label{nonsepGG}
Let $A$ be a unital C*-algebra of stable rank one, and let $k\in\NN$.
Then $A$ has no nonzero representations of dimension less than $k$ if and only if there exists a *-homomorphism $\varphi\colon M_k(C_0((0,1]))\to A$ with full range.
\end{thm}
\begin{proof}
The proof of the easy direction in \autoref{CstarGG} does not make use of the separability hypothesis.
Hence, it applies here. 

Suppose that $A$ is a unital C*-algebra of stable rank one without nonzero representations of dimension less than $k$.
By \cite[Corollary~5.4]{RobRor13Divisibility}, the element $[1]$ is weakly $(k,n)$-divisible for some $n$ (see 
\autoref{pgrweakdiv}).
Thus, there exist $a_1,\ldots,a_n\in A_+$ such that $k[a_i]\leq [1]$ for all $i$ and $[1]\leq \sum_{i=1}^n [a_i]$.  Apply the Downward L\"owenheim-Skolem Theorem to obtain a separable  C*-subalgebra $B\prec A$ that  contains $1,a_1,\ldots,a_n$.
Since the inclusion of $B$ in $A$ induces an order-embedding of $\Cu(B)$ in $\Cu(A)$ (\cite[Lemma~8.1.3]{FarHarLupRobTikVigWin06arX:ModelThy}), the inequalities 
$k[a_i]\leq [1]$ for all $i$ and $[1]\leq \sum_{i=1}^n [a_i]$ also hold in $\Cu(B)$.
By \cite[Corollary~5.4]{RobRor13Divisibility}, $B$ has no representations of dimension less than $k$.
On the other hand, by \cite[Lemma~3.8.2]{FarHarLupRobTikVigWin06arX:ModelThy}), the property of having stable rank one is elementary and therefore passes to elementary submodels.	
We can thus apply \autoref{CstarGG} in $B$ to obtain a *-homomorphism $\varphi\colon M_k(C_0((0,1]))\to B\subseteq A$ whose range is full in $B$. 
Since $1\in B$, the range of $\varphi$ is also full in $A$.
\end{proof}

Next we extend \autoref{thm:realizingprob2} to the nonseparable case.
We start with a preparatory result.

\begin{lma}
\label{lambdaAB}
Let $A$ be a C*-algebra, let $B\prec A$, and let $a,b\in B_+$.
Then $\lambda([a])\leq \lambda([b])$ for all $\lambda\in F(\Cu(A))$ if and only if $\lambda([a])\leq \lambda([b])$ for all $\lambda\in F(\Cu(B))$.	
\end{lma}		
\begin{proof}
The backward implication follows directly using that every functional on $\Cu(A)$ restricts to a functional on $\Cu(B)$.
To show the converse, suppose that $\lambda([a])\leq \lambda([b])$ for all $\lambda\in F(\Cu(A))$.
Let $\varepsilon>0$ and $\delta>0$.
By \cite[Proposition~2.2.6]{Rob13Cone} there exist $M,N\in\NN$ such that $\tfrac{M}{N}\geq 1-\delta$ and $M[(a-\varepsilon)_+]\leq N[b]$ in $\Cu(A)$.
Since the inclusion $B\to A$ induces an order-embedding $\Cu(B)\to\Cu(A)$, this inequality also holds in $\Cu(B)$.
Fix $\lambda\in F(\Cu(B))$.
Evaluating both sides of $M[(a-\varepsilon)_+]\leq N[b]$ on $\lambda$ we get  
\[
(1-\delta)\lambda\big( [(a-\varepsilon)_+] \big) \leq \lambda([b]).
\] 
Since this holds for all $\delta,\varepsilon>0$ we conclude that $\lambda([a])\leq \lambda([b])$, as desired.	
\end{proof}

\begin{thm}
\label{prob2nonsep}
Let $A$ be a unital C*-algebra of stable rank one with no finite dimensional quotients.
Set $u=[1_A]$ and recall that $F_u(\Cu(A))\subseteq F(\Cu(A))$ denotes the set of functionals normalized at $u$.
Then for each $f\in \LAff(F_u(\Cu(A)))_{++}^\sigma$ there exists $z\in \Cu(A)$ such that $\widehat{z}|_{F_u(\Cu(A))}=f$.
\end{thm}	
\begin{proof}
Let us regard $A$ embedded in $A\otimes\mathcal{K}$ as the `upper left corner'.
Let $1_A\in A\otimes \mathcal{K}$ denote the unit of $A$.
Given $f\in \LAff(F_u(\Cu(A)))_{++}^\sigma$, apply \autoref{surjectiveres} to obtain $\tilde{f}\in L(F(\Cu(A)))$ that extends $f$. 
As noted in \autoref{lma:basic_cones} (i), we have $L(F(\Cu(A)))=\Cu(A)_R$, which allows us to choose a sequence $(x_i)_i$ in $\Cu(A)$ and a sequence of positive integers $(m_i)_i$ such that $(\tfrac{\widehat{x_i}}{m_i})_i$ is increasing and $\sup_i \tfrac{\widehat{x_i}}{m_i}=\tilde{f}$.
Choose $a_i\in (A\otimes\mathcal{K})_+$ such that $x_i=[a_i]$ for all $i$.  

Since $A$ has no finite dimensional representations, by \cite[Corollary~5.4]{RobRor13Divisibility} there exists for each $k$ an $n_k\in\NN$ such that $[1_A]$ is weakly $(k,n_k)$-divisible in $\Cu(A)$.
We thus find $b_{k,l}\in A_+$ for $k=1,2,\ldots$ and $l=1,\ldots,n_k$ such that $k[b_{k,l}]\leq [1_A]$ for all $k,l$ and $[1_A]\leq \sum_{l=1}^{n_k} [b_{k,l}]$ for all $k$.
Apply the Downward L\"owenheim-Skolem theorem to obtain a separable elementary submodel $B\prec A\otimes \mathcal K$ that contains all $a_i$, all $b_{k,l}$, and $1_A$. 

As argued in the proof of \autoref{nonsepGG}, $B$ has stable rank one.
Further, the inclusion of $B$ in $A\otimes\mathcal K$ induces a natural order-embedding $\Cu(B)\to\Cu(A)$.

We claim that $B$ is stable.
To prove this we use the Hjelmborg-R{\o}rdam criterion for stability established in \cite{HjeRor98Stability}, see also \cite[Proposition~2.7.7]{FarHarLupRobTikVigWin06arX:ModelThy}. 
By the stability of $A\otimes\mathcal K$, for each $b\in B_+$ we have
\[
\inf_{v\in A\otimes\mathcal K} \big( \|b-v^*v\| + \|bvv^*\| \big) =0.
\]
Since $B\prec A\otimes \mathcal K$, the displayed formula also evaluates to 0 in $B$. That is,
for every $\varepsilon>0$ there exists $w\in B$ such that $\|b-w^*w\|<\varepsilon$ and $\|bww^*\|<\varepsilon$.
Since $B$ is separable, \cite[Theorem 2.1 and Proposition 2.2]{HjeRor98Stability} implies that $B$ is stable.

Let us show that $1_A\in B$ is full in $B$.
For every $b\in B_+$ we have $[b]\leq \infty [1_A]$ in $\Cu(A)$, as $1_A$ is full in $A\otimes \mathcal K$ (see \autoref{pgr:full}).
Using that $\Cu(B)\to\Cu(A)$ is an order-embedding, we get $[b]\leq \infty [1_A]$ in $\Cu(B)$, which implies that $1_A$ is full in $B$. 

The inequalities $k[b_{k,l}]\leq [1_A]$ and $[1_A]\leq \sum_{l=1}^{n_k} [b_{k,l}]$ hold in $\Cu(B)$ for all $k,l$, using again that $\Cu(B)\to\Cu(A)$ is an order-embedding.
Therefore, the element $[1_A]$ is weakly $(k,n_k)$-divisible in $\Cu(B)$ for all $k$.
By \cite[Corollary~5.4]{RobRor13Divisibility}, the hereditary C*-subalgebra $1_AB1_A$ has no finite dimensional representations. 

By \autoref{lambdaAB}, the sequence $(\tfrac{\widehat{x_i}}{m_i})_i$ considered in the first paragraph of the proof is increasing when regarded as a sequence in $L(F(\Cu(B)))$. Let $h\in L(F(\Cu(B)))$ be its supremum.
The function $h$ is full, since $\tfrac{\widehat{x_i}}{m_i}$ is full for large enough $i$.
By \autoref{thm:realizingprob2} applied to the C*-algebra $B$, we have $h=\widehat{x}$ for $x=\alpha(h)$. 
Since $B$ is stable, there exists $c\in B_+$ such that $[c]=x$, and thus $\widehat{[c]}=h$. 

We claim that $[c]$, regarded as an element in $\Cu(A)$, satisfies $\widehat{[c]}=\tilde f$.
By \autoref{lambdaAB}, the inequalities $\tfrac{\widehat{x_i}}{m_i}\leq \widehat{[c]}$, which hold in $L(F(\Cu(B)))$, also hold in $L(F(\Cu(A)))$ for all $i$.
Passing to the supremum over $i$, we get that $\tilde f\leq \widehat{[c]}$.
Let $[c']\in \Cu(B)$ be such that $[c']\ll [c]$.
By the definition of $\alpha(h)$, we have that $\widehat{[c']}\ll h=\widehat {[c]}$ in $L(F(\Cu(B)))$.
Hence $\widehat{[c']}\leq\tfrac{\widehat{x_i}}{m_i}$ for some $i$.
By \autoref{lambdaAB}, this inequality holds also in $L(F(\Cu(A)))$.
Hence, $\widehat{[c']}\leq f$. This holds for $c'=(c-\varepsilon)_+$ and arbitrary $\varepsilon>0$.
Hence, $\widehat{[c]}\leq \tilde f$ in $L(F(\Cu(A)))$, as desired.
\end{proof}

\providecommand{\etalchar}[1]{$^{#1}$}
\providecommand{\href}[2]{#2}

\end{document}